\documentclass[]{article}

\usepackage{graphicx}
\usepackage{amsmath}
\usepackage{amssymb}
\usepackage{amsthm}
\usepackage{mathrsfs}
\usepackage{pxfonts}
\usepackage{enumerate}
\usepackage{color}
\usepackage{mathdots}
\usepackage{sectsty}
\usepackage{tikz}
\usepackage{adjustbox}
\usepackage{enumitem}
\usepackage{caption}
\usepackage{bbold}
\usepackage[hidelinks]{hyperref}
\allowdisplaybreaks
\sectionfont{\scshape\centering\fontsize{11}{14}\selectfont}
\subsectionfont{\scshape\fontsize{11}{14}\selectfont}
\usepackage{fancyhdr}
\usepackage[nottoc,notlot,notlof]{tocbibind}

\newcommand\shorttitle{ISDE with logarithmic interaction and characteristic polynomials}
\newcommand\authors{T. Assiotis and Z. S. Mirsajjadi}

\newtheorem{thm}{Theorem}[section]
\newtheorem{cor}[thm]{Corollary}
\newtheorem{lem}[thm]{Lemma}
\newtheorem{defn}[thm]{Definition}
\newtheorem{rmk}[thm]{Remark}
\newtheorem{prop}[thm]{Proposition}

\newtheorem*{theorem*}{Theorem}

\title{\large \bf ISDE WITH LOGARITHMIC INTERACTION AND CHARACTERISTIC POLYNOMIALS}
\author{\small THEODOROS ASSIOTIS AND ZAHRA SADAT MIRSAJJADI}
\date{}

\begin{document}

\maketitle

\begin{abstract} We consider certain random matrix eigenvalue dynamics, akin to Dyson Brownian motion, introduced by Rider and Valko \cite{Rider-Valko}. This is a multi-dimensional generalisation of a one-dimensional diffusion studied by Matsumoto and Yor \cite{MatsumotoYor2}. We show that from every initial condition, including ones involving coinciding coordinates, the dynamics, enhanced with more information, converge on path-space to a new infinite-dimensional Feller-continuous diffusion process. We show that the limiting diffusion solves an infinite-dimensional system of stochastic differential equations (ISDE) with logarithmic interaction. Moreover, we show convergence in the long-time limit of the infinite-dimensional dynamics starting from any initial condition to the equilibrium measure, given by the inverse points of the Bessel determinantal point process. As far as we can tell, this is: (a) the first path-space convergence result of random matrix dynamics starting from every initial condition to an infinite-dimensional Feller diffusion, (b) the first construction of solutions to an ISDE with logarithmic interaction from every initial condition for which the singular drift term can be defined at time $0$, (c) the first convergence to equilibrium result from every initial condition for an ISDE of this kind. The argument splits into two parts. The first part builds on the method of intertwiners introduced and developed by Borodin and Olshanski \cite{MarkovProc-pathSpace-GT}. The main new ingredients are a uniform, in a certain sense, approximation theorem of the spectrum of a family of random matrices indexed by an infinite-dimensional space and an extension of the method of intertwiners to deal with convergence to equilibrium. The second part introduces a new approach towards convergence of the singular drift term in the dynamics and for showing non-intersection of the limiting paths via certain ``characteristic polynomials" associated to the process. We believe variations of it will be applicable to other infinite-dimensional dynamics coming from random matrices.
\end{abstract}

\tableofcontents

\section{Introduction}

\subsection{Background}

The rigorous study of the non-equilibrium statistical mechanics of an infinite system of particles interacting via a potential was initiated in the seminal works \cite{LanfordI,LanfordII,LanfordNotes} of Lanford. The study of the stochastic version of this problem, when one adds independent Brownian forces, and the dynamics are now described by an infinite-dimensional system of stochastic differential equations (henceforth abbreviated ISDE) was then first taken up by 
Lang \cite{Lang1,LangII} who proved well-posedness of the equilibrium dynamics for certain nice potentials. Many authors then contributed to this program, including Doss, Royer, Rost, Lippner and Fritz, and extended such results to non-equilibrium dynamics and more general potentials, see \cite{Rost,Lippner,DossRoyer,Fritz1}. This program essentially culminates with the fundamental paper of Fritz \cite{Fritz2}, which at least when the dimension of individual particles is one, essentially provides a complete solution theory for potentials which are twice continuously differentiable everywhere and with compact support.

The solution theory of \cite{Fritz2} however excludes singular potentials and the construction of solutions is stated therein as an important open problem. In particular it excludes, as both assumptions fail, arguably the most famous potential in $1$-dimension, namely the logarithm. Formally, such a system of interacting one-dimensional particles is governed by the following singular stochastic differential equations\footnote{The model formally makes sense for an arbitrary positive constant $c$ in front of the singular drift. The choice of $c=2$ is somewhat distinguished as it (again completely formally) corresponds to a Doob $h$-transform (or Doob conditioning) \cite{Doob,Revuz-Yor} of infinitely many independent one-dimensional diffusions solving the equation $\mathrm{d}\mathsf{x}(t)=\sqrt{2\mathfrak{a}(\mathsf{x}(t))}\mathrm{d}\mathsf{w}(t)+\mathfrak{b}(\mathsf{x}(t))\mathrm{d}t$, by the infinite-dimensional version of the Vandermonde determinant $\prod_{i<j}|x_i-x_j|$. For finitely many particles, and special choices of  the functions $\mathfrak{a}(\cdot),\mathfrak{b}(\cdot)$, this conditioning has well-defined and very useful probabilistic meaning, see \cite{Grabiner,KonigOConnell,KonigOConnellRoch,CorwinHammond,BesselLineEnsemble}.}:
\begin{equation}\label{LogInteraction}
\mathrm{d}\mathsf{x}_i(t)=\sqrt{2\mathfrak{a}(\mathsf{x}_i(t))}\mathrm{d}\mathsf{w}_i(t)+\mathfrak{b}(\mathsf{x}_i(t))\mathrm{d}t+``2\mathfrak{a}(\mathsf{x}_i(t))\sum_{j\neq i}\partial_{\mathsf{x}_i(t)}\log \left|\mathsf{x}_i(t)-\mathsf{x}_j(t)\right|\mathrm{d}t",
\end{equation}
for some nice diffusion $\mathfrak{a}$ and drift $\mathfrak{b}$ functions and with the $\mathsf{w}_i$ being independent standard Brownian motions. The singular force of interaction experienced by the $i$-th particle $\mathsf{x}_i$ which involves all other particles $(\mathsf{x}_j)_{j\neq i}$ is in quotes because it may need to be renormalised in some way in order to make sense.

An important motivation for studying systems of particles interacting via a logarithmic potential comes from the theory of random matrices \cite{AGZ,ForresterBook} with the most famous example being the following. If $\left(\mathbf{W}_t;t\ge 0\right)$ is the Brownian motion on $N\times N$ Hermitian matrices then its eigenvalues satisfy the closed system of stochastic differential equations (SDE), called Dyson Brownian motion \cite{Dyson}: 
\begin{align}\label{DysonIntro}
\mathrm{d}\mathsf{x}_i(t)=\mathrm{d}\mathsf{w}_i(t)+\sum_{j=1,j\neq i}^N\frac{1}{\mathsf{x}_i(t)-\mathsf{x}_j(t)}\mathrm{d}t, \ \ i=1,\dots, N.
\end{align}
This SDE has a unique strong solution with almost surely no collisions between particles for all $t>0$, even if started from initial conditions with coinciding coordinates, see \cite{AGZ,Graczyk-Malecki}. Dyson Brownian motion, and its variants, has been the object of study for decades. Beyond its intrinsic interest from the perspective of stochastic analysis, probability and integrable systems, it has been a key tool in proving universality for random matrices, see \cite{ErdosYau}.

Coming back to our discussion of infinite systems, as far as we are aware, the only existence (and uniqueness) result of solutions to an ISDE with logarithmic interaction from concrete initial conditions is 
 in the tour-de-force work of Tsai \cite{Tsai}. This corresponds to the bulk limit (see the seminal work of Spohn \cite{SpohnDyson}) of Dyson Brownian motion \eqref{DysonIntro} and the ISDE is given by (as noted in the above footnote it is possible to include a parameter $\beta\ge 1$ in the drift which corresponds to so-called $\beta$-ensembles \cite{ForresterBook}): 
\begin{align}\label{DysonISDE}
\mathrm{d}\mathsf{x}_i(t)=\mathrm{d}\mathsf{w}_i(t)+\frac{\beta}{2} \lim_{k \to \infty }\sum_{j\neq i:|j-i|\le k}\frac{1}{\mathsf{x}_i(t)-\mathsf{x}_j(t)}\mathrm{d}t, \ \ i\in \mathbb{Z}.
\end{align}
The initial conditions allowed need to satisfy a certain quantitative ``balanced condition": particles' positions are approximately uniform in some averaged sense with lower order corrections, see \cite{Tsai}. Moreover, the solution constructed in \cite{Tsai} satisfies this property almost surely for all times. The same property is also almost surely satisfied by the $\beta$-sine point process \cite{ForresterBook,BrownianCarousel,ValkoVirag} which is the invariant\footnote{We note that the labelled ISDE (\ref{DysonISDE}) itself does not have an invariant measure, see the discussion around equation (1.7) in \cite{OsadaErgodicity}. The simplest possible instance of this phenomenon is that of infinitely many independent Brownian motions $\left(\mathsf{w}_i\right)_{i\in \mathbb{Z}}$. This process does not have an invariant measure on $\mathbb{R}^\mathbb{Z}$ but when viewed as a process on unlabelled configurations the Poisson point process is invariant for it.} \footnote{Strictly speaking this invariance does not follow directly from the techniques of \cite{Tsai} and it is only proven for $\beta=1,2,4$ by showing the solutions in \cite{Tsai} coincide with the ``near-equilibrium" solutions of \cite{OsadaTanemura}.} measure of the dynamics if we look at (\ref{DysonISDE}) as an unlabelled point process. 

It is also important to mention a remarkable, almost thirty year-long program of Osada, later in collaboration with Tanemura on ISDE, see \cite{OsadaCMP,OsadaAOP,OsadaPTRF,OsadaTanemura}. This program develops a general solution theory and can treat a number of examples, also of other singular interactions and in higher dimensions, but it only yields ``near-equilibrium" solutions. These are solutions starting from an abstract, and in particular non-explicit, set of allowed configurations in which the solution stays for all times. This set is of full measure, when viewed as a set of unlabelled configurations, with respect to the invariant\footnote{This probability measure on unlabelled configurations is given as data for the problem. One should think of it as playing the role the sine point process plays for the Dyson bulk ISDE \eqref{DysonISDE}.} probability measure\footnote{For a certain parameter range in our model this measure is infinite and this is another novel feature of the present work.} on point processes. The fact that one deals with an abstract set of initial conditions is not a technical restriction but rather an intrinsic feature of the approach which is based on the theory of Dirichlet forms \cite{MaRockner}. For other applications of Dirichlet form theory on such particle systems, from a geometric viewpoint, see \cite{AKRconfSpace,suzuki2022curvature}.

Although, as far as we can tell there are no other results on ISDE with logarithmic interaction, there is a great number of results on scaling limits of random matrix dynamics. Possibly the most famous being the edge scaling limit of Dyson Brownian motion to the Airy line ensemble  $(\mathcal{ALE}_i(\cdot))_{i=1}^\infty$, see \cite{CorwinHammond}, a central object in the KPZ universality class \cite{CorwinKPZ}. To do this, a probabilistic technique, that employs a certain Gibbs resampling property of the paths was developed in \cite{CorwinHammond}. This led to an enormous amount of activity and has been extremely successful in showing that various models belong to the KPZ universality class \cite{CorwinHammond,DauvergneVirag,XuanWu,aggarwal2024scaling}. Although our goal in this paper is rather orthogonal to these works there are some interesting connections with a new Gibbs resampling property, see Section \ref{SectionGibbs}. Finally, the way limits of random matrix dynamics, and related discrete models, were first rigorously studied was through the scaling limit of their space-time correlation functions, see \cite{ForresterDynamical1,ForresterDynamical2,ForresterBook,PrahoferSpohn,Johansson1,Johansson2,KatoriTanemura1,KatoriTanemura2}. This approach provides delicate quantitative information about the models and explicit formulae for their correlations, but as far as we can tell, cannot be used to establish the type of results we present below. It is interesting to note that, in contrast to other limits of Hermitian matrix stochastic dynamics, the first construction of the limiting object below is directly as an infinite-dimensional diffusion process and described via SDE while the explicit computation of its correlation functions is still an open problem. We will survey relevant previous works and how they compare with ours in Section \ref{SectionHistory}.

The main purpose of this paper is to analyse an ISDE with logarithmic interaction coming from random matrix dynamics for which we can go beyond what was known previously (as far as we can tell each of our main results below is new for any random matrix model and does not follow by other methods) and to introduce some new ideas for studying the singular drift term and non-intersection of the paths via certain ``characteristic polynomials" of the process. We believe that these ideas will be useful in proving analogous results for a class of infinite dimensional dynamics coming from random matrices. The main reason we have chosen to start with this specific model is that it enjoys certain integrable properties which allow us to build on a powerful method of Borodin and Olshanski, see \cite{MarkovProc-pathSpace-GT,MarkovDynam-ThomaCone,OlshanskiApproximation,OlshanskiGenerator,OlshanskiLectureNotes,OlshanskiICM}, originally developed in the discrete setting of dynamics on partitions. It is plausible that this part of the argument may be replaced by more robust techniques in the future. We have not attempted to do this, partly because the intermediate results we need are of independent interest beyond the setting of ISDE.

\subsection{Main results}

The model we consider is the eigenvalue evolution of a diffusion on the space of $N\times N$ non-negative definite matrices introduced by Rider and Valko in \cite{Rider-Valko}, that we recall in Section \ref{SubsectionWellposedness}. They used this to prove a matrix analogue \cite{Rider-Valko} of Dufresne's identity \cite{Dufresne}, partly motivated by problems in the theory of stochastic operators related to random matrices, see \cite{RiderRamirez,RiderSpiking,RiderRamirezVirag,ValkoVirag}. The $N=1$ case of this process is a celebrated one-dimensional diffusion that comes up in Matsumoto-Yor's study of exponential functionals of Brownian motion \cite{MatsumotoYor1,MatsumotoYor2} and is also connected to random polymers \cite{OConnellToda}. It is much-studied in the applied, statistical and financial mathematics literature \cite{AppliedInhomogeneousGeomBm,Pearson,finance} and sometimes goes by the name inhomogeneous geometric Brownian motion.

The dynamics we consider are given by the following system of stochastic equations, with parameter $\eta \in \mathbb{R}$,
\begin{align}\label{FSDE}
\mathrm{d}\mathfrak{x}_i(t) = \mathfrak{x}_i(t) \mathrm{d}\mathsf{w}_i(t) -\frac{\eta}{2}\mathfrak{x}_i(t)\mathrm{d}t+\frac{1}{2}\mathrm{d}t+\sum_{j=1,j\neq i}^N\frac{\mathfrak{x}_i(t)\mathfrak{x}_j(t)}{\mathfrak{x}_i(t)-\mathfrak{x}_j(t)}\mathrm{d}t, \ \ i=1,2,\dots,N,
\end{align}
with the $\mathsf{w}_i$ being independent standard Brownian motions. By general results \cite{Graczyk-Malecki} on such systems in finite dimensions we obtain that \eqref{FSDE} has a unique strong non-exploding solution starting from any initial condition in the chamber:
\begin{align}\label{PosWeylChamber}
 \mathbb{W}_{N,+} = \left\{\mathbf{x}=(x_1,x_2,\dots,x_N)\in\mathbb{R}^{N}: x_1\geq x_2 \ge \dots\geq x_N\ge 0 \right\},
\end{align}
and almost surely, for all positive times lives in the interior of $\mathbb{W}_{N,+}$, see Lemma \ref{lem-wellPosednessN}. It is instructive to see that the SDE \eqref{FSDE}, by rewriting the interaction term, is given by 
\begin{equation}\label{LogSDE}
\mathrm{d}\mathfrak{x}_i(t) = \mathfrak{x}_i(t) \mathrm{d}\mathsf{w}_i(t) -\frac{\eta}{2}\mathfrak{x}_i(t)\mathrm{d}t+\frac{1}{2}\mathrm{d}t+\left[(-N+1) \mathfrak{x}_i(t)+\mathfrak{x}_i(t)^2\sum_{j=1,j\neq i}^N \partial_{\mathfrak{x}_i(t)}\log\left|\mathfrak{x}_i(t)-\mathfrak{x}_j(t)\right|\right]\mathrm{d}t, 
\end{equation}
which elucidates the fact that the stochastic dynamics \eqref{FSDE} correspond to independent one-dimensional diffusions interacting via a logarithmic potential.

In order to state our results precisely we need some notation and definitions. Define the following spaces, endowed with the topology of coordinate-wise convergence,
\begin{align}
\mathbb{W}_{\infty,+}&=\left\{\mathbf{x}=(x_i)_{i=1}^\infty \in \mathbb{R}^{\mathbb{N}}:x_1 \ge x_2 \ge x_3 \ge \cdots \ge 0 \right\},\\
\mathbb{W}_{\infty,+}^\circ&=\left\{\mathbf{x}=(x_i)_{i=1}^\infty \in \mathbb{R}^{\mathbb{N}}:x_1 > x_2 > x_3 > \cdots > 0 \right\}.
\end{align}

In fact, our limiting dynamics will live on an enhanced space containing somewhat more information than the above. It is on this enhanced space, and this is essential, that the dynamics will enjoy a Feller-continuity property. The space is defined as follows.

\begin{defn}
We define the space
\begin{align}\label{Omega+}
\Omega_+\overset{\textnormal{def}}{=}
		\left\{\omega=(\mathbf{x},\gamma)\in \mathbb{W}_{\infty,+}\times \mathbb{R}_+: \sum_{i=1}^{\infty}x_{i}\leq\gamma\right\},
	\end{align}
endowed with the topology of coordinate-wise convergence.   
\end{defn}
We observe that, $\Omega_+$ is locally compact, metrizable and separable and the topology can be metrised with the metric $\mathsf{d}_{\Omega_+}$ as follows, with $\omega=\left(\mathbf{x},\gamma\right),\tilde{\omega}=\left(\tilde{\mathbf{x}},\tilde{\gamma}\right) \in \Omega_+$,
\begin{equation*}
\mathsf{d}_{\Omega_+}(\omega,\tilde{\omega})=\sum_{i=1}^\infty \frac{\left|x_i-\tilde{x}_i\right|}{2^i\left(1+\left|x_i-\tilde{x}_i\right|\right)} + \left| \gamma-\tilde{\gamma} \right|.
\end{equation*}
Moreover, note that $\Omega_+$ is complete and thus it is a Polish space.

Returning to the dynamics, heuristics related to the hard-edge scaling in random matrix theory \cite{ForresterHardEdge,ForresterBook} indicate that we should rescale our stochastic processes in space by $1/N$ to see non-trivial behaviour. In particular, the rescaled SDE (\ref{FSDE}) becomes:
\begin{align}\label{rescaledSDE}
		\mathrm{d}\mathsf{x}_i^{(N)}(t) = \mathsf{x}_i^{(N)}(t) \mathrm{d}\mathsf{w}_i(t) -\frac{\eta}{2}\mathsf{x}_i^{(N)}(t)\mathrm{d}t+\frac{1}{2N}\mathrm{d}t+\sum_{j=1,j\neq i}^N\frac{\mathsf{x}_i^{(N)}(t)\mathsf{x}^{(N)}_j(t)}{\mathsf{x}^{(N)}_i(t)-\mathsf{x}^{(N)}_j(t)}\mathrm{d}t, \ \  i=1,2,\dots,N.
\end{align}

We then embed, for different $N \in \mathbb{N}$, all these processes on the space $\Omega_+$.

\begin{defn}
Define the process $\left(\mathbf{X}^{(N)}(t);t\ge 0\right)$ on $\Omega_+$ by, with $\left(\mathsf{x}^{(N)}(t);t\ge 0\right)$ the solution of \eqref{rescaledSDE}, 
\begin{equation*}
\mathbf{X}^{(N)}(t)=\left(\left(\mathsf{x}_i^{(N)}(t)\right)_{i=1}^\infty,\sum_{i=1}^\infty \mathsf{x}_i^{(N)}(t)\right), \ \ \forall t \ge 0,
\end{equation*}
where $\mathsf{x}_i^{(N)}\equiv 0$, for $i>N$.
\end{defn}

Finally, we recall the definition of a Feller semigroup and Feller process in Definition \ref{DefnFellerSemigroup}, see \cite{Kallenberg}. For a Polish space $\mathfrak{X}$ write $C(\mathbb{R}_+,\mathfrak{X})$ for the space of continuous functions on $\mathbb{R}_+$ with values in $\mathfrak{X}$, endowed with the topology of locally uniform convergence. Note that, this is again a Polish space, see \cite{Kallenberg}. We can now state our first main result.

\begin{thm}\label{MainThm1Intro} Let $\eta \in \mathbb{R}$. Then, there exists a unique Feller semigroup $\left(\mathfrak{P}_\infty(t)\right)_{t\ge 0}$ with associated Feller-Markov process $\left(\mathbf{X}_t^{\Omega_+};t\ge 0\right)=\left(\left(\left(\mathsf{x}_i(t)\right)_{i=1}^\infty,\boldsymbol{\gamma}(t)\right);t\ge 0\right)$ on $\Omega_+$ satisfying the following. Let $\mathbf{X}_0^{\Omega_+}=\omega \in \Omega_+$ be arbitrary and assume $\mathbf{X}^{(N)}(0)\to \omega$ in the topology of $\Omega_+$. Then, as $N \to \infty$,
\begin{equation}\label{PathSpaceConvIntro}
 \mathbf{X}^{(N)} \overset{\textnormal{d}}{\longrightarrow} \mathbf{X}^{\Omega_+}\textnormal{ in } C(\mathbb{R}_+,\Omega_+),
\end{equation}
with $\overset{\textnormal{d}}{\longrightarrow}$ denoting convergence in distribution. In particular, $\left(\mathbf{X}_t^{\Omega_+};t \ge 0\right)$ is a diffusion\footnote{A strong Markov process with continuous sample paths.}. Moreover, there exists a coupling of the $\mathbf{X}^{(N)}$ and $\mathbf{X}^{\Omega_+}$ on a single probability space such that, almost surely, for any $T \ge 0$,
\begin{equation}\label{l^2Intro}
  \sup_{t \in [0,T]}\sum_{i=1}^\infty \left(\mathsf{x}_i^{(N)}(t)-\mathsf{x}_{i}(t)\right)^2 \overset{N \to \infty }{\longrightarrow}  0.
\end{equation}

\end{thm}

As far as we can tell, this is the first result on path-space convergence of random matrix dynamics from every single initial condition for which the finite-dimensional dynamics are well-defined. Observe that, the limiting diffusion $(\mathbf{X}_t^{\Omega_+};t \ge 0)$ can start from configurations with coinciding points. In fact, it can even start from the very singular configurations $((0)_{i=1}^\infty,\gamma)$ of having all particles at $0$!

It is a natural question whether the limiting process solves an ISDE with logarithmic interaction. Our second main result answers this in the affirmative and in fact gives a lot more information. Throughout this paper we use the notation $\mathsf{Law}(\mathbf{Y})$ for the law of a random variable $\mathbf{Y}$ taking values in a Polish space.

\begin{thm}\label{MainThm2Intro} Let $\eta \in \mathbb{R}$. Let $\mathbf{x} \in \mathbb{W}_{\infty,+}^\circ$ and $\gamma \ge \sum_{i=1}^\infty x_i$. Consider the Feller process on $\Omega_+$, $\left(\mathbf{X}_t^{\Omega_+};t\ge 0\right)=\left(\left(\left(\mathsf{x}_i(t)\right)_{i=1}^\infty,\boldsymbol{\gamma}(t)\right);t\ge 0\right)$ from Theorem \ref{MainThm1Intro} with initial condition $\mathbf{X}_0^{\Omega_+}=(\mathbf{x},\gamma)$. Then, almost surely, for all $t \ge 0$, $\left(\mathsf{x}_i(t)\right)_{i=1}^\infty \in \mathbb{W}_{\infty,+}^\circ$ and the $\left(\mathsf{x}_{i}(\cdot)\right)_{i=1}^\infty$ is a weak\footnote{Recall that a weak solution means that on a (filtered) probability space $\left(\mathcal{V},\left(\mathfrak{V}_t\right)_{t\ge 0},\mathbf{V}\right)$ we can define a $\left(\mathfrak{V}_t\right)_{t\ge 0}$-adapted sequence of independent standard Brownian motions $(\mathsf{w}_i)_{i=1}^\infty$ and the process $(\mathsf{x}_i)_{i=1}^\infty$  which is adapted with respect to $\left(\mathfrak{V}_t\right)_{t\ge 0}$ such $\mathbf{V}$-a.s. (\ref{ISDEintro}) holds for all $t\ge 0$; in particular the pair $((\mathsf{x}_i)_{i=1}^\infty,(\mathsf{w}_i)_{i=1}^\infty)$ is considered a weak solution to \eqref{ISDEintro}, see \cite{Revuz-Yor,KaratzasShreve}. }
solution to the following ISDE with logarithmic interaction starting from $\mathbf{x}$, namely it satisfies,

\begin{equation}\label{ISDEintro}
 \mathsf{x}_i(t)=x_i+\int_{0}^t \mathsf{x}_i(s)\mathrm{d}\mathsf{w}_i(s) -\frac{\eta}{2}\int_0^t\mathsf{x}_i(s)\mathrm{d}s+\int_0^t\sum_{j=1,j\neq i}^{\infty}\frac{\mathsf{x}_i(s)\mathsf{x}_j(s)}{\mathsf{x}_i(s)-\mathsf{x}_j(s)}\mathrm{d}s, \ \ i \in \mathbb{N},
\end{equation}
where the $\left(\mathsf{w}_i(t);t\ge 0\right)$ are independent standard Brownian motions. Moreover, if we denote by $\left(\mathbf{X}(t;\gamma);t\ge 0\right)$ the solution above corresponding to $\gamma$, then
\begin{equation*}
\mathsf{Law} \big(\mathbf{X}(\cdot;\gamma)\big)\neq \mathsf{Law}\big(\mathbf{X}(\cdot;\tilde{\gamma})\big), \textnormal{whenever } \gamma \neq \tilde{\gamma}.
\end{equation*}
Finally, out of these solutions there exists a unique one such that almost surely $t \mapsto \sum_{i=1}^\infty \mathsf{x}_i(t;\gamma)$ is continuous for all $t \ge 0$ given by the choice $\gamma=\sum_{i=1}^\infty x_i$
and this solution is a Markov process.
\end{thm}

Note that, the singular interaction drift in \eqref{ISDEintro} is really a renormalised\footnote{Observe that, without subtracting the diverging counterterm the logarithmic interaction is clearly infinite.} logarithmic interaction term, since
\begin{equation*}
\sum_{j=1, j\neq i}^\infty\frac{\mathsf{x}_i(t)\mathsf{x}_j(t)}{\mathsf{x}_i(t)-\mathsf{x}_j(t)}=\lim_{N\to \infty }\left[(-N+1)\mathsf{x}_i(t)+\mathsf{x}_i(t)^2\sum_{j=1, j\neq i}^N \frac{1}{\mathsf{x}_i(t)-\mathsf{x}_j(t)}\right].
\end{equation*}
In some sense, we could interpret the dynamics \eqref{ISDEintro} as geometric/exponential Brownian motions interacting via a renormalised logarithmic potential. As far as we know, this result is the first construction of solutions to an ISDE with logarithmic interaction from every single initial condition for which the infinitesimal singular drift term makes sense at time $t=0$. It is optimal since we need the condition $\sum_{i=1}^\infty x_i<\infty$ for otherwise the drift term is infinite. Also, we note that, the initial configurations, and more generally the distribution of the dynamics at time $t\ge 0$, do not need to resemble in any qualitative sense the invariant measure (when it exists) to be defined below, like they do, to some extent, in the case of the bulk Dyson ISDE \cite{Tsai}. This result is also the first, as far as we can tell\footnote{As far as we can tell, even for the class of initial conditions considered in the works \cite{Tsai,OsadaTanemura}, a single solution is first constructed which is then shown (under additional assumptions) to be unique.}, concrete construction of non-unique (even from a single initial condition) solutions for ISDE with logarithmic interaction. This is an explicit illustration of a phenomenon that is not seen in finite-dimensional random matrix dynamics. 

Observe that, by the change of variables $\mathsf{y}_i=\log \mathsf{x}_i$ and It\^{o}'s formula, which is valid by virtue of Theorem \ref{MainThm2Intro}, we construct solutions to the following infinite system of singularly interacting Brownian motions from any initial condition, $\mathbf{y}=(y_i)_{i=1}^\infty \in \mathbb{R}^\mathbb{N}$, so that $y_i>y_{i+1}$, for all $i\in \mathbb{N}$, and $\sum_{i=1}^\infty \mathrm{e}^{y_i}<\infty$ (again this requirement is optimal),
\begin{equation}\label{ISDElogcoord}
\mathrm{d}\mathsf{y}_i(t)=\mathrm{d}\mathsf{w}_i(t)-\frac{1+\eta}{2}\mathrm{d}t+
\sum_{j=1, j\neq i}^{\infty}\frac{\mathrm{e}^{\mathsf{y}_j(t)}}{\mathrm{e}^{\mathsf{y}_i(t)}-\mathrm{e}^{\mathsf{y}_j(t)}}\mathrm{d}t, \  \ i\in\mathbb{N}.
\end{equation}
The interaction term ($\partial_{\mathsf{y}_i}\log|1-\mathrm{e}^{\mathsf{y}_j-\mathsf{y}_i}|$) between particles $\mathsf{y}_i$ and $\mathsf{y}_j$ is reminiscent, but not the same, to the interaction term $\coth(\mathsf{y}_i-\mathsf{y}_j)$ coming up in (finite-dimensional) radial Heckman-Opdam processes, see \cite{HeckmanOpdam}.

We now move to our results on the invariant measure and convergence to equilibrium for the dynamics when $\eta>-1$. Note that, this restriction is also required for the finite-dimensional dynamics \eqref{FSDE} for otherwise there is no invariant probability measure. In fact, the construction of solutions to the ISDE with logarithmic interaction above, when $\eta \le -1$, appears to be the first, of any kind, for which there is no invariant probability measure when viewed as an evolution on unlabelled point processes. We need some more notation and terminology. Let $\mathscr{I}\subset \mathbb{R}$ be a union of open intervals. Let $\mathsf{Conf}(\mathscr{I})$ denote the space of configurations over $\mathscr{I}$, namely locally finite collections of points in $\mathscr{I}$, or equivalently the space of $\mathbb{Z}_+$-valued measures on $\mathscr{I}$ endowed with the vague topology, see \cite{BorodinDet} for details. A determinantal point process on $\mathscr{I}$, with correlation kernel $\mathfrak{K}$, is a probability measure on $\mathsf{Conf}(\mathscr{I})$ which is determined\footnote{Subject to certain mild conditions so that the correlation functions determine the point process, see \cite{Lenard}.} by the fact that all its correlation functions $\left(\rho_{n}\right)_{n=1}^\infty$ with respect to Lebesgue measure are given as $\rho_n(x_1,x_2,\dots,x_n)=\det(\mathfrak{K}(x_i,x_j))_{i,j=1}^n$, see \cite{JohanssonDet,BorodinDet} for rigorous details.

\begin{defn}
Let $\eta>-1$. The inverse Bessel point process $\mathsf{IBes}^{\textnormal{Conf}}_\eta$ with parameter $\eta$ is the determinantal point process on $(0,\infty)$ with correlation kernel $\mathsf{K}_{\eta}$ given by 
\begin{align*}
\mathsf{K}_\eta(x,y)=
\frac{8}{xy}\mathbb{J}_\eta\left(\frac{8}{x},\frac{8}{y}\right), \ \ x,y \in (0,\infty),
\end{align*}
where $\mathbb{J}_\eta\left(x,y\right)$ is the so-called the Bessel kernel, with $J_\eta$ the Bessel function of order $\eta$,
\begin{align*}
\mathbb{J}_\eta\left(x,y\right)=
\frac{x^{\frac{1}{2}}J_{\eta+1}\left(x^{\frac{1}{2}}\right)J_{\eta}\left(y^{\frac{1}{2}}\right)-
y^{\frac{1}{2}}J_{\eta+1}\left(y^{\frac{1}{2}}\right)J_{\eta}\left(x^{\frac{1}{2}}\right)
}{2(x-y)}.
\end{align*}
\end{defn}

$\mathsf{IBes}^{\textnormal{Conf}}_\eta$ is called the inverse Bessel point process because it simply consists of the inverse points (up to multiplicative constant) of the Bessel determinantal point process with correlation kernel $\mathbb{J}_\eta$ \cite{ForresterHardEdge}. The Bessel point process is the universal limit arising at the hard edge of random matrices, see \cite{ForresterHardEdge,ForresterBook,RiderBesselUniversality}.

Since $\mathsf{IBes}^{\textnormal{Conf}}_\eta$ consists of distinct points it gives rise to a unique probability measure  $\mathsf{IBes}_\eta$ on $\mathbb{W}_{\infty,+}^\circ$ by labelling the points of $\mathsf{IBes}^{\textnormal{Conf}}_\eta$ in a decreasing order. Define the space
\begin{align*}
\Omega_+^0\overset{\textnormal{def}}{=}\left\{\omega=\left(\mathbf{x},\gamma\right)\in\Omega_+:\sum_{i=1}^{\infty}x_i=\gamma\right\},
\end{align*}
and write $\pi:\Omega_+ \to \mathbb{W}_{\infty,+}$ for the map given by $\pi((\mathbf{x},\gamma))=\mathbf{x}$. 

\begin{defn}\label{DefInvariantMeasure}
Let $\eta>-1$. Define $\mathfrak{M}^\eta$ to be the unique Borel probability measure on $\Omega_+$ which is supported on $\Omega_+^0$ and satisfies 
\begin{align*}
\pi_*\mathfrak{M}^{\eta}=\mathsf{IBes}_{\eta}.
\end{align*}
\end{defn}

Then, we have the following result on convergence to equilibrium for the process $\left(\mathbf{X}_t^{\Omega_+};t \ge 0\right)$ on $\Omega_+$.

\begin{thm}\label{MainThmConvEq} 
Let $\eta>-1$ and $\mathfrak{K}$ a Borel probability measure on $\Omega_+$. Consider the Feller process $\left(\mathbf{X}_t^{\Omega_+};t \ge 0\right)$ constructed in Theorem \ref{MainThm1Intro} and suppose $\mathsf{Law}\left(\mathbf{X}_0^{\Omega_+} \right)=\mathfrak{K}$. Then, as $t\to\infty$,
\begin{equation}
\mathbf{X}_{t}^{\Omega_+} \overset{\textnormal{d}}{\longrightarrow} \mathbf{Z}, \ \ \textnormal{where } \mathsf{Law}\left(\mathbf{Z}\right)=\mathfrak{M}^\eta.
\end{equation}
In particular, if $\eta>-1$, $\mathfrak{M}^\eta$ is the unique invariant measure of $\left(\mathbf{X}_t^{\Omega_+};t \ge 0\right)$.
\end{thm}

\begin{rmk}
It is interesting to note that the finite-dimensional version of \eqref{ISDEintro}, which is different from \eqref{FSDE}, does not have an invariant probability measure. Looking at the equation in log-coordinates \eqref{ISDElogcoord} one may guess why $\eta=-1$ is the critical value. For $\eta>-1$ the individual Brownian particles have a negative drift and want to move towards $(-\infty,0)$, while the interaction force wants to push particles up towards $(0,\infty)$ which balances things out. Making this vague intuition rigorous, let alone identifying the invariant measure so explicitly, directly at the level of \eqref{ISDEintro}, \eqref{ISDElogcoord} appears to be difficult. We instead prove Theorem \ref{MainThmConvEq} by making use of the finite-dimensional dynamics \eqref{FSDE}.
\end{rmk}

As far as we can tell, the result above is the only convergence to equilibrium result on ISDE with logarithmic interaction. Nevertheless, there is very interesting recent work on ergodicity of unlabelled infinite-dimensional diffusions (namely looking at diffusions on the space $\mathsf{Conf}$) with logarithmic interaction. Dyson's model in the bulk is considered in \cite{OsadaErgodicity} and a more general class of models with determinantal correlations is treated by Suzuki in \cite{Suzuki}. Both \cite{OsadaErgodicity} and \cite{Suzuki}  use  a Dirichlet form approach. The paper \cite{OsadaErgodicity} builds on Osada-Tanemura's theory and \cite{Suzuki} makes use of the tail-triviality and number rigidity properties of determinantal point processes, see \cite{OsadaErgodicity,Suzuki} for details.

We finally prove the following result which says that although we can construct infinitely many non-equal in law solutions for the ISDE (\ref{ISDEintro}) starting from the inverse Bessel points $\mathsf{IBes}_\eta$, there is in some sense only one true equilibrium solution. Another interesting feature is that all of these different solutions start distributed according to $\mathsf{IBes}_\eta$, all but one of them will stop being distributed according to $\mathsf{IBes}_\eta$ instantaneously (this does not appear in the statement below but will be clear from the proof), but all of them will converge in distribution back to $\mathsf{IBes}_\eta$ as $t\to \infty$.

\begin{thm}\label{MainThmEquilibriumProcess}
Let $\eta>-1$. Let $\mathfrak{m}$ be a Borel probability measure on $\Omega_+$ satisfying $\pi_*\mathfrak{m}=\mathsf{IBes}_{\eta}$. Consider the Feller process $\left(\mathbf{X}_t^{\Omega_+};t\ge 0\right)$ with random initial condition $\mathsf{Law}\big(\mathbf{X}_0^{\Omega_+}\big)=\mathfrak{m}$ and write $\left(\mathbf{X}_t^{\Omega_+};t\ge 0\right)=\left(\left(\mathbf{X}(t;\mathfrak{m}),\boldsymbol{\gamma}(t;\mathfrak{m})\right);t\ge 0\right)$ where $\mathbf{X}(t;\mathfrak{m})=\left(\mathsf{x}_i(t;\mathfrak{m})\right)_{i=1}^\infty$. Then, $\left(\mathsf{x}_i(\cdot;\mathfrak{m})\right)_{i=1}^\infty$ solves (\ref{ISDEintro}) with initial condition distributed according to $\mathsf{IBes}_\eta$. Out of these solutions (for different $\mathfrak{m}$) there exists a unique solution $\mathbf{X}(\cdot;\mathfrak{m})$ satisfying
\begin{equation*}
\mathsf{Law}\big(\mathbf{X}(t;\mathfrak{m})\big)=\mathsf{IBes}_\eta, \ \ \forall t \ge 0,
\end{equation*}
and it is given by the choice $\mathfrak{m}=\mathfrak{M}^\eta$.
\end{thm}

\begin{rmk}
 Observe that, for all $\lambda>0$, if $(\mathsf{x}_i(\cdot))_{i=1}^\infty$ is a solution to \eqref{ISDEintro} with initial condition $(\mathsf{x}_i(0))_{i=1}^\infty=\mathbf{x}\in \mathbb{W}_{\infty,+}^\circ$ then $(\lambda\mathsf{x}_i(\cdot))_{i=1}^\infty$ is a solution with initial condition $\lambda \mathbf{x}$. There is an obvious analogue of constructing solutions to \eqref{ISDElogcoord} by translation. Fix $\mathbf{x} \in \mathbb{W}_{\infty,+}^\circ$, $\lambda >0$, $\gamma \ge \lambda^{-1}\sum_{i=1}^\infty x_i$ and define,
 \begin{equation}\label{TwoParSol}
\left(\mathsf{x}_i(\cdot;\lambda,\gamma)\right)_{i=1}^\infty\overset{\textnormal{def}}{=}\pi\left(\lambda \mathbf{X}_\cdot^{\Omega_+}\right), \ \ \textnormal{where} \ \mathbf{X}_0^{\Omega_+}=(\lambda^{-1}\mathbf{x},\gamma).
 \end{equation}
 Then, by virtue of Theorem \ref{MainThm2Intro} and the observation above, for any pair $(\lambda,\gamma)$, \eqref{TwoParSol} is a solution to \eqref{ISDEintro} with initial condition $\mathbf{x}$. Moreover, at least for $\eta>-1$, whenever $(\lambda,\gamma)\neq (\tilde{\lambda},\tilde{\gamma})$, we have
 \begin{equation*}
\mathsf{Law}\left(\left(\mathsf{x}_i(\cdot;\lambda,\gamma)\right)_{i=1}^\infty\right) \neq \mathsf{Law}\left(\left(\mathsf{x}_i(\cdot;\tilde{\lambda},\tilde{\gamma})\right)_{i=1}^\infty\right).
 \end{equation*}
 For $\lambda \neq \tilde{\lambda}$ this follows from Theorem \ref{MainThmConvEq} as the two processes have different $t \to \infty $ limits (we need $\eta>-1$ here), while for $\lambda=\tilde{\lambda}$ and $\gamma \neq \tilde{\gamma}$ (we do not need $\eta>-1$ for this case) this follows by the same argument given in Theorem \ref{thm-NonUniqueness} to prove the $\lambda=1$ case. This phenomenon is in stark contrast with the finite-dimensional version of \eqref{ISDEintro} and also the bulk Dyson ISDE \eqref{DysonISDE}.
\end{rmk}

\subsection{Strategy of proof}\label{SectionStrategy}
\textbf{(a) Convergence on path-space and Feller property.} Our starting point is a certain non-obvious consistency relation the finite-dimensional dynamics enjoy when $N$ varies. Let $\left(\mathfrak{P}_N(t)\right)_{t\ge 0}$ be the Markov semigroup associated to the unique strong solution of \eqref{FSDE}. For $\mathbf{x}=(x_1>x_2>\cdots>x_{N+1})$, we consider the Markov kernel, 
\begin{equation*}
\Lambda_N^{N+1}(\mathbf{x},\mathrm{d}\mathbf{y})=\frac{N!\prod_{1\le i<j \le N}(y_i-y_j)}{\prod_{1\le i<j \le N+1}(x_i-x_j)}\mathbf{1}_{x_1>y_1>x_2>\cdots>x_N>y_N>x_{N+1}}\mathrm{d}\mathbf{y}.
\end{equation*}
This Markov kernel has an important interpretation in terms of Hermitian random matrices whose law is invariant under unitary conjugation \cite{Baryshnikov,ProjOrbitalMeas,GelfandNaimark}  and can be extended to $\mathbf{x}$ with coinciding coordinates. We will say more in Section \ref{Preliminaries}. The consistency relation we are alluding to, is the following so-called intertwining, between the semigroups:
\begin{equation}\label{IntertwiningIntro1}
\mathfrak{P}_{N+1}(t)\Lambda_{N}^{N+1}=\Lambda_N^{N+1}\mathfrak{P}_N(t), \ \ \forall t\ge 0, \ N\in \mathbb{N}.
\end{equation}
We give a direct proof of this relation but also a different argument using matrix stochastic calculus which gives more intuition regarding where it actually comes from. It is worth mentioning that Dyson Brownian motion also satisfies a relation of this kind, see \cite{OConnellYor,Warren}.

Now, one can view $(\mathbb{W}_{N,+},\Lambda_{N}^{N+1})_{N=1}^\infty$ as a projective system of measures or equivalently as a Markov chain moving in discrete time, see \cite{Winkler} for background. Then, it is a classical theorem of Pickrell \cite{Pickrell} and Olshanski-Vershik \cite{Olshanski-Vershik}, in equivalent form, that the entrance boundary of this Markov chain can be identified with the space $\Omega_+$ and the extremal entrance laws $\Lambda_N^{\infty}(\omega,\mathrm{d}\mathbf{x})$, with $\omega\in \Omega_+$, have a random matrix interpretation. Moreover, it can be shown that all Markov kernels appearing above are Feller continuous. Then, by the method of intertwiners of Borodin and Olshanski \cite{MarkovProc-pathSpace-GT} the above considerations guarantee the existence of a Feller semigroup $\left(\mathfrak{P}_\infty(t)\right)_{t\ge 0}$ and associated Feller process $\mathbf{X}^{\Omega_+}$ on $\Omega_+$ which is uniquely determined via the intertwinings:
\begin{equation}\label{IntertwiningIntro2}
\mathfrak{P}_{\infty}(t)\Lambda_{N}^{\infty}=\Lambda_N^{\infty}\mathfrak{P}_N(t),\ \ \forall t \ge 0, \ N \in \mathbb{N}.
\end{equation}

To prove convergence to equilibrium in the long-time $t$ limit for $\mathbf{X}^{\Omega_+}$, Theorem \ref{MainThmConvEq}, we extend the method of intertwiners to deal with such questions by proving the following simple statement. If the finite-dimensional processes converge to equilibrium from any initial condition then the infinite-dimensional process converges to equilibrium from any initial condition. This holds in a general setting and can be applied to other models. In particular, it applies directly to the infinite-dimensional Feller processes constructed in \cite{MarkovProc-pathSpace-GT,MarkovDynam-ThomaCone} to show convergence to equilibrium for them.

Coming back to describing the Feller process $\mathbf{X}^{\Omega_+}$, at this stage we know essentially nothing about its trajectories. We do not know if and how it is approximated by the finite dimensional dynamics. We do not know if it has continuous sample paths. It might even be, in some way, degenerate\footnote{This may well happen. The most famous example satisfying \eqref{IntertwiningIntro1}, Dyson Brownian motion itself, also gives rise, by the method of intertwiners, to a Feller process on an extension $\Omega$ of the space $\Omega_+$, see Section \ref{Preliminaries}, but this ``infinite-dimensional stochastic process" is a degenerate deterministic flow on the coordinates of $\Omega$.}.

Proving convergence on the path-space $C(\mathbb{R}_+,\Omega_+)$ would resolve these questions. Since we are dealing with Feller processes one can prove this convergence by proving convergence at the level of the infinitesimal generators for a sufficiently large class of functions. This is in general highly non-trivial though since we are dealing with infinite-dimensional objects. But we have more structure which helps. Let us write $\mathsf{L}_N$ and $\mathsf{L}_\infty$ for the infinitesimal generators of the Feller semigroups $\left(\mathfrak{P}_N(t)\right)_{t\ge 0}$ and $\left(\mathfrak{P}_\infty(t)\right)_{t\ge0}$ respectively. First of all, if $\mathscr{C}_N$ is a core for $\mathsf{L}_N$ for each $N \in \mathbb{N}$, the space of functions
\begin{equation}\label{IntroCore}
\textnormal{span}\left(\bigcup_{K=1}^{\infty}\Lambda_K^{\infty}\mathscr{C}_K\right),
\end{equation}
is known to be a core for $\mathsf{L}_\infty$. In our setting, $\mathscr{C}_K$ will be a class of smooth functions with compact support with symmetry, see Section \ref{SubSec-UnifAppThm}. Then, for an arbitrary $\mathbf{f}$ in \eqref{IntroCore}, of the form $\mathbf{f}=\Lambda_K^\infty g$, for $g \in \mathscr{C}_K$, we will show that there exists $\mathbf{f}_N$ in the domain of $\mathsf{L}_N$ such that we have the following convergence in a uniform sense (we are abusing notation here; to be precise we have to embed all of them on $\Omega_+$), as $N\to\infty$,
\begin{align}
\mathbf{f}_N &\longrightarrow \mathbf{f}, \label{IntroFnApprox}\\
\mathsf{L}_N \mathbf{f}_N &\longrightarrow \mathsf{L}_\infty \mathbf{f}.\label{IntroGenApprox}
\end{align}
The right choice for $\mathbf{f}_N$ will be nothing but,
\begin{equation*}
\mathbf{f}_N=\Lambda_{N-1}^N \Lambda_{N-2}^{N-1}\cdots \Lambda_{K}^{K+1}g.
\end{equation*}
Thus, showing \eqref{IntroFnApprox} will turn out to be equivalent to proving 
\begin{equation}\label{IntroUnifApproxThm}
 \sup_{\mathbf{x}^{(N)}\in \mathbb{W}_{N,+}}\left|\Lambda_K^{\infty}g\left(\omega\left(\mathbf{x}^{(N)}\right)\right)-\Lambda_{N-1}^{N}\Lambda_{N-2}^{N-1}\cdots\Lambda_{K}^{K+1}g\left(\mathbf{x}^{(N)}\right)\right|\overset{N\to \infty}{\longrightarrow} 0,   
\end{equation}
where $\omega\left(\mathbf{x}^{(N)}\right)$ denotes a certain embedding of $\mathbf{x}^{(N)}\in \mathbb{W}_{N,+}$ into $\Omega_+$, see Section \ref{Preliminaries}. 

This is what we call the uniform approximation theorem and it is arguably our main contribution for this part of the argument. It is inspired by an analogous theorem in the combinatorial setting of the Gelfand-Tsetlin graph \cite{Boundary-GT-newApp}. There are some important new difficulties which arise in the continuous setting of random matrices however. First of all, the explicit formulae for $\Lambda_K^\infty$ and $\Lambda_{N-1}^{N}\cdots \Lambda_{K}^{K+1}$ as $K\times K$ determinants, due to Olshanski-Vershik \cite{Olshanski-Vershik} and Olshanski \cite{ProjOrbitalMeas}, that we start with, when extended to arbitrary $\omega\in \Omega_+$ and $\mathbf{x}\in \mathbb{W}_{N,+}$ involve as entries derivatives of distributions which is problematic. Fortunately, we can make sense of them by virtue of the fact that we are testing\footnote{There is a tension here between the need for the function spaces $\mathscr{C}_K$ to be large enough to deduce convergence from them while at the same time they should only include nice enough functions for the explicit formulae to make sense. We note that, there are alternative, random matrix interpretations of $\Lambda_K^{\infty}$ and $\Lambda_{N-1}^{N}\cdots \Lambda_K^{K+1}$, see Section \ref{Preliminaries}, which can be tested against arbitrary bounded Borel functions but, as far as we can tell, these cannot be used to prove \eqref{IntroUnifApproxThm}, or even guess that a statement like \eqref{IntroUnifApproxThm} is true.} against compactly supported symmetric (which is important) smooth functions $g\in \mathscr{C}_K$. Moreover, in the combinatorial setting once one gets a good explicit formula as a $K\times K$ determinant for the analogues of $\Lambda_K^{\infty}$ and $\Lambda_{N-1}^{N}\cdots \Lambda_K^{K+1}$ the asymptotic analysis is fairly straightforward: one needs to take the limit of each individual entry in the determinant. In the random matrix setting, despite some post-processing of the formulae of \cite{Olshanski-Vershik,ProjOrbitalMeas} to make them look as similar as possible, this does not appear to be the case. For this reason we needed to devise a non-trivial inductive argument in order to establish the uniform asymptotics \eqref{IntroUnifApproxThm}. Finally, to prove \eqref{IntroGenApprox}, by virtue of the intertwining \eqref{IntertwiningIntro2}, it again boils down to \eqref{IntroUnifApproxThm}.
Hence, given \eqref{IntroFnApprox} and \eqref{IntroGenApprox}, by general results \cite{Ethier-Kurtz}, we obtain the path-space convergence in \eqref{PathSpaceConvIntro}. Some additional arguments then also give convergence in $\ell^2$ as in \eqref{l^2Intro}.

It is interesting to note that in the above argument we did not need to use any explicit information about the generator $\mathsf{L}_\infty$. This should not be too surprising. The whole point, and in some sense appeal, of the method of intertwiners is that at the level of dynamics we only need to work with $\mathsf{L}_N$ for fixed $N$. However, this is also a caveat. If one wants to understand the actual dynamics of $\mathbf{X}^{\Omega_+}$ new arguments are needed. 

\textbf{(b) Non-intersection of the paths.} The next key step is to show non-intersection of the limiting paths $\left(\mathsf{x}_i(\cdot)\right)_{i=1}^\infty$. This is essential input to obtain the ISDE. It is also essential input in the approach of Tsai \cite{Tsai} and Osada-Tanemura \cite{OsadaTanemura}; in Tsai's work proving this is one of the major challenges, see discussion around Remark 2.9 therein. As far as we can tell, the method of intertwiners is not useful for proving this type of statement. There is also a new Gibbs resampling property behind the SDE \eqref{FSDE}, that we attempted to use, but this approach runs into problems that we explain in Section \ref{SectionGibbs}. So we turned back to make use of the stochastic equations. This cannot be done directly in limit: we have no stochastic equations for the infinite-dimensional process to speak of yet! Instead, we need to control, uniformly in $N$, the probabilities that the finite-dimensional paths solving \eqref{rescaledSDE} come close directly from the SDE. As far as we can tell, this is the first\footnote{Both Tsai's \cite{Tsai} and Osada-Tanemura's \cite{OsadaTanemura} approaches for non-intersection do not make direct use of the finite-dimensional SDE. Similarly, the Gibbsian line ensemble approach \cite{CorwinHammond} to prove non-intersection is totally different and does not use SDE in any way.} time this has been done. Of course, for fixed $N$, showing non-intersection is standard and this is the mainstream approach to attack Dyson-like SDE, see \cite{AGZ}, but the bounds in the literature blow up with $N$. There is good reason for this, that we comment on below.

In essence, although the actual proof is presented very differently in terms of multiple stopping times, we establish statements which informally can be interpreted as the following. Let $\mathbf{x}^{(N)}$ be arbitrary in the interior of $\mathbb{W}_{N,+}$ such that $\omega(\mathbf{x}^{(N)})\longrightarrow \omega=(\mathbf{x},\gamma)$, with $\mathbf{x}\in \mathbb{W}_{\infty,+}^\circ$. Then, for any $n \in \mathbb{N}$, $T\ge 0$ and $\varepsilon>0$, we can find $\delta$ such that,
\begin{equation}\label{CollisionTimesBOundIntro}
\sup_{N\ge n+1}\mathrm{Prob}\left(\inf_{t\in [0,T]}\left|\mathsf{x}_n^{(N)}(t)-\mathsf{x}_{n+1}^{(N)}(t)\right|<\delta\Big|\mathsf{x}^{(N)}(0)=\mathbf{x}^{(N)}\right)< \varepsilon.
\end{equation}

The reason the bounds, and observables used to obtain them, that are given in the literature for fixed $N$, are not useful as $N \to \infty$, is that they are controlling the collision times of all $N$ paths simultaneously and this cannot work. Instead we perform a double induction argument to prove the statements above one pair of paths at a time. The main ingredients are certain novel observables or Lyapunov functions associated to the dynamics given by variations of the reverse characteristic polynomial of a matrix with eigenvalues $(\mathsf{x}^{(N)}_{i}(t))_{i=1}^N$, see the functions $f_n^N$ and $\Psi_{n}^N$ and discussion around \eqref{f_n^N}. The crux of the argument, and the reason considering the aforementioned observables is actually useful is the following. When applying It\^{o}'s formula to them, the term which comes from the singular drift, which would be hard to control uniformly in $N$, has negative sign and can simply be dropped to give an $N$-independent bound. This feature is rather generic. For example the argument works for the $\beta$-ensemble version\footnote{Include a multiplicative constant $\beta$ in front of the singular interaction drift.} of the dynamics and variations of it should work for other models. Even if we hadn't known convergence of the finite-dimensional dynamics, the argument gives that any possible subsequential limit would necessarily consist of non-intersecting paths.

\textbf{(c) The ISDE.} Finally, to prove $(\mathsf{x}_i(\cdot))_{i=1}^\infty$ solves the ISDE \eqref{ISDEintro}, modulo technicalities, this boils down to convergence of the singular drift term in the finite-dimensional SDE to its infinite-dimensional counterpart. Again, viewing things in terms of certain ``characteristic polynomials" of the process will do the trick. In particular, if we write
\begin{equation*}
\Phi_i^N\left(z;\mathbf{x}^{(N)}\right)=\prod_{j=1,j\neq i}^{N}\left(1-x_j^{(N)}z\right)^2,\  \  \mathbf{x}^{(N)}\in\mathbb{W}_{N,+},
\end{equation*}
then we can observe that,
\begin{equation*}
\sum_{j\neq i}^N \frac{\mathsf{x}_i^{(N)}(t)\mathsf{x}_j^{(N)}(t)}{\mathsf{x}_i^{(N)}(t)-\mathsf{x}_j^{(N)}(t)}=\frac{1}{2}\frac{\mathrm{d}}{\mathrm{d} z}\log \Phi_i^N\left(z,\mathsf{x}^{(N)}(t)\right)\bigg|_{z=\left(\mathsf{x}_i^{(N)}(t)\right)^{-1}}.
\end{equation*}
Then, making use of the following fact, by virtue of convergence in $\Omega_+$, where convergence holds uniformly on compact sets in $z\in \mathbb{C}$,
\begin{equation*}
\Phi_i^N\left(z;\mathsf{x}^{(N)}(t)\right) \longrightarrow  \mathrm{e}^{-2\left(\boldsymbol{\gamma}(t)-\mathsf{x}_i(t)\right) z}\prod_{j=1,j\neq i}^{\infty}\mathrm{e}^{2\mathsf{x}_j(t)z}\left(1-\mathsf{x}_j(t)z\right)^2,
\end{equation*}
combined with some further arguments that make use of the fact that paths are non-intersecting, and no $\mathsf{x}_i$ hits zero, we can show convergence of the drift term.

At this stage the ISDE for the $(\mathsf{x}_i(\cdot))_{i=1}^\infty$ coordinates may still depend on $\boldsymbol{\gamma}(\cdot)$. Then, the final part of the argument, making use of the structure of the ISDE shows that almost surely, for all $t>0$, $\sum_{i=1}^\infty \mathsf{x}_i(t)=\boldsymbol{\gamma}(t)$ which completes the proof. Thus, in some sense, the true state space of the diffusion $\mathbf{X}^{\Omega_+}$ is $\Omega_+^0$ and $\Omega_+ \backslash \Omega_+^0$ acts as a kind of entrance boundary for it. At this point all the hard work is done and the proofs of the remaining statements in our main results are fairly straightforward. Finally, an interesting point to note in the argument is that, in this infinite-dimensional limit, some ``additional randomness" is created, which is in some sense the source of non-uniqueness of solutions to the ISDE \eqref{ISDEintro}, see the discussion preceding Theorem \ref{thm-ISDE} and its proof for more details.

\subsection{Previous approaches}\label{SectionHistory}

The kinds of problems we are considering have been studied intensely for decades. As a result a number of rather sophisticated approaches have been developed to attack them. We survey them below. To different extents, we have attempted to apply each of them to our problem.

\textbf{(1)} Tsai's important paper \cite{Tsai} was the only, up until now, work which gave existence of solutions to an ISDE with logarithmic interaction from explicit initial conditions. Tsai uses in a very essential and delicate way a  certain shift-invariance of the dynamics and a novel monotonicity property of the corresponding process of gaps between particles. The shift-invariance is definitely not true for our model and it is unclear what the right analogue of the monotonicity property (if one exists) should be. For these fundamental reasons it seems hard to adapt the approach of \cite{Tsai} to even give solutions  from a class of initial conditions in our case. Finally, in regards to invariant measures and convergence to equilibrium this approach does not seem well-adapted for dealing with such questions.

\textbf{(2)} We now move to Osada and Tanemura's program \cite{OsadaCMP,OsadaAOP,OsadaPTRF,OsadaTanemura}. This spans almost three decades and several papers so it is impossible to survey in a few lines. The upshot is the following: one is given as basic data a probability measure $\mathsf{M}$ on the configuration space $\mathsf{Conf}(\mathscr{Z})$ of some Euclidean space $\mathscr{Z}$ consisting of unlabelled collections of points. From this, under certain assumptions, one first constructs a diffusion on $\mathsf{Conf}(\mathscr{Z})$, starting from any initial condition in an abstract set $\mathscr{S}_\mathsf{M} \subset \mathsf{Conf}(\mathscr{Z})$,  so that $\mathsf{M}(\mathscr{S}_\mathsf{M})=1$, with $\mathsf{M}$ as invariant measure. With this as input, and under further assumptions, a weak solution to a labelled ISDE is constructed (always starting and remaining in a labelled version of $\mathscr{S}_\mathsf{M}$), which under additional conditions is shown to be the unique one satisfying certain properties. For some of these conditions\footnote{These are for the most part conditions on the correlation functions of $\mathsf{M}$ and its Campbell measures \cite{OsadaTanemura}.} black-box theorems already exist to verify them, however some of them do need to be checked by hand. This is in general a highly non-trivial task. For example, the main purpose of the long preprint \cite{OsadaTanemuraAiry} is to verify these conditions in the case of the ISDE associated to the stationary version of the Airy line ensemble in order to apply the theory.

Let us be more precise on how we think the Osada-Tanemura theory may relate to our approach. We believe that if one attempts (for $\eta>-1$ certain modifications are needed but we will not check the technical conditions here, while for $\eta \le -1$ a fundamental extension of this theory is required since there is no invariant probability measure on $\mathsf{Conf}((0,\infty))$ to start with) to use this theory to construct ``near-equilibrium" solutions for our model, these should match, in a way that needs to be made precise, for initial conditions $\mathbf{x}\in \mathbb{W}_{\infty,+}^\circ$ that qualitatively resemble the invariant measure, the solutions that arise from $\Omega_+^0$. It seems plausible that this is a more general phenomenon. The Dirichlet form approach will pick out a distinguished solution from the ones living on an enhanced space. Another related point is that the singular interaction drift of the ISDE in Osada-Tanemura's theory is given by the so-called logarithmic derivative of the point process \cite{OsadaPTRF}. In some sense this notion associates to the invariant point process a canonical limiting characteristic polynomial (a random analytic function). It would be interesting to understand these potential connections more precisely and rigorously in the future.

\textbf{(3)} The recent work of Landon \cite{landon2020edge} on the edge scaling of $\beta$-Dyson Brownian motion at equilibrium is also worth mentioning. Using a different approach to previous works, that takes rigidity estimates for the invariant point process (the $\beta$-Airy point process \cite{RiderRamirezVirag}) as a key ingredient, Landon can prove convergence under a subsequence for the equilibrium dynamics for all $\beta \ge 1$ (the paper \cite{OsadaTanemuraAiry} only deals with $\beta=1,2,4$); the fact that the process is in equilibrium is essential for the use of rigidity estimates. However, the limiting process is not shown to solve an ISDE, or even have non-intersecting paths. Nevertheless, it is very natural to expect that these statements should be true.

\textbf{(4)} A different method for proving path-space convergence of random matrix dynamics, having special structure, is the probabilistic approach of Gibbsian line ensembles \cite{CorwinHammond,XuanWu}. This type\footnote{We will not give the precise definitions here, see for example \cite{CorwinHammond,XuanWu}.} of Gibbs property is, in some sense, the right dimension-independent analogue of the fact that the $N$-dimensional Dyson Brownian motion \eqref{DysonIntro} can equivalently be constructed as $N$ independent Brownian motions conditioned to never intersect \cite{Grabiner}. This approach is  useful for proving tightness of the paths. To conclude convergence some more information is needed. The current state of the art, in the case of models in the KPZ universality class, is the recent remarkable strong characterisation theorem \cite{aggarwal2023strong} of the Airy line ensemble. This says that, beyond a Brownian resampling property, we only need to know that the top path approximates a parabola to pin down the Airy line ensemble.

\textbf{(5)} Yet another approach to taking limits of random matrix dynamics is to view the $N$-dimensional interacting SDE as diffusions on $\mathsf{Conf}$ and take the limit of their space-time correlation functions. The state of the art, for limits of Dyson Brownian motion in the bulk and non-intersecting squared Bessel processes starting from a class of initial conditions, appears to be the work of Katori-Tanemura \cite{KatoriTanemura1,KatoriTanemura2}. As far as we can tell, this approach cannot, on its own at least, be used to obtain results on ISDE.

\textbf{(6)} We now comment on the method of intertwiners\footnote{For uses of intertwining relations, with other applications in mind, see for example \cite{BorodinDynamics,Warren,OConnellToda,OConnellMatrixDiffusions,MicloPatie1,MicloPatie2}.}. An exposition can be found in the lecture notes and ICM proceedings by Olshanski \cite{OlshanskiLectureNotes,OlshanskiICM}. The method was originally introduced in a discrete setting to construct infinite-dimensional Feller processes coming from continuous-time jump dynamics on partitions which satisfy an intertwining. These Markov dynamics can be thought of as discrete variants of finite-dimensional SDE with logarithmic interaction. The main example the method was developed for in \cite{MarkovProc-pathSpace-GT} is related to the projective system associated to the Gelfand-Tseltin graph \cite{Boundary-GT-newApp}, whose boundary can be identified with an infinite-dimensional space $\mathfrak{GT}$ with continuous parameters, akin to $\Omega_+$. The invariant measure of the process, the analogue of $\mathfrak{M}^\eta$, is given by a so-called $zw$-measure which comes up in representation theory and harmonic analysis \cite{BorodinOlshanskiHarmonic,OlshanskiHarmonic}. One could think of these measures as a more sophisticated version of determinantal point processes that arise in random matrix theory. The current state of the art on this construction is the paper \cite{OlshanskiGenerator} of Olshanski, where the generator of the process on $\mathfrak{GT}$ is computed explicitly, not in the natural coordinates of $\mathfrak{GT}$, but rather in a different set of variables and is realised by a second order differential operator. The main idea is that in this other set of variables, which come from abstract algebraic considerations, the generator is not singular. Nevertheless, it is still not known if the process of \cite{MarkovProc-pathSpace-GT,OlshanskiGenerator} has continuous\footnote{Of course, by the Feller property it has a modification with c\`{a}dl\`{a}g paths \cite{Kallenberg,Revuz-Yor}.} sample paths  and this is one of the questions posed in \cite{OlshanskiGenerator}. In fact, it is not unreasonable, in analogy with the results of our paper, to expect even more: that in natural coordinates this process should solve some kind of ISDE with singular interaction, at least for some initial conditions. This is probably a hard problem.

In the random matrix setting the method of intertwiners was first applied by one of us in \cite{H-P}. There an abstract construction of an infinite-dimensional Feller process related to the Cauchy ensemble \cite{ForresterWitte,BorodinOlshanski} is provided. We recall the main result of \cite{H-P} in Theorem \ref{thm-HPabstractThm} below and we extend it, using the machinery we develop in this paper, to a strong path-space convergence result which also proves that the limiting process has continuous sample paths.

\textbf{(7)} Finally, let us mention that for a wide class of finite-dimensional systems with interactions of the form  \eqref{LogInteraction} a complete solution theory exists, see \cite{Graczyk-Malecki}. In finite dimensions one can even construct dynamics when collisions occur between particles (this corresponds to $\beta<1$ in the Dyson SDE) using the theory of multivalued SDE, see \cite{Cepa1,Cepa2,Demni}. In infinite dimension, for Dyson's model for $\beta<1$, when viewing the dynamics as a process on configuration space,  the corresponding Dirichlet form was constructed in \cite{suzuki2022curvature}, making use of the DLR equations for the sine process \cite{DLRsine}. It is natural to expect that this process on configuration space (for $\beta<1$) will involve collisions between particles but this has not been shown yet. At the level of labelled SDE, as far we can tell, nothing is known for $\beta<1$.

\subsection{Open problems}\label{SectionOpenProblems}
We end this introduction with a number of open problems.
\begin{enumerate}
    \item \textbf{Gibbs property.} It would be  interesting if the limiting non-intersecting paths $(\mathsf{x}_i)_{i=1}^\infty$ satisfy a Gibbs resampling property \cite{CorwinHammond}. We say more about this in Section \ref{SectionGibbs}.
    \item \textbf{Space-time determinantal correlations.} A special class of finite-dimensional dynamics with logarithmic interaction, including \eqref{FSDE}, starting from any deterministic initial condition, is known to have determinantal space-time correlations, see \cite{ExactSolution}. It is an interesting question whether the limiting process we constructed enjoys the space-time determinantal property with an explicit correlation kernel. A somewhat easier question, which nevertheless involves the same fundamental difficulties, is to try to do this for the dynamics starting from the equilibrium measure.
    \item \textbf{Rate of convergence to equilibrium.} It would be interesting if one could quantify (in any sense) the convergence to equilibrium from Theorem \ref{MainThmConvEq}.
    \item \textbf{Existence of solutions for other models.} 
    Our approach to construct solutions to ISDE with logarithmic interaction will also work for the dynamical Cauchy model in  Section \ref{SectionCauchy}. There is an additional and non-trivial complication there in that the pre-limit paths can cross zero. Towards a more general class of models, one needs to replace the first, integrable, part of our argument with more robust techniques coming from stochastic analysis or probability. Then\footnote{In fact, most likely an extension of the second part of our argument could be used to replace certain aspects of the first. At least when it comes to tightness of the paths, a uniform bound like \eqref{CollisionTimesBOundIntro} is the main ingredient in \cite{CorwinHammond}. Note that, if one only cares about the ISDE we do not need the full power of a result like Theorem \ref{MainThm1Intro}.}, we believe variations of the second part of our argument can be applied.
    
    \item \textbf{Uniqueness of solutions.} We proved a certain kind of uniqueness in Theorem \ref{MainThm2Intro}. However, a more natural question is to prove uniqueness, for a general class of ISDE with singular interaction, within a space of paths that take values in certain ``rigid configurations", the analogue of ``balanced configurations" from the work of Tsai for the Dyson ISDE \cite{Tsai}. The precise meaning of ``rigid configurations" will need to be model-specific. One would hope that in such spaces of paths the tail-end of the singular drift term can be controlled which is the key requirement for uniqueness. Finally, one needs to ensure solutions taking values in such spaces exist in the first place; a-priori it is not clear at all that they should.
\end{enumerate}

\paragraph{Organisation of the paper} In Section \ref{SectionConvergence}, we first recall the method of intertwiners of \cite{MarkovProc-pathSpace-GT} in a general setting. We then extend it to deal with convergence to equilibrium and most importantly we prove the uniform approximation theorem and its various consequences on consistent Feller processes. In Section \ref{SectionFiniteDimDiffusions}, we establish various properties of the solution to the finite-dimensional SDE \eqref{FSDE} that allow us to apply the results of Section \ref{SectionConvergence}. We also briefly discuss a new Gibbs resampling property for \eqref{FSDE}. In Section \ref{Section-ISDE}, we first show non-intersection of the limiting paths in Theorem \ref{thm-NonCollision,>0} and then obtain the ISDE in Theorem \ref{thm-ISDE}. In Section \ref{SectionCauchy}, we quickly apply our framework to a dynamical model for the Cauchy ensemble, upgrading the abstract results of \cite{H-P}.

\paragraph{Acknowledgements} TA is grateful to Duncan Dauvergne and Evgeni Dimitrov for a useful discussion regarding the Gibbs resampling property and to Kohei Suzuki for very informative discussions on the Dirichlet form approach to infinite-dimensional diffusions with singular interaction and helpful comments on an earlier draft. Part of this research was performed while TA was visiting the Institute for Pure and Applied Mathematics (IPAM), which is supported by the National Science Foundation (Grant No. DMS-1925919). ZSM is grateful to EPSRC for funding through a PhD studentship.
   
 \section{Convergence of consistent Markov processes}\label{SectionConvergence}
The main purpose of this section is to establish a general result for path space convergence for a sequence of Markov processes related to random matrices  satisfying  consistency relations.

\subsection{Method of intertwiners: Construction of the Feller process}\label{SubSec-Intertwining}

We recall the method of intertwiners of \cite{MarkovProc-pathSpace-GT}. Since there is a subtle difference in that we put the weak topology on spaces of probability measures instead of the topology induced by the total variation norm used in \cite{MarkovProc-pathSpace-GT} we give a detailed exposition and some proofs. To set things up, we will need to collect several definitions and facts which can be found, for example, in \cite{Winkler}.

The measurable structures on topological spaces below are always understood to be the Borel $\sigma$-algebras. A Polish space $\mathscr{X}$ is a separable completely metrisable topological space. A standard Borel space is a measurable space $(\mathscr{X},\mathscr{F})$ such that there exists a metric on the base set so that the induced Borel $\sigma$-algebra is $\mathscr{F}$. A Borel isomorphism is a bi-measurable bijective function between two standard Borel spaces. 

We write $\mathscr{M}_{\textnormal{p}}(\mathscr{X})$ for the set of Borel probability measures on a Polish space $\mathscr{X}$. Unless otherwise stated, we always endow $\mathscr{M}_{\textnormal{p}}(\mathscr{X})$ with the weak topology. This makes it a Polish space (and so a standard Borel space). Moreover, the Borel $\sigma$-algebra coincides with the $\sigma$-algebra generated by the evaluation maps $\mu \mapsto \mu(\mathscr{A})$, for $\mathscr{A}$ measurable (but the topology on the base set $\mathscr{M}_{\textnormal{p}}(\mathscr{X})$ generated by these maps is different from the weak topology). 

For a Polish space $\mathscr{X}$ denote by $\mathscr{B}(\mathscr{X})$ the space of real-valued bounded Borel functions on $\mathscr{X}$. A Markov kernel $\mathscr{K}: \mathscr{X}_1 \to \mathscr{X}_2$ from $\mathscr{X}_1$ to $\mathscr{X}_2$ is a function $\mathscr{K}(x,\mathscr{A})$, where $x\in \mathscr{X}_1$ and $\mathscr{A} \subset \mathscr{X}_2$ is a measurable subset, such that $\mathscr{K}(x,\cdot)$ is a probability measure on $\mathscr{X}_2$ and $x\mapsto\mathscr{K}(x,\mathscr{A})$ is a measurable function on $\mathscr{X}_1$. Observe that, a Markov kernel $\mathscr{K}: \mathscr{E}_1 \to \mathscr{X}_2$ can be viewed as a map on probability measures $\mathscr{K}:\mathscr{M}_{\textnormal{p}}(\mathscr{X}_1)\to \mathscr{M}_\textnormal{p}(\mathscr{X}_2)$ or as a map on functions $\mathscr{K}:\mathscr{B}(\mathscr{X}_2) \to \mathscr{B}(\mathscr{X}_1)$ in the obvious way.

For a locally compact Polish space $\mathscr{X}$, we write $C_0(\mathscr{X})$ for the space of continuous (real-valued) functions vanishing at infinity, endowed with the supremum norm which makes it a Banach space.
 
\begin{defn} Let $\mathscr{X}_1$ and $\mathscr{X}_2$ be locally compact Polish spaces. A Markov kernel $\mathscr{K}:\mathscr{X}_1 \to \mathscr{X}_2$ is Feller if the induced map from $\mathscr{B}(\mathscr{X}_2)$ to $\mathscr{B}(\mathscr{X}_1)$ maps $C_0(\mathscr{X}_2)$ to $C_0(\mathscr{X}_1)$.
\end{defn}

The following is our basic data. Suppose that we are given:
\begin{enumerate}
    \item A sequence of locally compact Polish spaces $(\mathcal{E}_N)_{N=1}^\infty$.
    \item For each $N\in \mathbb{N}$, a Feller Markov kernel $\mathcal{L}_N^{N+1}:\mathcal{E}_{N+1} \to \mathcal{E}_N$.
\end{enumerate}

We would now like to associate a certain limit object to the above data. There are a number of (not always equivalent) ways of doing this \cite{Winkler,OlshanskiHarmonic}; see in particular Section 4.3.1 of \cite{Winkler} for the connection to entrance boundaries of Markov chains mentioned in the introduction. Below, we discuss the bare minimum we need for our applications.

To begin with, observe that for any $N \in \mathbb{N}$, the affine maps $\mathcal{L}_N^{N+1}:\mathscr{M}_\textnormal{p}(\mathcal{E}_{N+1}) \to \mathscr{M}_\textnormal{p}(\mathcal{E}_N)$ are continuous. Hence, we can take the projective limit of the convex topological spaces $\mathscr{M}_{\textnormal{p}}(\mathcal{E}_N)$, which is defined as follows, see Section 3 of \cite{Winkler},
\begin{equation}\label{ProjectiveLimitDef}
\lim_{\leftarrow \textnormal{Top}} \mathscr{M}_\textnormal{p}\left(\mathcal{E}_N\right)\overset{\textnormal{def}}{=}\left\{(\mu_N)_{N=1}^\infty \in \prod_{N=1}^\infty \mathscr{M}_{\textnormal{p}}(\mathcal{E}_N): \mu_{N+1}\mathcal{L}_{N}^{N+1}=\mu_N, \ \forall N \in \mathbb{N}\right\}.
\end{equation}
This is endowed with the projective limit topology. It is in fact a Polish simplex (and in particular a standard Borel space), see Section 4 of \cite{Winkler}. The following terminology will be used throughout the paper.
\begin{defn}
We call any element of \eqref{ProjectiveLimitDef} a consistent sequence of (probability) measures.
\end{defn}

We then come to the following definition.
\begin{defn}
Suppose that $\mathcal{E}_\infty$ is a locally compact Polish space, equipped with Feller Markov kernels $\mathcal{L}_N^{\infty}:\mathcal{E}_\infty \to \mathcal{E}_N$ satisfying, for all $N \in \mathbb{N}$, $\mathcal{L}_{N+1}^{\infty}\mathcal{L}_N^{N+1}=\mathcal{L}_{N}^\infty$, so that the map,
\begin{align}\label{Borel isomorphism}
    \mathscr{M}_\textnormal{p}(\mathcal{E}_\infty) &\to \lim_{\leftarrow \textnormal{Top}} \mathscr{M}_\textnormal{p}\left(\mathcal{E}_N\right)\\
\mathfrak{m}&\mapsto \left(\mathfrak{m}\mathcal{L}_N^\infty\right)_{N=1}^\infty,\nonumber
\end{align}
is a Borel isomorphism. Then, we say that $\mathcal{E}_\infty$ is a Feller boundary of the system $(\mathcal{E}_N,\mathcal{L}_N^{N+1})_{N=1}^\infty$.
\end{defn}

We will not be concerned with existence and uniqueness of such an object for general systems here. In related settings, abstract existence and uniqueness (in an appropriate measurable sense) results exist, see \cite{Winkler,OlshanskiHarmonic}. For our ultimate purposes of constructing dynamics on $\mathcal{E}_\infty$, the above abstract setting is only useful if we have a nice explicit description for it and the kernels $\mathcal{L}_N^\infty$. This is in general a difficult problem but in some cases it admits a good solution. This is the case for the concrete applications to random matrix dynamics we have in mind.

We can now prove the following lemma, which is essentially Lemma 2.3 of \cite{MarkovProc-pathSpace-GT}.
\begin{lem}\label{LemmaDensity}
Assume $\mathcal{E}_\infty$ is a Feller boundary of $\left(\mathcal{E}_N,\mathcal{L}_N^{N+1}\right)_{N=1}^\infty$. Then, the space 
\begin{equation}\label{DefDenseSet}
\textnormal{span}\left(\bigcup_{K=1}^{\infty}\mathcal{L}_K^{\infty}C_0(\mathcal{E}_K)\right)
\end{equation}
is dense in $C_0(\mathcal{E}_\infty)$.
\end{lem}

\begin{proof}
On a locally compact Polish space $\mathscr{X}$, let $\mathscr{M}(\mathscr{X})$ denote the set of signed measures with finite total variation norm $||\cdot ||_{\textnormal{tv}}$.
By the Riesz representation theorem the Banach space $(\mathscr{M}(\mathcal{E}_\infty),||\cdot ||_{\textnormal{tv}})$ is the Banach dual of $C_0(\mathcal{E}_\infty)$. Hence, in order to prove the result it suffices to show that if $\mu \in \mathscr{M}(\mathcal{E}_\infty)$ annihilates all functions in \eqref{DefDenseSet} then necessarily $\mu = 0$.
Observe that, this last statement makes no reference to the topology we put on $\mathscr{M}(\mathcal{E}_\infty)$. Assume that $\mu \in \mathscr{M}(\mathcal{E}_\infty)$  annihilates $\mathcal{L}_N^\infty C_0(\mathcal{E}_N)$, for all $N \in \mathbb{N}$. This is equivalent, again by the Riesz representation theorem applied to each $C_0(\mathcal{E}_N)$, to $\mu\mathcal{L}_N^\infty \in \mathscr{M}(\mathcal{E}_N)$ being $0$, for all $N \in \mathbb{N}$. We can write $\mu(\cdot)=||\mu_+||_{\textnormal{tv}}\mu_+(\cdot)-||\mu_-||_{\textnormal{tv}}\mu_-(\cdot)$ for $\mu_+,\mu_- \in \mathscr{M}_{\textnormal{p}}(\mathcal{E}_\infty)$. Since $\mu \mathcal{L}_N^\infty =0$, for all $N \in \mathbb{N}$ we obtain $||\mu_+||_{\textnormal{tv}}=||\mu_-||_{\textnormal{tv}}$ and thus $\mu_+\mathcal{L}_N^\infty=\mu_-\mathcal{L}_N^\infty$, for all $N \in \mathbb{N}$. Since the map $\eqref{Borel isomorphism}$ is a bijection we must have $\mu_+=\mu_-$ and hence $\mu =0$ as desired.
\end{proof}

Before constructing Feller Markov dynamics on $\mathcal{E}_\infty$, we need to recall the definition of a Feller semigroup, see \cite{Kallenberg,Revuz-Yor}.

\begin{defn}\label{DefnFellerSemigroup} Let $\mathscr{X}$ be a locally compact Polish space.
A semigroup $(\mathsf{P}_t)_{t\ge 0}$ of positive contraction operators on $C_0(\mathscr{X})$ is called a Feller (or Feller-Dynkin) semigroup if: (i) $\mathsf{P}_t C_0(\mathscr{X})\subset C_0(\mathscr{X})$, for all $t\ge 0$, (ii) $\mathsf{P}_tf(x)\to f(x)$, as $t\to 0$, for all $f\in C_0(\mathscr{X})$, $x\in \mathscr{X}$. Associated to $(\mathsf{P}_t)_{t\ge 0}$ we have a unique Markov process, called a Feller process, on $\mathscr{X}$.
\end{defn}

Now, for each $N\in \mathbb{N}$, let $\left(\mathbf{X}^{(N)}(t);t\ge 0\right)$ be a Feller process on $ \mathcal{E}_N $ with  associated semigroup $ \left(\mathsf{P}_{N}(t)\right)_{t\geq 0}$, and suppose these semigroups are consistent 
with $ (\mathcal{L}_N^{N+1})_{N=1}^{\infty}$, namely, 
\begin{align}\label{MethodOfInterEq}
\mathsf{P}_{N+1}(t)\mathcal{L}_N^{N+1} = \mathcal{L}_N^{N+1}\mathsf{P}_{N}(t),\ \ \forall t\geq 0,\; \forall N\in \mathbb{N}.
\end{align}
The following result from \cite{MarkovProc-pathSpace-GT}, essentially Proposition 2.4 therein, gives the existence of a Feller process $\mathbf{X}^{(\infty)} $ on $\mathcal{E}_{\infty}$ consistent with the $\mathbf{X}^{(N)}$. Because the main idea is beautifully simple and since there are some subtleties with measurability we give some details.

\begin{thm}\label{thm-MarkovProcessonBoundary}
	Let $\mathcal{E}_{\infty}$ be a Feller boundary of $ (\mathcal{E}_N,\mathcal{L}_N^{N+1})_{N=1}^{\infty}$.  For all $N\in \mathbb{N}$, let $\left(\mathsf{P}_N(t)\right)_{t\ge 0}$ be Feller semigroups satisfying (\ref{MethodOfInterEq}). Then, there exists a unique Feller Markov semigroup
	$ \left(\mathsf{P}_{\infty}(t)\right)_{t\geq 0} $ on $ \mathcal{E}_{\infty}$, with associated process $\left(\mathbf{X}^{(\infty)}(t);t\ge 0\right)$, such that 
	\begin{align}\label{Pt-intertwining}
		\mathsf{P}_{\infty}(t)\mathcal{L}_N^{\infty}
		= \mathcal{L}_N^{\infty}\mathsf{P}_{N}(t),\ \ \forall t\geq 0,\;\forall N\in\mathbb{N}.
	\end{align}
 \end{thm}

\begin{proof}
To construct the semigroup $(\mathsf{P}_\infty(t))_{t\ge 0}$ we first need to define, for each $t \ge 0$, $x\in \mathcal{E}_\infty$, a probability measure $\mathsf{P}(t)(x,\cdot)$ on $\mathcal{E}_\infty$. Note that, this probability measure needs to satisfy,
\begin{equation}\label{IntermediateIntertwiningProof}
\mathsf{P}(t)(x,\cdot)\mathcal{L}_N^\infty(\mathrm{d}y)=\mathcal{L}_N^\infty \mathsf{P}_N(t)(x,\mathrm{d}y).
\end{equation}
Then, note that, because of \eqref{MethodOfInterEq}, and since $\mathcal{L}_{N}^{\infty}=\mathcal{L}^\infty_{N+1}\mathcal{L}_N^{N+1}$, we have
\begin{equation*}
 \left(\mathcal{L}_N^\infty\mathsf{P}_N(t)(x,\cdot)\right)_{N=1}^\infty\in   \lim_{\leftarrow \textnormal{Top}} \mathscr{M}_\textnormal{p}\left(\mathcal{E}_N\right),
\end{equation*}
and by virtue of the fact that the map \eqref{Borel isomorphism} is a Borel isomorphism this defines a unique probability measure $\mathsf{P}(t)(x,\cdot)$ on $\mathcal{E}_\infty$ as desired. Observe also that,  since \eqref{IntermediateIntertwiningProof} holds for each $x\in \mathcal{E}_\infty$, it is equivalent to \eqref{Pt-intertwining}. To see that $\mathsf{P}(t):\mathcal{E}_\infty \to \mathcal{E}_\infty$ is a genuine Markov kernel we first observe that the map $\mathcal{E}_\infty \to \mathscr{M}_\textnormal{p}(\mathcal{E}_\infty)$ given by $x\mapsto \mathsf{P}(t)(x,\cdot)$ is measurable. This is a consequence of the following facts: the map $\mathcal{E}_\infty \to \mathscr{M}_{\textnormal{p}}(\mathcal{E}_\infty)$ which takes $x$ to its delta measure is continuous, the map \eqref{Borel isomorphism} is a Borel isomorphism and so bi-measurable and the fact that, for each $t\ge 0$, the map,
\begin{align*}
\lim_{\leftarrow \textnormal{Top}} \mathscr{M}_\textnormal{p}\left(\mathcal{E}_N\right) &\to \lim_{\leftarrow \textnormal{Top}} \mathscr{M}_\textnormal{p}\left(\mathcal{E}_N\right) \\
(\mu_N)_{N=1}^\infty&\mapsto \left(\mu_N\mathsf{P}_N(t)\right)_{N=1}^\infty,
\end{align*}
is continuous, since for each $N \in \mathbb{N}$, $(\mathsf{P}_N(t))_{t\ge 0}$ is Feller. Hence, to conclude that, for $\mathscr{A}$ measurable, $x\mapsto \mathsf{P}(t)(x,\mathscr{A})$ is a measurable function, it suffices to observe that the evaluation maps (for measurable $\mathscr{A}$) $\mathscr{M}_\textnormal{p}(\mathcal{E}_\infty) \to [0,1]$ given  by $\mu \mapsto \mu(\mathscr{A})$ are measurable; recall the $\sigma$-algebra generated by these maps is the same as the Borel $\sigma$-algebra (but the induced topology on the set $\mathscr{M}_\textnormal{p}(\mathcal{E}_\infty)$ is different from the weak topology). The semigroup property then follows by virtue of \eqref{IntermediateIntertwiningProof} and the semigroup property of the $(\mathsf{P}_N(t))_{t\ge 0}$. Finally, the Feller continuity property, boils down, by virtue of Lemma \ref{LemmaDensity}, to the Feller property of the $\mathcal{L}_N^\infty$ and $(\mathsf{P}_N(t))_{t\ge 0}$. See \cite{MarkovProc-pathSpace-GT} for more details.
\end{proof}

Note that, the way we obtained $(\mathsf{P}_\infty(t))_{t\ge 0}$ is non-constructive. Thus a direct description of its dynamics becomes a whole separate task. On the other hand, we have obtained a-apriori the highly non-trivial Feller continuity property, which especially when $\mathcal{E}_\infty$ turns out to be infinite-dimensional is a major achievement. We also remark that it is easy to see that we can  obtain invariant measures for $\left(\mathsf{P}_\infty(t)\right)_{t\ge 0}$ from invariant measures for the $\left(\mathsf{P}_N(t)\right)_{t\ge 0}$, see \cite{MarkovProc-pathSpace-GT}. Since we will prove a stronger statement in Theorem \ref{thm-ergodicity} below we do not make this explicit. We now move to some important definitions.

\begin{defn}\label{DefGenerator}
 Let $\mathscr{X}$ be a locally compact Polish space and let $(\mathsf{P}_t)_{t\ge 0}$ be a Feller semigroup on $C_0(\mathscr{X})$. Its infinitesimal generator $\mathcal{A}$ is defined to be,
 \begin{equation*}
    \mathcal{A}f=\lim_{t\downarrow 0}t^{-1}\left(\mathsf{P}_t f-f\right),
 \end{equation*}
 whenever the limit exists in $C_0(\mathscr{X})$.
 The domain of $\mathcal{A}$, that we denote throughout the paper by $\mathcal{D}(\mathcal{A})$, is the set of $f\in C_0(\mathscr{X})$ for which the limit exists.
\end{defn}

In the setting of Theorem \ref{thm-MarkovProcessonBoundary}, we denote by 
$\mathsf{L}_N$ and $\mathsf{L}_{\infty}$ the generators, with domains $\mathcal{D}(\mathsf{L}_N)\subset C_0(\mathcal{E}_N)$ and $\mathcal{D}(\mathsf{L}_\infty)\subset C_0(\mathcal{E}_\infty)$, associated to the Feller processes $(\mathbf{X}^{(N)}(t);t\ge 0)$ and $(\mathbf{X}^{(\infty)}(t);t\ge 0)$ respectively. The following standard definition will be useful.

    \begin{defn}\label{DefCore}
    Let $\mathcal{A}$ be a Feller generator with domain $\mathcal{D}(\mathcal{A})$. 
A subspace $\mathscr{C}\subseteq \mathcal{D}(\mathcal{A})$ is
    called a core of $\mathcal{A}$ if the closure of $\mathcal{A}|_{\mathscr{C}}$ coincides with $\mathcal{A}$, where, by $\mathcal{A}|_{\mathscr{C}}$ we mean restriction of $\mathcal{A}$ to $\mathscr{C}$.
    \end{defn}
    \noindent
    The generator is clearly uniquely determined by its action on a core. Then, the following very useful lemma, see Proposition 5.2 of \cite{MarkovDynam-ThomaCone} for a proof, provides a core for $\mathsf{L}_{\infty}$.

    \begin{lem}\label{lem core}
    In the setting of Theorem \ref{thm-MarkovProcessonBoundary}, for each $K\in \mathbb{N}$ suppose $\mathscr{C}_K\subseteq \mathcal{D}(\mathsf{L}_K)$ is a core for the generator $\mathsf{L}_K$. The subspace $\mathscr{C}_{\infty} \subseteq C_0(\mathcal{E}_{\infty})$ spanned by all subspaces of the form $\mathcal{L}_K^{\infty}\mathscr{C}_K,\; K = 1, 2, \dots $, namely
\begin{equation}\label{core}
    \mathscr{C}_{\infty}\overset{\textnormal{def}}{=} \textnormal{span}\left(\bigcup_{K=1}^{\infty}\mathcal{L}_K^{\infty}\mathscr{C}_K\right)
\end{equation}
is then a core for the generator $\mathsf{L}_{\infty}$.
\end{lem}

\subsection{Method of intertwiners: Convergence to equilibrium}

We now extend the formalism of the method of intertwiners to also deal with convergence to equilibrium for the process on the boundary. The following result, despite its simplicity, appears to be new. Let $C_b(\mathscr{X})$ denote the space of (real-valued) bounded continuous functions on a locally compact Polish space $\mathscr{X}$.

 \begin{thm}\label{thm-ergodicity}
     In the setting of Theorem \ref{thm-MarkovProcessonBoundary}, suppose $(\mu_N)_{N=1}^\infty \in \lim_{\leftarrow \textnormal{Top}} \mathscr{M}_\textnormal{p}\left(\mathcal{E}_N\right) $ with  corresponding boundary measure $\mu_\infty \in \mathscr{M}_{\textnormal{p}}(\mathcal{E}_\infty)$ uniquely determined by,
     \begin{equation*}
      \mu_\infty \mathcal{L}_N^\infty=\mu_N, \ \ \forall N\in\mathbb{N}.
     \end{equation*}
     Moreover, suppose that for all $\left(\nu_N\right)_{N=1}^{\infty}\in \lim_{\leftarrow \textnormal{Top}} \mathscr{M}_\textnormal{p}\left(\mathcal{E}_N\right) $, we have
       \begin{align}
           \nu_N\mathsf{P}_{N}(t)(f)\overset{t\to\infty}{\longrightarrow} \mu_{N}(f),\ \ \forall f\in C_b(\mathcal{E}_N), \  \forall N \in \mathbb{N}.
       \end{align}  
       Then, for any $\nu_\infty \in \mathscr{M}_\textnormal{p}(\mathcal{E}_\infty)$, we have 
       \begin{align}\label{ergodicity}
            \nu_\infty\mathsf{P}_{\infty}(t)(f)\overset{t\to\infty}{\longrightarrow} \mu_{\infty}(f),\ \ \forall f\in C_b(\mathcal{E}_\infty).
       \end{align} 
       In particular, $\mu_\infty$ is the unique invariant measure of $\left(\mathsf{P}_\infty(t)\right)_{t\ge 0}$.
    \end{thm}
    \begin{proof}
    We show \eqref{ergodicity} holds for all functions in
$\mathscr{C}=\bigcup_{N=1}^{\infty}\mathcal{L}_N^{\infty}C_0(\mathcal{E}_N)$. Observe that, this implies \eqref{ergodicity} for all $f\in C_0(\mathcal{E}_\infty)$ as, by virtue of Lemma \ref{LemmaDensity}, $\mathscr{C}$ is a dense subset of $C_0(\mathcal{E}_\infty)$. This is vague convergence, which together with the fact that the limiting measure $\mu_{\infty}$ is a probability measure, by a standard argument upgrades it to convergence in distribution which proves the theorem.
        
        Fix $N\in\mathbb{N}$ and let $f_N\in C_0(\mathcal{E}_N)$. We have using the intertwining relation \eqref{Pt-intertwining} that
        \begin{align*}
            \nu_\infty\mathsf{P}_{\infty}(t)\mathcal{L}_N^{\infty}(f_N)=\nu_\infty\mathcal{L}_N^{\infty}\mathsf{P}_{N}(t)(f_N),\  \ \forall t\ge 0.
        \end{align*}
        Now, set $\nu_N=\nu_\infty\mathcal{L}_N^{\infty}$. Since $\left(\nu_N\right)_{N=1}^{\infty} \in \lim_{\leftarrow \textnormal{Top}} \mathscr{M}_\textnormal{p}\left(\mathcal{E}_N\right) $, by our assumption $\nu_N\mathsf{P}_{N}(t)(f_N)\to\mu_N(f_N)$ as $t\to\infty$. Hence, 
       \begin{align*}
       \nu_\infty\mathcal{L}_N^{\infty}\mathsf{P}_{N}(t)(f_N)\to\mu_N(f_N),\  \  t\to\infty.
        \end{align*}
       Recall that $\mu_N(f_N)=\mu_{\infty}\left(\mathcal{L}_N^{\infty}f_N\right)$, and thus, display  (\ref{ergodicity}) follows. Then, the fact that $\mu_\infty$ is the unique invariant measure follows by a standard argument.
    \end{proof}

\subsection{Projective system from unitarily invariant Hermitian matrices}\label{Preliminaries}

From now on we focus on two specific projective systems that are related to random matrices. They form the basic integrable structure behind our problem. They come from the study of unitarily invariant measures on infinite Hermitian matrices, see \cite{Pickrell,Olshanski-Vershik,BorodinOlshanski,Theo-Joseph} for background. We will need several definitions and preliminary results. For the most part, the material in this subsection is standard and has appeared in equivalent forms in \cite{Pickrell,Olshanski-Vershik,BorodinOlshanski,Theo-Joseph,H-P}.

For any $N\in\mathbb{N}$, let $\mathbb{W}_N$ be the Weyl chamber,
\begin{align}\label{Weyl chamber}
 \mathbb{W}_N = \left\{\mathbf{x}=(x_1,x_2,\dots,x_N)\in\mathbb{R}^{N}: x_1\geq x_2 \ge \dots\geq x_N\right\},
\end{align}
The interior of $\mathbb{W}_N$ and $\mathbb{W}_{N,+}$ from \eqref{PosWeylChamber} will be denoted by $\mathbb{W}_N^{\circ}$ and $\mathbb{W}_{N,+}^{\circ}$ respectively.

 For each $N\in\mathbb{N}$, let 
 $\mathbb{H}(N)$ denote the space of $N\times N$ Hermitian matrices, and $\mathbb{H}_+(N)\subset\mathbb{H}(N)$ the subspace of non-negative definite ones.
 
 \begin{defn}
 We define the eigenvalue maps
 \begin{align*}
 \mathsf{eval}_N : \mathbb{H}(N) &\to \mathbb{W}_N,\\
\mathbf{A}&\mapsto \mathbf{x}=(x_1\ge x_2 \ge \dots\ge x_N),
 \end{align*}
 where $\mathbf{x}=(x_1\ge x_2 \ge \dots\ge x_N)$ are the ordered eigenvalues of the matrix $\mathbf{A}\in\mathbb{H}(N)$. Observe that, $\mathsf{eval}_N$ restricts to a map from $\mathbb{H}_+(N)$ to $\mathbb{W}_{N,+}$.
 \end{defn}

 We now let, for each $N \in \mathbb{N}$, 
\begin{align*}
\Pi_N^{N+1}:\mathbb{H}(N+1)&\to\mathbb{H}(N),\\ 
\left[\mathbf{A}_{ij}\right]_{i,j=1}^{N+1}&\mapsto\left[\mathbf{A}_{ij}\right]_{i,j=1}^{N},
\end{align*}
be the so-called corners map. Also, observe that, $\Pi_{N}^{N+1}$ restricts to a map from $\mathbb{H}_+(N+1)$ to $\mathbb{H}_+(N)$. Let $\mathbb{U}(N)$ be the space of $N\times N$ unitary matrices. We will denote by $\mathbf{A}^\dag$ the complex conjugate of a matrix $\mathbf{A}$. The following definition is one of our basic building blocks.

\begin{defn}
For any $N\in\mathbb{N}$, we define
\begin{align}\label{Lambda^N=lawEval}
\Lambda_N^{N+1}\left(\mathbf{x},\cdot\right)\overset{\textnormal{def}}{=}\mathsf{Law}\left[\mathsf{eval}_N\left(\Pi_N^{N+1}\left(\mathbf{U}\mathsf{diag}(\mathbf{x})\mathbf{U}^{\dag}\right)\right)\right],\  \ \forall\mathbf{x}\in \mathbb{W}_{N+1},
\end{align}   
where $\mathbf{U}$ is Haar-distributed on $\mathbb{U}(N)$ and $\mathsf{diag}(\mathbf{x})$ indicates the diagonal matrix with entries $\mathbf{x}$.  More generally, we define the composition
 	\begin{align}\label{Lambda_K^N}
\Lambda_K^{N}\overset{\textnormal{def}}{=}\Lambda_{N-1}^{N}\Lambda_{N-2}^{N-1}\cdots\Lambda_K^{K+1},\  \ \forall N>K\geq 1.
	\end{align} 
\end{defn}

We write $ \Delta_N$ for the Vandermonde determinant given by, for $\mathbf{x}\in \mathbb{R}^N$,  
 \begin{align*}
\Delta_N(\mathbf{x})=\prod_{1\leq i<j\leq N} (x_i-x_j).
 \end{align*}
When $\mathbf{x}\in\mathbb{W}_{N+1}$ and $\mathbf{y}\in\mathbb{W}_N$, $\mathbf{y}\prec \mathbf{x} $ denotes interlacing:
	$ x_1\geq y_1\geq\dots\geq y_N\geq x_{N+1}$. 
It is a classical result by Baryshnikov \cite{Baryshnikov} 
 (in fact the computation is implicit in \cite{GelfandNaimark}, see also \cite{ProjOrbitalMeas} for an alternative proof) that the following holds. This explicit formula will be useful in the sequel.

\begin{lem} For $N\in\mathbb{N}$, we have
	\begin{align}\label{link}
		\Lambda_N^{N+1}(\mathbf{x},\mathrm{d}\mathbf{y})=\frac{N!\Delta_N(\mathbf{y})}{\Delta_{N+1}\left(\mathbf{x}\right)}\mathbf{1}_{\{\mathbf{y}\prec \mathbf{x}\}}\mathrm{d}\mathbf{y},\  \ \forall \mathbf{x}\in \mathbb{W}_{N+1}^{\circ}.
	\end{align} 
\end{lem}

As expected these Markov kernels are Feller.

 \begin{lem}\label{lem-FellerPropLambdaN}
      Let $N>K\geq 1$. The Markov kernel $\Lambda_K^{N}$ is Feller, that is, for all $f\in C_0(\mathbb{W}_K)$, we have $\Lambda_K^{N}f\in C_0(\mathbb{W}_N)$. Similarly, for all $f\in C_0(\mathbb{W}_{K,+})$, we have $\Lambda_K^{N}f\in C_0(\mathbb{W}_{N,+})$.
 \end{lem}
 \begin{proof}
     The statement follows from Lemma 2.5 of \cite{H-P}, and using the fact that the composition of Feller kernels is Feller. 
 \end{proof}

We now define the infinite dimensional space $\Omega$, an analogue of the space $\Omega_+$ from the introduction, that will allow in the sequel to also consider negative coordinates.
\begin{defn}

We define the space $\Omega$ as follows:
\begin{align}\label{Omega}
		\Omega\overset{\textnormal{def}}{=}
		\left\{\omega=(\mathbf{x}^{+},\mathbf{x}^{-},\gamma,\delta)\in \mathbb{W}_{\infty,+}\times \mathbb{W}_{\infty,+} \times \mathbb{R} \times \mathbb{R}_+:
		\sum_{i=1}^{\infty}(x_{i}^{+})^2+\sum_{i=1}^{\infty}(x_{i}^{-})^2\leq \delta\right\},
	\end{align} 
endowed with the topology of coordinate-wise convergence. 
\end{defn}
 Observe that, $\Omega$ like $\Omega_+$ is locally compact, metrizable and separable. Its topology can be metrized with  the following metric $\mathsf{d}_{\Omega}$, with $\omega=\left(\mathbf{x}^+,\mathbf{x}^-,\gamma,\delta\right)$
\begin{equation}
\mathsf{d}_\Omega(\omega,\tilde{\omega})=\sum_{i=1}^\infty\frac{\left|x_i^+-\tilde{x}_i^+\right|}{2^i\left(1+\left|x_i^+-\tilde{x}_i^+\right|\right)}+\sum_{i=1}^\infty\frac{\left|x_i^--\tilde{x}_i^-\right|}{2^i\left(1+\left|x_i^--\tilde{x}_i^-\right|\right)}+\left|\gamma-\tilde{\gamma}\right|+\left|\delta-\tilde{\delta}\right|.
\end{equation}
We note moreover that $\Omega_+$ can be embedded into $\Omega$ as follows:
\begin{align}\label{omega_+ToOmega}
(\mathbf{x},\gamma)\mapsto\left(\mathbf{x},\mathbf{0},\gamma,\sum_{j=1}^\infty x_j^2\right),
\end{align}
and observe that that this is a homeomorphism. This embedding is the natural embedding of $\Omega_+$ into $\Omega$ and this will be clearer in the sequel.

We now embed $\mathbb{W}_N$ and $\mathbb{W}_{N,+}$ into $\Omega$ and $\Omega_+$ respectively, in a natural way (again the fact that this is a reasonable thing to do will be clear in what follows).

\begin{defn}
For each $N\in\mathbb{N}$, we embed the space $\mathbb{W}_N$ into $\Omega$ as follows, with $\mathbf{x}^{(N)}\in \mathbb{W}_N$,
 \begin{align}\label{embedding}
 \mathbf{x}^{(N)}\mapsto \left(\left(N^{-1}\max\left\{x_{i}^{(N)},0\right\}\right)_{i=1}^\infty,\left(N^{-1}\max\left\{-x_{N+1-i}^{(N)},0\right\}\right)_{i=1}^\infty,N^{-1}\sum_{i=1}^Nx_i^{(N)} ,N^{-2}\sum_{i=1}^N \left(x_i^{(N)}\right)^2\right),  
 \end{align}
 where by convention $x_i^{(N)}\equiv 0$ whenever $i\notin \{1,\dots,N\}$. We denote this embedded point by $\omega\left(\mathbf{x}^{(N)}\right)$.
Similarly, we can define an embedding of $\mathbb{W}_{N,+}$ into $\Omega_+$ as follows, with $\mathbf{x}^{(N)}\in \mathbb{W}_{N,+}$,
\begin{align}\label{embedding+}
    \mathbf{x}^{(N)}\mapsto \left(\left(N^{-1}x_i^{(N)}\right)_{i=1}^\infty,N^{-1}\sum_{i=1}^N x_i^{(N)}\right).
\end{align}
Abusing notation, we denote this embedded point in $\Omega_+$ by $\omega\left(\mathbf{x}^{(N)}\right)$ again.
\end{defn}
 
Note that, the embedding of $\mathbb{W}_{N,+}$ into $\Omega_+$ matches the embedding of the stochastic dynamics from the introduction.

Now, let 
 $\mathbb{H}(\infty)$ denotes the space of infinite Hermitian matrices, and write $\mathbb{H}_+(\infty)$ for the space of non-negative definite infinite Hermitian matrices (defined as the projective limits of the spaces $\left(\mathbb{H}(N)\right)_{N=1}^\infty$ and $\left(\mathbb{H}_+(N)\right)_{N=1}^\infty$ respectively under the corners maps $\Pi_N^{N+1}$; we do not make this more precise as we will only need to work with their finite-dimensional projections). Define the map,
\begin{align*}
\Pi_N^{\infty}:\mathbb{H}(\infty)&\to\mathbb{H}(N),\\ 
\left[\mathbf{A}_{ij}\right]_{i,j=1}^{\infty}&\mapsto\left[\mathbf{A}_{ij}\right]_{i,j=1}^{N}.
\end{align*}
Observe that this restricts to a map from $\mathbb{H}_+(\infty)$ to $\mathbb{H}_+(N)$. 

The following definition is important.

 \begin{defn}
For $\omega \in \Omega$, define the infinite Hermitian random matrix  
 \begin{align*}     \mathbf{A}_{\omega}^{\infty}\overset{\textnormal{def}}{=} \gamma\mathbf{I} + \sqrt{\delta-\sum_{i=1}^\infty\left[(x_i^+)^2+(x_i^-)^2\right]}\mathbf{G}+\sum_{j=1}^{\infty}x_j^+\left( \boldsymbol{\xi}(j)\boldsymbol{\xi}(j)^{\dag}-\mathbf{I}\right) - \sum_{j=1}^{\infty}x_j^-\left(\boldsymbol{\zeta}(j)\boldsymbol{\zeta}^{\dag}(j)-\mathbf{I}\right),
 \end{align*}
where, $\mathbf{G}$ is an infinite Gaussian Hermitian matrix, namely, 
\begin{align*}
    \mathbf{G}_{ii}\overset{\textnormal{d}}{=}\mathcal{N}(0,1),\;\;\forall i\in\mathbb{N},\  \ 
    \Re( \mathbf{G}_{ij}),\; \Im( \mathbf{G}_{ij})\overset{\textnormal{d}}{=}\mathcal{N}\left(0,\frac{1}{2}\right),\;\;\forall i<j,
\end{align*}
with all independent entries subject to the Hermitian constraint and where $\mathcal{N}\left(0,\sigma^2\right)$ denotes the centered Gaussian random variable of variance $\sigma^2$. Moreover, for any $j\in\mathbb{N}$
the $\boldsymbol{\xi(j)}$ and $\boldsymbol{\zeta(j)}$ are infinite vectors of $\textnormal{i.i.d.}$ (independent, identically distributed) 
 standard complex Gaussian random variables (real and imaginary parts are independent normally distributed random variables with mean zero and variance $\frac{1}{2}$).

For $\omega \in \Omega_+$, we defined, abusing notation,
\begin{align}\label{A-omega}
     \mathbf{A}_{\omega}^{\infty}\overset{\textnormal{def}}{=}\left(\gamma-\sum_{j=1}^{\infty}x_j\right)\mathbf{I} +\sum_{j=1}^{\infty}x_j\boldsymbol{\xi}(j)\boldsymbol{\xi}^{\dag}(j),
 \end{align}
 where, as before, $\left(\boldsymbol{\xi(j)}\right)_{j=1}^\infty$ is a sequence of infinite vectors of $\textnormal{i.i.d.}$ complex Gaussian random variables.
 \end{defn}

 We note that the definition of the matrices $\mathbf{A}_\omega^\infty$, for $\omega\in \Omega_+$, is consistent with viewing $\omega \in \Omega$ under the embedding \eqref{omega_+ToOmega}. We also note that we will never really need to consider the $\mathbf{A}_\omega^\infty$ as infinite matrices but rather only consider them through their $K\times K$ top-left corner projections $\Pi_K^\infty(\mathbf{A}_\omega^\infty)$. We can now define a family of Markov kernels that play a key role in our argument.

\begin{defn}
For any $K\in\mathbb{N}$, we define, with (abusing notation) either $\omega \in \Omega$ or $\omega \in \Omega_+$,
  \begin{align}\label{Lambda=lawEval}
\Lambda_K^{\infty}(\omega,\cdot)\overset{\textnormal{def}}{=}\mathsf{Law}\left[\mathsf{eval}_K\left(\Pi_K^{\infty}\left(\mathbf{A}_{\omega}^{\infty}\right)\right)\right].
\end{align}
    
\end{defn}

The following is easy to see.

\begin{prop}
For any $K\in\mathbb{N}$ and $\omega \in \Omega$ or $\omega \in \Omega_+$, $\Lambda_K^{\infty}(\omega,\cdot)$ is a probability measure on $\mathbb{W}_K$ or $\mathbb{W}_{K,+}$ respectively. Moreover, we have
    \begin{align*}
         \Lambda_N^{\infty} \Lambda_K^{N} = \Lambda_K^{\infty},\  \ \forall N>K.
    \end{align*}
\end{prop}
\begin{proof}
This first statement is immediate from the definition and the second statement follows by unitary-invariance by virtue of Baryshnikov's result \cite{Baryshnikov}.
\end{proof}

The kernels $\Lambda_K^\infty$ are Markov and in fact Feller as expected.

 \begin{prop}\label{prop-FellerPropLambda}
    Let $K\in\mathbb{N}$. Then, we have
    \begin{align*}
        \Lambda_K^{\infty}f&\in C_0(\Omega),\  \ \;\,\forall f\in C_0\left(\mathbb{W}_K\right),\\
        \Lambda_K^{\infty}f&\in C_0(\Omega_+),\  \ \forall f\in C_0\left(\mathbb{W}_{K,+}\right).
    \end{align*} 
In particular, the kernels are Feller Markov.
 \end{prop}
\begin{proof} Continuity of $\omega \mapsto \Lambda_N^{\infty}(\omega,\mathrm{d}\mathbf{y})$, in the sense of weak convergence, was proven in \cite{Theo-Joseph} (see the proof of Proposition 3.5 therein). From this we also obtain the measurability of $\omega \mapsto \Lambda_{K}^\infty(\omega,\mathscr{A})$ for measurable $\mathscr{A}$. It remains to prove the vanishing at infinity property. In the case of $\Omega$ this was shown in \cite{H-P}. We give a more direct argument in the case of $\Omega_+$ (an adaptation will work for $\Omega$ as well). Let $f\in C_0\left(\mathbb{W}_{K,+}\right)$ be arbitrary. 
We need to show that for any sequence $\left(\omega^{(n)}\right)_{n=1}^{\infty}$ in $\Omega_+$, with $\omega^{(n)}=\left(\mathbf{x}^{(n)},\gamma^{(n)}\right)$, such that $\omega^{(n)}\overset{n\to\infty}{\longrightarrow}\infty$ (which means $\gamma^{(n)}\to \infty$), we have $\left(\Lambda_K^{\infty}f\right)\left(\omega^{(n)}\right)\overset{n\to\infty}{\longrightarrow} 0$.
To do so, it suffices to show that when $\gamma^{(n)}\overset{n\to\infty}{\longrightarrow}\infty$, the largest eigenvalue of $\Pi_K^\infty\left(\mathbf{A}_{\omega^{(n)}}^\infty\right)$ converges to infinity in distribution. Since the eigenvalues of sub-matrices of different orders interlace (Cauchy interlacing theorem), we only need to prove the $K=1$ case. Note that, from the explicit formula \eqref{A-omega}, the corresponding random variable $\mathsf{x}_n$ is simply
 \begin{align*}
     \mathsf{x}_n\overset{\textnormal{def}}{=}\gamma^{(n)}-\sum_{j=1}^{\infty}x_j^{(n)} + \sum_{j=1}^{\infty}\mathsf{\Gamma}_{x_j^{(n)}},
 \end{align*}
 where each $\mathsf{\Gamma}_\theta$ denotes a gamma random variable with Laplace transform $z\mapsto (1+\theta z)^{-1}$ and they are all independent. Thus, we want to show that 
 $\mathsf{x}_n\overset{n\to\infty}{\longrightarrow}\infty$
 when $\gamma^{(n)}\overset{n\to\infty}{\longrightarrow}\infty$. We have two possibilities; either $\sum_{j=1}^{\infty}x_j^{(n)}$ remains bounded, in which case the result follows immediately, otherwise, $\sum_{j=1}^{\infty}x_j^{(n)}\overset{n\to\infty}{\longrightarrow}\infty$  which implies, $\sum_{j=1}^{\infty}\mathsf{\Gamma}_{x_j^{(n)}}{\longrightarrow}\infty$ in distribution. This last claim can easily be seen by virtue of the fact that, for any $z,R\ge 0$, we have,
 \begin{equation*}
  \mathbb{P}\left(\sum_{j=1}^\infty \mathsf{\Gamma}_{x_j^{(n)}}\le R\right)   \le \mathbb{E}\left[\exp\left(-z\sum_{j=1}^\infty \mathsf{\Gamma}_{x_j^{(n)}}\right)\right] \mathrm{e}^{zR} \le \frac{\mathrm{e}^{zR}}{1+z\sum_{j=1}^\infty x_j^{(n)}}.
 \end{equation*}
 This completes the proof. 
\end{proof}

The following fundamental result is due to Pickrell \cite{Pickrell} and Olshanski and Vershik \cite{Olshanski-Vershik}, in equivalent form (modulo the Feller property of the kernels). In rather less obvious equivalent form, it can be traced back to the work of Schoenberg \cite{Schoenberg,KarlinTotalPositivity} on totally positive functions, see \cite{Olshanski-Vershik}. It was translated to the setting of projective systems in \cite{H-P}. A proof, different from \cite{Pickrell,Olshanski-Vershik}, directly in this setup and a generalisation is given in \cite{Theo-Joseph}.

\begin{thm}\label{thm-Boundary}
The space $\Omega$ is a Feller boundary of $\left(\mathbb{W}_N,\Lambda_N^{N+1}\right)_{N=1}^{\infty}$.
\end{thm}

The following result is again due to Olshanski and Vershik in equivalent form (modulo the Feller property of the kernels), see Remark 2.11 in \cite{Olshanski-Vershik}, .

\begin{thm}\label{thm-Boundary+}
 The space $\Omega_+$ is a Feller boundary of $\left(\mathbb{W}_{N,+},\Lambda_N^{N+1}\right)_{N=1}^{\infty}$.
\end{thm}

The uniform approximation theorem that we state below, could also be used to prove the results above, by following the argument in Section 3 of \cite{Boundary-GT-newApp}, modulo an independent proof (we cannot use Lemma \ref{LemmaDensity} as the argument would be circular) of the density of $\cup_N \Lambda_N^\infty C_0(\mathbb{W}_N)$ and $\cup_N \Lambda_N^\infty C_0(\mathbb{W}_{N,+})$ in $C_0(\Omega)$ and $C_0(\Omega_+)$ respectively. In the combinatorial setting of the Gelfand-Tsetlin graph this is not too difficult, see Section 2.7 in \cite{Boundary-GT-newApp}, but in the continuous setting there are additional subtleties. We will not pursue this here.

\subsection{Uniform Approximation theorem}\label{SubSec-UnifAppThm}

Let $(\mathcal{E}_N)_{N=1}^{\infty}=\left(\mathbb{W}_N\right)_{N=1}^{\infty}$ or $\left(\mathbb{W}_{N,+}\right)_{N=1}^{\infty}$. Let $C_c^{\infty}
(\mathbb{R}^{K})$ denote the space of (real-valued) smooth and compactly supported functions on $\mathbb{R}^{K}$. It will be very convenient in what follows to consider the symmetric extension of the kernels given in \eqref{Lambda^N=lawEval} and \eqref{Lambda=lawEval} to $\mathbb{R}^{K}$, which we denote by the same notation. Thus, both $\Lambda_K^N(\mathbf{x},\mathrm{d}\mathbf{y})$ and $\Lambda_K^\infty(\omega,\mathrm{d}\mathbf{y})$ can be viewed as positive measures on $\mathbb{R}^{K}$ with total mass $K!$ (we choose not to normalise them to be probability measures as it will be more convenient) and can act on functions defined on $\mathbb{R}^K$ (this will be clear from context). The following is the main technical result of this part of the paper.
    
\begin{thm}\label{thm-UnifApp} Let $K\in \mathbb{N}$.
For all $g\in C_c^\infty(\mathbb{R}^K)$,
we have 
\begin{align}\label{UniformApproximation}
\sup_{\mathbf{x}^{(N)}\in \mathcal{E}_N}\left|\Lambda_K^{\infty}g\left(\omega\left(\mathbf{x}^{(N)}\right)\right)-\Lambda_K^{N}g\left(\mathbf{x}^{(N)}\right)\right|\to 0,
\ \ \textnormal{as } N\to\infty. 
 \end{align}
\end{thm}
The proof of this theorem will be given in Subsection \ref{SubSec-ProofUnifAppThm}. The following function spaces will appear frequently below.
\begin{defn}
Define $C_{c,\textnormal{sym}}^{\infty}(\mathbb{R}^{K})$ to be the space of smooth compactly supported functions on $\mathbb{R}^{K}$ that are moreover symmetric in their variables. Then, we define 
$C_{c,\textnormal{sym}}^{\infty}(\mathcal{E}_K)$ to be the space of functions on $\mathcal{E}_K$ which are restrictions of functions $C_{c,\textnormal{sym}}^{\infty}(\mathbb{R}^{K})$ to $\mathcal{E}_K$.  
\end{defn}

Observe that, by definition
 \begin{align*}
  \int_{\mathcal{E}_K}g(\mathbf{y})\Lambda_K^N(\mathbf{x},\mathrm{d}\mathbf{y})&=
     \frac{1}{K!}
     \int_{\mathbb{R}^{K}}g(\mathbf{y})\Lambda_K^N(\mathbf{x},\mathrm{d}\mathbf{y}),\\
     \int_{\mathcal{E}_K}g(\mathbf{y})\Lambda_K^\infty(\omega,\mathrm{d}\mathbf{y})&=
     \frac{1}{K!}
     \int_{\mathbb{R}^{K}}g(\mathbf{y})\Lambda_K^\infty(\omega,\mathrm{d}\mathbf{y}), \  \  \forall g\in C_{c,\textnormal{sym}}^{\infty}(\mathbb{R}^{K}).
 \end{align*}
 We have the following corollary of Theorem \ref{thm-UnifApp}.
 \begin{cor}\label{cor-UnifApp}
Let $K \in \mathbb{N}$. For all $g \in C_{c,\textnormal{sym}}^{\infty}(\mathcal{E}_K)$, we have
\begin{align}\label{UnifApp}
\sup_{\mathbf{x}^{(N)}\in \mathcal{E}_N}\left|\int_{\mathcal{E}_K}g\left(\mathbf{y}\right)\Lambda_K^{\infty}\left(\omega\left(\mathbf{x}^{(N)}\right),\mathrm{d}\mathbf{y}\right)-\int_{\mathcal{E}_K}g(\mathbf{y})\Lambda_K^{N}\left(\mathbf{x}^{(N)},\mathrm{d}\mathbf{y}\right)\right|\to 0,
\ \ \textnormal{as } N\to\infty. 
 \end{align}
 \end{cor}

\subsection{Approximation of Markov processes}\label{SubSec-AppMarkovProcesses}

In this subsection we show how using the uniform approximation theorem above and an underlying intertwining we can prove a path-space convergence result to the abstract Feller process associated to $\left(\mathsf{P}_\infty(t)\right)_{t\ge 0}$. 

To be precise, the running assumption is the following. In the setting of Subsection \ref{SubSec-Intertwining}, let $(\mathcal{E}_N)_{N=1}^{\infty}=\left(\mathbb{W}_N\right)_{N=1}^{\infty}$ or $\left(\mathbb{W}_{N,+}\right)_{N=1}^{\infty}$ and 
$ \left(\mathcal{L}_N^{N+1}\right)_{N=1}^{\infty}=\left(\Lambda_N^{N+1}\right)_{N=1}^{\infty}$. Namely, we consider the systems $\left(\mathbb{W}_N,\Lambda_N^{N+1}\right)_{N=1}^{\infty}$ and $\left(\mathbb{W}_{N,+},\Lambda_N^{N+1}\right)_{N=1}^{\infty}$, 
with $\Lambda_N^{N+1}$ given by \eqref{Lambda^N=lawEval}.  As we have seen in Theorems \ref{thm-Boundary} and \ref{thm-Boundary+}, the Feller boundaries of these systems are given by $\Omega$ and $\Omega_+$ respectively.

We moreover consider, for each $N\in \mathbb{N}$, a Feller Markov process $\left(\mathfrak{x}^{(N)}(t);t\ge 0\right)$, with continuous sample paths,
on $\mathcal{E}_N$ with transition semigroup $\left(\mathsf{P}_N(t)\right)_{N=1}^\infty$ and we assume that these semigroups are intertwined:
\begin{align}\label{P^N-intertwining}
\mathsf{P}_{N+1}(t)\Lambda_N^{N+1} = \Lambda_N^{N+1}\mathsf{P}_{N}(t),\ \ \forall t\geq 0,\; \forall N\in \mathbb{N}.
\end{align}
Let $\left(\mathbf{X}_t^{\Omega};t\geq 0\right)$ and
$\left(\mathbf{X}_t^{\Omega_+};t\geq 0\right)$
be the corresponding Feller processes on $\Omega$ and $\Omega_+$ constructed as in Theorem \ref{thm-MarkovProcessonBoundary},  respectively. We denote the Feller semigroup of the process on the boundary, abusing notation, for either $\Omega$ or $\Omega_+$, by $(\mathsf{P}_{\infty}(t))_{t\ge 0}$.
Moreover, recall that we have 
\begin{equation}\label{P-intertwining}
		\mathsf{P}_{\infty}(t)\Lambda_N^{\infty}
		= \Lambda_N^{\infty}\mathsf{P}_{N}(t),\ \ \forall t\geq 0,\;\forall N\in\mathbb{N}.
	\end{equation}
 
We need two abstract definitions that we will specialise to our setting shortly. 

\begin{defn}\label{Def fn to f}
    Suppose $\left(\mathscr{L}_n\right)_{n=1}^\infty$ is a sequence of Banach spaces, with the norm of $\mathscr{L}_n$ denoted by $\|\cdot\|_n$. Let also $\mathscr{L}$ be another Banach space with norm denoted by $\|\cdot\|$. 
    We say that a sequence of vectors $(\mathbf{f}_n)_{n=1}^\infty$ with $\mathbf{f}_n\in \mathscr{L}_n$ approximates a vector $\mathbf{f}\in\mathscr{L}$, and we write $\mathbf{f}_n\to \mathbf{f}$, if and only if there exist bounded linear operators $\pi_n: \mathscr{L}\to\mathscr{L}_n$, satisfying $\sup_{n\in\mathbb{N}}\||\pi_n\||_n<\infty$, where $\||\cdot\||_n$ denotes the operator norm of linear operators from $\mathscr{L}$ to $\mathscr{L}_n$, such that,
    \begin{equation*}
        \lim_{n\to\infty}\big\|\mathbf{f}_n-\pi_n \mathbf{f}\big\|_n=0.
    \end{equation*}
\end{defn}

\begin{defn}\label{defn-GenApprox}
Let $\mathcal{A}_N$ and $\mathcal{A}_\infty$ be the generators (defined in an analogous way to Definition \ref{DefGenerator}) corresponding to strongly continuous semigroups of contractions on the Banach spaces $\mathscr{L}_N$ and $\mathscr{L}_\infty$, respectively. We say that $\mathcal{A}_\infty$ is approximated by $\left(\mathcal{A}_N\right)_{N=1}^{\infty}$
  if for any $\mathbf{f}$ in a core (defined in an analogous way to Definition \ref{DefCore}) of $\mathcal{A}_\infty$, there exist a sequence  $\mathbf{f}_N\in\mathcal{D}(\mathcal{A}_N)$ such that $\mathbf{f}_N\to \mathbf{f}$ and $\mathcal{A}_N\mathbf{f}_N \to \mathcal{A}_{\infty}\mathbf{f}$.
\end{defn}

The following proposition is key for our purposes. Here, we take the Banach spaces to be $\mathscr{L}_N=C_0(\mathcal{E}_N)$ and $\mathscr{L}=C_0(\mathcal{E}_\infty)$ and the implicit maps $\pi_N: C_0(\mathcal{E}_\infty) \to C_0(\mathcal{E}_N)$ to be the ones induced by embeddings of $\mathcal{E}_N$ into $\mathcal{E}_\infty$, namely \eqref{embedding} or \eqref{embedding+}.

\begin{prop}\label{prop-LNtoL}
In the setting described above, we let $\mathsf{L}_N$ and $\mathsf{L}_\infty$ denote the generators of the Feller semigroups  $\left(\mathsf{P}_N(t)\right)_{t\ge 0}$ and $\left(\mathsf{P}_{\infty}(t)\right)_{t\ge 0}$, respectively. Let $\mathscr{C}_N=C_{c,\textnormal{sym}}^{\infty}(\mathcal{E}_N)$ and $\mathscr{C}_{\infty}$ be as in \eqref{core}. Moreover, assume that, for all $N \in \mathbb{N}$, $\mathscr{C}_N\subseteq\mathcal{D}(\mathsf{L}_N)$ and $\mathsf{L}_N\mathscr{C}_N \subset \mathscr{C}_N$. Then, for any $\mathbf{f}\in\mathscr{C}_{\infty}$, there exists a sequence $(\mathbf{f}_N)_{N=1}^{\infty}$ with $\mathbf{f}_N\in\mathscr{C}_N$ 
such that $\mathbf{f}_N\to \mathbf{f}$ and $\mathsf{L}_N\mathbf{f}_N \to \mathsf{L}_{\infty}\mathbf{f}$.
\end{prop}

An immediate consequence of Proposition \ref{prop-LNtoL} and Lemma \ref{lem core} is the following result.

\begin{cor}\label{cor-LNtoL} 
    In the setting of Proposition \ref{prop-LNtoL}, suppose moreover, that for each $N\in \mathbb{N}$,
    $\mathscr{C}_N$
    is a core for $\mathsf{L}_N$. Then, $\mathsf{L}_{\infty}$ is approximated by  $\mathsf{L}_N$ in the sense of Definition \ref{defn-GenApprox}.
\end{cor}

The two main ingredients to prove Proposition \ref{prop-LNtoL} is
Theorem \ref{thm-UnifApp}, in the form of Corollary \ref{cor-UnifApp}, and the intertwining relation (\ref{P-intertwining}). Before giving the proof we obtain as an immediate consequence the convergence of the corresponding Markov dynamics on path space.

\begin{thm}\label{thm-ConvMarkovProc}
In the setting described above, let $\mathcal{E}_N=\mathbb{W}_N$ and 
$\left(\mathbf{X}^{(N)}_t; t\geq 0\right)$ be the process on $\Omega$ corresponding to the Feller process $\left(\mathfrak{x}^{(N)}(t); t\geq 0\right)$, having continuous trajectories, with transition semigroup $\left(\mathsf{P}_N(t)\right)_{t\ge 0}$,
under the embedding \eqref{embedding} of $\mathbb{W}_N$ into $\Omega$. Suppose moreover that, for each $N \in \mathbb{N}$, $\mathscr{C}_N=C_{c,\textnormal{sym}}^{\infty}(\mathbb{W}_N)$ 
is a core, and invariant under $\mathsf{L}_N$, the generator of $\left(\mathsf{P}_N(t)\right)_{t\ge 0}$. Then, if $\mathbf{X}^{(N)}_0\overset{\textnormal{d}}{\underset{N\to\infty}{\longrightarrow}}\mathbf{X}^{\Omega}_0$, we have that,
\begin{equation*}
    \mathbf{X}^{(N)}\overset{\textnormal{d}}{\underset{N\to\infty}{\longrightarrow}}\mathbf{X}^{\Omega}, \ \ \textnormal{in } C(\mathbb{R}_+,\Omega).
\end{equation*}
   \begin{proof}
        Using Corollary \ref{cor-LNtoL} (recall that the implicit maps $\pi_N$ therein are induced by the embeddings \eqref{embedding} of $\mathbb{W}_N$ into $\Omega$), it follows from \cite{Ethier-Kurtz}[Chapter 4, Theorem 2.11] that $\mathbf{X}^{(N)}\longrightarrow\mathbf{X}^{\Omega}$, in distribution, in $D(\mathbb{R}_+,\Omega)$
        the space of càdlàg functions on $\mathbb{R}_+$ with values in $\Omega$ endowed with the Skorokhod topology, see \cite{Ethier-Kurtz}. 
       The desired result is then concluded from the well-known fact that the space $C(\mathbb{R}_+,\Omega)$ is closed in $D(\mathbb{R}_+,\Omega)$ in the Skorokhod topology and recalling that, for each $N\in \mathbb{N}$, $\mathbf{X}^{(N)} \in C(\mathbb{R}_+,\Omega)$ .
   \end{proof}
\end{thm}   

An analogous statement holds when the processes live in $\left(\mathbb{W}_{N,+}\right)_{N=1}^{\infty}$. The proof is the same.

 \begin{thm}\label{thm-ConvMarkovProc+}
In the setting described above, let $\mathcal{E}_N=\mathbb{W}_{N,+}$ and 
$\left(\mathbf{X}^{(N)}_t; t\geq 0\right)$ be the process on $\Omega_+$ corresponding to the Feller process $\left(\mathfrak{x}^{(N)}(t); t\geq 0\right)$, having continuous trajectories, with semigroup $\left(\mathsf{P}_N(t)\right)_{t\ge 0}$
under the embedding (\ref{embedding+}) of $\mathbb{W}_{N,+}$ into $\Omega_+$. Suppose moreover that, for each $N \in \mathbb{N}$, $\mathscr{C}_N=C_{c,\textnormal{sym}}^{\infty}(\mathbb{W}_{N,+})$ 
is a core, and invariant under $\mathsf{L}_N$, the generator of $\left(\mathsf{P}_N(t)\right)_{t\ge 0}$. Then, if $\mathbf{X}^{(N)}_0\overset{\textnormal{d}}{\underset{N\to\infty}{\longrightarrow}}\mathbf{X}^{\Omega_+}_0$, we have that,
\begin{align}
    \mathbf{X}^{(N)}\overset{\textnormal{d}}{\underset{N\to\infty}{\longrightarrow}}\mathbf{X}^{\Omega_+}, \ \ \textnormal{in } C(\mathbb{R}_+,\Omega_+).
\end{align}
  \end{thm}

  Let us write $\mathbf{X}_{\cdot}^{\Omega_+}=\left((\mathsf{x}_i(\cdot))_{i=1}^\infty,\boldsymbol{\gamma}(\cdot)\right)$ below for the process on $\Omega_+$ in coordinates and similarly $\mathbf{X}_{\cdot}^{(N)}=\left((\mathsf{x}_i^{(N)}(\cdot))_{i=1}^\infty,\boldsymbol{\gamma}^{(N)}(\cdot)\right)$ for the embedded process (recall $\boldsymbol{\gamma}^{(N)}=\sum \mathsf{x}_i^{(N)}$).
  
  \begin{rmk}
Observe that, Theorem \ref{thm-ConvMarkovProc+}, in particular implies that
\begin{align}\label{omega(XN)ToX}
   \mathsf{x}_i^{(N)}
    \overset{\textnormal{d}}{\underset{N\to\infty}{\longrightarrow}}
    \mathsf{x}_i,\ \ \textnormal{in } C(\mathbb{R}_+,\mathbb{R}_+), \ \ \forall i \in \mathbb{N}, \ 
   \textnormal{ and } \
    \sum_{i=1}^{N}\mathsf{x}_i^{(N)}\overset{\textnormal{d}}{\underset{N\to\infty}{\longrightarrow}} \boldsymbol{\gamma},\ \ \textnormal{in } C(\mathbb{R}_+,\mathbb{R}_+).
\end{align}
  \end{rmk}

  \begin{prop}\label{Prop-Conv in l^2}
    In the setting of Theorem \ref{thm-ConvMarkovProc+}, there exists a coupling of the $\mathbf{X}^{(N)}$ and $\mathbf{X}^{\Omega_+}$ on a single probability space such that, almost surely, for any $T \ge 0$,
    \begin{align}\label{l^2 convergence}
\sup_{t\in[0,T]}\sum_{i=1}^{\infty}\left(\mathsf{x}_i^{(N)}(t)-\mathsf{x}_i(t)\right)^2 \overset{N\to\infty}{\longrightarrow} 0.
    \end{align}
\end{prop}

    \begin{proof} 
    We use the Skorokhod representation theorem to couple $\mathbf{X}^{(N)}$ and $\mathbf{X}^{\Omega_+}$ on a single probability space, and, abusing notation, denote the new processes by the same symbol. One can then write for all $M\in\mathbb{N}$,
    \begin{align*}
\sup_{t\in[0,T]}\sum_{i=1}^{\infty}\left(\mathsf{x}_i^{(N)}(t)-\mathsf{x}_i(t)\right)^2\leq &\, 
2\sup_{t\in[0,T]}\sum_{i=M+1}^{\infty}\left(\mathsf{x}_i^{(N)}(t)\right)^2+2\sup_{t\in[0,T]}\sum_{i=M+1}^{\infty}\left(\mathsf{x}_i(t)\right)^2\\&+\sup_{t\in[0,T]}\sum_{i=1}^{M}\left(\mathsf{x}_i^{(N)}(t)-\mathsf{x}_i(t)\right)^2. 
    \end{align*}
    Let $\varepsilon>0$ be given. Observe first that as $\sup_{t\in[0,T]}\sum_{i=1}^{\infty}\left(\mathsf{x}_i(t)\right)^2<\infty$, there exists $M_0\in\mathbb{N}$, such that for all $M\geq M_0$, 
    \begin{align*}
         2\sup_{t\in[0,T]}\sum_{i=M+1}^{\infty}
        \left(\mathsf{x}_i(t)\right)^2\leq\frac{\varepsilon}{4}.
    \end{align*}
    Then, using Theorem \ref{thm-ConvMarkovProc+} we have 
    \begin{align*}
\sup_{t\in[0,T]}\sum_{i=1}^{M_0}\left(\mathsf{x}_i^{(N)}(t)-\mathsf{x}_i(t)\right)^2\leq\frac{\varepsilon}{4},
    \end{align*}
    for large enough $N$. 
    Note moreover that,
    \begin{align*}
\sup_{t\in[0,T]}\sum_{i=M_0+1}^{\infty}\left(\mathsf{x}_i^{(N)}(t)\right)^2\leq \sup_{t\in[0,T]}\mathsf{x}_{M_0}^{(N)}(t)\sup_{t\in[0,T]}\sum_{i=M_0}^{\infty}\mathsf{x}_i^{(N)}(t).
    \end{align*}
Now, from \eqref{omega(XN)ToX}, we have 
 \begin{align*}
       \lim_{N\to\infty}\sup_{t\in[0,T]}\mathsf{x}_{M_0}^{(N)}(t)=\sup_{t\in[0,T]}\mathsf{x}_{M_0}(t),\ 
 \ \
       \sup_{N\in\mathbb{N}}\sup_{t\in[0,T]}\sum_{i=M_0}^{\infty}\mathsf{x}_i^{(N)}(t)\overset{\textnormal{def}}{=}c<\infty.
    \end{align*}
    Hence, one can pick a possibly larger $M_0$ such that $\sup_{t\in[0,T]}\mathsf{x}_{M_0}(t)<\frac{\varepsilon}{4c}$ also holds, and thus,
    \begin{align*}
    \limsup_{N \ge 1}\sup_{t\in[0,T]}\sum_{i=M_0+1}^{\infty}
    \left(\mathsf{x}_i^{(N)}(t)\right)^2\leq\frac{\varepsilon}{4}.
    \end{align*}
    Putting everything together, we conclude \eqref{l^2 convergence}.
    \end{proof}

   When considering processes on $\mathbb{W}_N$ and $\Omega$ instead, a very similar argument can be used to show the following convergence in $\ell^3$ as a consequence of Theorem \ref{thm-ConvMarkovProc}. Below we write $\mathbf{X}_{\cdot}^{\Omega}=\left(\left(\mathsf{x}^+_i(\cdot)\right)_{i=1}^\infty,\left(\mathsf{x}^-_i(\cdot)\right)_{i=1}^\infty,\boldsymbol{\gamma}(\cdot),\boldsymbol{\delta}(\cdot)\right)$  for the process on $\Omega$ in coordinates and similarly $\mathbf{X}_{\cdot}^{(N)}=\left(\left(\mathsf{x}^{(N),+}_i(\cdot)\right)_{i=1}^\infty,\left(\mathsf{x}^{(N),-}_i(\cdot)\right)_{i=1}^\infty,\boldsymbol{\gamma}^{(N)}(\cdot),\boldsymbol{\delta}^{(N)}(\cdot)\right)$ for the embedded process.

  \begin{prop}\label{Prop-Conv in l^3} 
    In the setting of Theorem \ref{thm-ConvMarkovProc},  there exists a coupling of the $\mathbf{X}^{(N)}$ and $\mathbf{X}^{\Omega}$ on a single probability space such that, almost surely, for any $T \ge 0$,
    \begin{align}\label{l^3 convergence} 
\sup_{t\in[0,T]}\left(\sum_{i=1}^{\infty}\left|\mathsf{x}_i^{(N),+}(t)-\mathsf{x}_i^+(t)\right|^3+\sum_{i=1}^{\infty}\left|\mathsf{x}_i^{(N),-}(t)-\mathsf{x}_i^-(t)\right|^3\right)\overset{N\to\infty}{\longrightarrow} 0.
    \end{align}
\end{prop}

We now prove Proposition \ref{prop-LNtoL}.

   \begin{proof}[Proof of Proposition \ref{prop-LNtoL}]
   We prove the proposition in the case $\left(\mathcal{E}_N\right)_{N=1}^{\infty}=\left(\mathbb{W}_N\right)_{N=1}^{\infty}$. The exact same argument works for $\left(\mathcal{E}_N\right)_{N=1}^{\infty}=\left(\mathbb{W}_{N,+}\right)_{N=1}^{\infty}$.

We then take the Banach spaces $\left(\mathscr{L}_N\right)_{N=1}^\infty=\left(C_0(\mathbb{W}_N)\right)_{N=1}^\infty$ and $\mathscr{L}=C_0(\Omega)$ in the Definition \ref{Def fn to f}.
    Considering the definition of $\mathscr{C}_\infty$ from Lemma \ref{lem core},  it only suffices to prove the statement for functions of the form $\mathbf{f}=\Lambda_K^{\infty}g$  with some $K\in\mathbb{N}$ and $g\in\mathscr{C}_K$.
       We thus fix a positive integer $K$, and take an arbitrary function $g\in\mathscr{C}_K$. For any $N>K$, let 
       $\mathbf{f}_N\overset{\textnormal{def}}{=}\Lambda_K^{N}g$ and $\mathbf{f}\overset{\textnormal{def}}{=}\Lambda_K^{\infty}g$. Observe that, by the Feller property of the kernels, 
       $\mathbf{f}_N\in C_0(\mathbb{W}_N)$ and $\mathbf{f}\in C_0(\Omega)$, respectively. We want to show that 
       \begin{align}\label{fN to f}
           \mathbf{f}_N\to \mathbf{f},\  \  \text{ as } N\to\infty,
       \end{align}
        in the sense of Definition \ref {Def fn to f}. For each $N\in\mathbb{N}$, we then define the operator $\pi_N:C_0(\Omega)\to C_0(\mathbb{W}_N)$ as follows
       \begin{align}
           (\pi_Nf)\left(\mathbf{x}^{(N)}\right)=f\left(\omega\left(\mathbf{x}^{(N)}\right)\right),\  \  \mathbf{x}^{(N)}\in\mathbb{W}_N,
       \end{align}
       where by $\omega\left(\mathbf{x}^{(N)}\right)$ we mean the embedding given in \eqref{embedding}. It is easy to see  that the operator norm $\||\pi_N\||_N\le 1$, uniformly in $N$, and so the conditions given in Definition \ref{Def fn to f} are satisfied. We thus need to check that 
       \begin{align*}
        \big\|\pi_N\mathbf{f}-\mathbf{f}_N\big\|_N=\sup_{\mathbf{x}^{(N)}\in \mathbb{W}_N}\left|(\pi_N\mathbf{f})\left(\mathbf{x}^{(N)}\right)-\mathbf{f}_N\left(\mathbf{x}^{(N)}\right)\right|\overset{N\to\infty}{\longrightarrow} 0.
       \end{align*}
       Note that this is exactly the expression in \eqref{UnifApp} from Corollary \ref{cor-UnifApp}, and so \eqref{fN to f} is verified. We next want to show that
              \begin{align*}
           \mathsf{L}_N\mathbf{f}_N\to \mathsf{L}_{\infty}\mathbf{f},\  \  \text{ as } N\to\infty,
       \end{align*}
       also holds, that is, 
            \begin{align}\label{LNfN to Lf}
         \big\|\pi_N(\mathsf{L}_{\infty}\mathbf{f})-\mathsf{L}_N\mathbf{f}_N\big\|_N\to 0,\  \  \text{ as } N\to\infty.
       \end{align}
       
    First, it follows from the intertwining relations \eqref{P^N-intertwining} and \eqref{P-intertwining} between the semigroups that for all $f\in\mathcal{D}(\mathsf{L}_K)$, 
    \begin{align*}
    \mathsf{L}_{N}\Lambda_K^{N}f&=\Lambda_K^{N}\mathsf{L}_Kf,\  \  \forall N>K,\\
\mathsf{L}_{\infty}\Lambda_K^{\infty}f&=\Lambda_K^{\infty}\mathsf{L}_Kf.
    \end{align*}
    In particular, one can check that
    $\mathbf{f}_N\in \mathcal{D}(\mathsf{L}_{N})$. We can then write
    \begin{align*}
         \big\|\pi_N(\mathsf{L}_\infty\mathbf{f})-\mathsf{L}_N\mathbf{f}_N\big\|_N=
         \big\|\pi_N\left(\mathsf{L}_\infty\Lambda_K^{\infty}g\right)-\mathsf{L}_{N}\Lambda_K^{N}g\big\|_N
         =\big\|\pi_N\left(\Lambda_K^{\infty}\mathsf{L}_Kg\right)-\Lambda_K^{N}\mathsf{L}_Kg\big\|_N.
       \end{align*}
       Note moreover that, by assumption, the space $\mathscr{C}_K$ is invariant under the action of $\mathsf{L}_K$  and thus $\mathsf{L}_Kg\in\mathscr{C}_K$. Hence, \eqref{LNfN to Lf} follows exactly in the same way as \eqref{fN to f} from Corollary \ref{cor-UnifApp}. 
   \end{proof}
\subsection{Proof of the uniform approximation theorem}\label{SubSec-ProofUnifAppThm}

We only consider the case $\left(\mathcal{E}_N\right)_{N=1}^{\infty}=\left(\mathbb{W}_N\right)_{N=1}^{\infty}$ throughout this section. It is not hard to see, by inspecting \eqref{UniformApproximation} and the relevant definitions carefully, that in the case $\left(\mathcal{E}_N\right)_{N=1}^{\infty}=\left(\mathbb{W}_{N,+}\right)_{N=1}^{\infty}$ we are simply taking the supremum over $\mathbb{W}_{N,+}$ instead of $\mathbb{W}_N$ while the quantity we are taking the supremum of is exactly the same as for $\mathbb{W}_N$. Thus, the result for $\mathbb{W}_{N,+}$ is a special case of the one over $\mathbb{W}_{N}$. We need some preliminaries.

\subsubsection{Explicit formulae of Markov kernels}
We first present an explicit formula for $\Lambda_K^{N}\left(\mathbf{x}^{(N)},\mathrm{d}\mathbf{y}\right)$ which is the starting point in our analysis. For any $N\geq 2$ and $\mathbf{x}^{(N)}\in\mathbb{W}_N^{\circ}$, we consider the fundamental spline function $\mathcal{M}\left(y;\mathbf{x}^{(N)}\right):\mathbb{R}\to\mathbb{R}$, with knots $\mathbf{x}^{(N)}$, given by 
\begin{align}\label{M}
        \mathcal{M}\left(y;\mathbf{x}^{(N)}\right)= (N-1)\sum_{i=1}^N
        \frac{\left(x_i^{(N)}-y\right)_+^{N-2}}{\prod_{j=1,j\neq i}^N\left(x_i^{(N)}-x_j^{(N)}\right)},
    \end{align} 
   where, for $z\in \mathbb{R}$, $(z)_+=\max\{z,0\}$. One can check that, for all $\mathbf{x}^{(N)}\in\mathbb{W}_N^{\circ}$, $y \mapsto \mathcal{M}\left(y;\mathbf{x}^{(N)}\right)$ is a piecewise polynomial function 
   with $N-2$ continuous derivatives. Moreover, it satisfies
   \begin{align*}
\int_\mathbb{R}\mathcal{M}\left(y;\mathbf{x}^{(N)}\right)\mathrm{d}y=1.
   \end{align*}
The next theorem, proven by Olshanski in \cite{ProjOrbitalMeas}, gives an expression for the Markov kernel $\Lambda_{K}^{N}$ defined in \eqref{Lambda_K^N} in terms of the spline function introduced above.

\begin{thm}\label{thm-Lambda_K^N-detM}
    Let $N\geq 2$ and $\mathbf{x}^{(N)}\in\mathbb{W}_{N}^{\circ}$. Then, for any $1\leq K\leq N-1$, we have
\begin{align}\label{Lambda_K^N formula}
\Lambda_K^{N}\left(\mathbf{x}^{(N)},\mathrm{d}\mathbf{y}\right)&=
\prod_{l=1}^{K-1}\binom{N-K+l}{l}
\frac{\det {\begin{pmatrix}\mathcal{M}\left(y_{K-j+1};x_{N-i+1}^{(N)},\dots,x_{K-i+1}^{(N)}\right)\end{pmatrix}}_{i,j=1}^{K}}{\prod_{j-i\geq N-K+1}\left(x_{i}^{(N)}-x_{j}^{(N)}\right)}\Delta_K\left(\mathbf{y}\right)\mathrm{d}\mathbf{y},\  \ \forall\mathbf{y}\in\mathbb{W}_K.
\end{align}
\end{thm}

We note that, the case $K=1$, boils down to the following identity
\begin{align}\label{Lambda_1=M}
\Lambda_1^{N}\left(\mathbf{x}^{(N)},\mathrm{d}y\right)=\mathcal{M}\left(y;\mathbf{x}^{(N)}\right)\mathrm{d}y,\  \ \forall \mathbf{x}^{(N)}\in\mathbb{W}_N^{\circ},
\end{align}
which is an observation first due to Okounkov \cite{Olshanski-Vershik}.

The formula from Theorem \ref{thm-Lambda_K^N-detM} is not yet amenable to asymptotic analysis, partly due to the denominator, which may vanish and lead to singularities. Thankfully, we can use a trick based on a formula for derivatives of spline functions to get rid of the denominator entirely. The new formula for $\Lambda_K^N$ will however involve derivatives of the spline function $\mathcal{M}$ which when we extend in the sequel the formula to $\mathbf{x}^{(N)}$ with coinciding coordinates will need to be interpreted in a weak sense.

\begin{thm}\label{thm2-Lambda_K^N-detM}
In the setting of Theorem \ref{thm-Lambda_K^N-detM}, we have
    \begin{align}\label{Lambda_K^N-M}
\Lambda_K^{N}\left(\mathbf{x}^{(N)},\mathrm{d}\mathbf{y}\right)=
\left(\prod_{l=1}^{K-1}\frac{1}{l!}\right)
\det \left(\mathcal{M}^{(K-i)}\left(\mathrm{d}y_{K-j+1};x_{1}^{(N)},\dots,x_{N-i+1}^{(N)}\right)\right)_{i,j=1}^{K}\Delta_K\left(\mathbf{y}\right).
\end{align}
Here, for any $j=1,\dots,K-1$,  $\mathcal{M}^{(j)}$ denotes the $j$-th derivative of the function $\mathcal{M}$ with the convention $\mathcal{M}^{(0)}=\mathcal{M}$, 
and, $\mathcal{M}^{(j)}(\mathrm{d}y,\cdot)=\mathcal{M}^{(j)}(y,\cdot)\mathrm{d}y$
\end{thm}
\begin{proof}
    We have following relation for the derivative of the spline function in \eqref{M}, obtained by direct computation,
\begin{align}\label{M'} 
\mathcal{M}^{(1)}\left(y;\mathbf{x}^{(N)}\right)=\frac{N-1}{x_1^{(N)}-x_N^{(N)}}\left(\mathcal{M}\left(y;x_2^{(N)},\dots,x_{N}^{(N)}\right)-\mathcal{M}\left(y;x_{1}^{(N)},\dots,x_{N-1}^{(N)}\right)\right).
\end{align}
We now consider the formula given in Theorem \eqref{thm-Lambda_K^N-detM} and repeatedly apply elementary row operations to the determinant in \eqref{Lambda_K^N formula} and use \eqref{M'} to obtain \eqref{Lambda_K^N-M}.
\end{proof}

Let 
$g\in  C_c^\infty(\mathbb{R}^K)$
and consider
$\Lambda_K^{N}g\left(\mathbf{x}^{(N)}\right)$ for $\mathbf{x}^{(N)}\in\mathbb{W}_N^\circ$. We have
from \eqref{Lambda_K^N-M} that
\begin{align*}
\Lambda_K^{N}g\left(\mathbf{x}^{(N)}\right)
&=\int_{\mathbb{R}^K}\Lambda_K^{N}\left(\mathbf{x}^{(N)},\mathrm{d}\mathbf{y}\right)g\left(\mathbf{y}\right)
\\&=
\int_{\mathbb{R}^K}
\left(\prod_{l=1}^{K-1}\frac{1}{l!}\right)
\det \left(\mathcal{M}^{(K-i)}\left(\mathrm{d}y_{K-j+1};x_{1}^{(N)},\dots,x_{N-i+1}^{(N)}\right)\right)_{i,j=1}^{K}\Delta_K(\mathbf{y})g(\mathbf{y}).
\end{align*}
 Next, one can expand the determinant and use integration by parts for each term in the expansion to rewrite the above expression as 
\begin{align}\label{WeakExpansion}
\Lambda_K^{N}g\left(\mathbf{x}^{(N)}\right)&=(-1)^{\frac{K(K-1)}{2}}\int_{\mathbb{R}^K}\sum_{\sigma\in\mathbb{S}_K}
\textnormal{sgn}(\sigma)\prod_{j=1}^{K}\Lambda_1^{N-K+j}\left(\mathrm{d}y_{\sigma(j)};x_{1}^{(N)},\dots,x_{N-K+j}^{(N)}\right)\nonumber\\
&\hspace{8cm}\times
\partial_{\mathbf{y}}^{\sigma}\left(\Delta_K(\mathbf{y})g(\mathbf{y})\right),
\end{align}
where 
we used \eqref{Lambda_1=M} and for $\sigma$ in $\mathbb{S}_K$, the symmetric group of degree $K$, the differential operator $\partial^{\sigma}$ acts on functions $h\in C_c^\infty(\mathbb{R}^K)$ as follows:
\begin{align}\label{DifferentialOperator}
\partial_{\mathbf{y}}^{\sigma}h(\mathbf{y})=\frac{\partial^{\frac{K(K-1)}{2}}}{\partial y_1^{\sigma(1)-1}\partial y_2^{\sigma(2)-1}\cdots\partial y_K^{\sigma(K)-1}}h(\mathbf{y}).
\end{align}
 Now suppose $\mathbf{x}\in\mathbb{W}_N$ is arbitrary. Take a sequence $\left(\mathbf{x}_{n}\right)_{n=1}^{\infty}$ in $\mathbb{W}_N^\circ$ such that $\mathbf{x}_{n}\overset{n\to\infty}{\longrightarrow}\mathbf{x}$. It then follows from Lemma \ref{lem-FellerPropLambdaN} that both the left and right hand side of equality (\ref{WeakExpansion}) above, when evaluating at $\mathbf{x}_{n}$,  converge to the corresponding terms for $\mathbf{x}$, as $n\to\infty$. This gives an explicit expression of $\Lambda_K^{N}g\left(\mathbf{x}\right)$ for all $g\in C_c^\infty(\mathbb{R}^K)$ and general $\mathbf{x}\in\mathbb{W}_N$.
Note that, this expression matches the one we obtain by formally expanding the following determinant
 \begin{align*}
\int_{\mathbb{R}^K}
\left(\prod_{l=1}^{K-1}\frac{1}{l!}\right)
\det \left(\mathcal{M}^{(K-i)}\left(\mathrm{d}y_{K-j+1};x_{1},\dots,x_{N-i+1}\right)\right)_{i,j=1}^{K}\Delta_K\left(\mathbf{y}\right)g\left(\mathbf{y}\right), 
\end{align*}
 where, the derivatives correspond to the distributional derivative
 of the kernel $\mathcal{M}(\mathrm{d}y,\cdot)$. In particular, it is consistent with the formal expansion of the determinant along rows and columns.  From now on, all appearances of these distributional determinants will be understood to be interpreted in the weak sense explained above when applied to $C_c^{\infty}(\mathbb{R}^K)$ functions (and the functions we need to apply them to will always be in this space).

We next present an analogous explicit formula for the kernel $\Lambda_K^\infty$. We do this in the case of $\Omega$. The case of $\Omega_+$ is in fact a special case under the embedding \eqref{omega_+ToOmega}. So, suppose $\omega\in \Omega$ is arbitrary and written in coordinates as $\omega=\left(\mathbf{x}^+,\mathbf{x}^-,\gamma,\delta\right)$. We will make very frequent use of the following notation in everything that follows.

\begin{defn}
 For $\omega \in \Omega$, or, abusing notation, for $\omega \in \Omega_+$, we define the probability measure $\phi_\omega(\mathrm{d}y)$ on $\mathbb{R}$ by,
\begin{equation}
 \phi_{\omega}(\mathrm{d}y)\overset{\textnormal{def}}{=}\Lambda_1^\infty(\omega,\mathrm{d}y).
\end{equation}
\end{defn}

Given arbitrary $\omega \in \Omega$ as above, we define, for $\epsilon>0$, $\omega_\epsilon \in \Omega$ by:
\begin{equation*}
\omega_\epsilon=\left(\mathbf{x}^+,\mathbf{x}^-,\gamma,\delta+\epsilon\right).
\end{equation*}
We observe that, since $\phi_{\omega_\epsilon}$ involves a Gaussian component, it has a smooth density with respect to Lebesgue measure in $\mathbb{R}$, see \cite{Olshanski-Vershik}, which we denote by, abusing notation, $\phi_{\omega_\epsilon}(y)$. We write $\phi_{\omega_\epsilon}^{(\ell)}(\mathrm{d}y)=\phi_{\omega_\epsilon}^{(\ell)}(y)\mathrm{d}y$ for its $\ell$-th derivative. We have the following result.

\begin{prop}
Let $\epsilon>0$. Then, the Markov kernel $\Lambda_K^\infty(\omega_\epsilon,\mathrm{d}\mathbf{y})$ has an explicit density with respect to Lebesgue measure in $\mathbb{W}_K$ and is given by,
\begin{align}\label{Lambda^infty}         \Lambda_K^{\infty}\left(\omega_{\epsilon},\mathrm{d}\mathbf{y}\right)=
          \left(\prod_{l=1}^{K-1}\frac{1}{l!}\right)\det \left(\phi_{\omega_\epsilon}^{(j-1)}(\mathrm{d}y_{i})\right)_{i,j=1}^{K} \Delta_K(\mathbf{y})\mathrm{d}\mathbf{y}.
    \end{align} 
\end{prop}

\begin{proof}
This is implicit in the proof of Theorem 7.7 in \cite{Olshanski-Vershik}.
\end{proof}

Recall (or simply observe) that, as $\epsilon \to 0$, we have the weak convergence
\begin{equation}\label{ConvEpsilonTo0}
\Lambda^\infty_K(\omega_\epsilon,\mathrm{d}\mathbf{y})\longrightarrow \Lambda^\infty_K(\omega,\mathrm{d}\mathbf{y}).
\end{equation}
Now, given $g\in C_c^\infty(\mathbb{R}^K)$, we expand the determinant in \eqref{Lambda^infty} and use integration  by parts as before and then take $\epsilon \to 0$ by virtue of \eqref{ConvEpsilonTo0} to make sense of the potentially distributional determinant, for general $\omega \in \Omega$,
\begin{equation*}
        \left(\prod_{l=1}^{K-1}\frac{1}{l!}\right)\det \left(\phi_{\omega}^{(j-1)}(\mathrm{d}y_{i})\right)_{i,j=1}^{K} \Delta_K(\mathbf{y}),
\end{equation*}
as follows, when acting on $g\in C_c^\infty(\mathbb{R}^K)$,
 
\begin{align*}
\int_{\mathbb{R}^K}\Lambda_K^{\infty}\left(\omega,\mathrm{d}\mathbf{y}\right)g(\mathbf{y})
&=(-1)^{\frac{K(K-1)}{2}}
\int_{\mathbb{R}^K}\prod_{j=1}^{K}\phi_\omega(\mathrm{d}y_{j})\sum_{\sigma\in\mathbb{S}_K}\textnormal{sgn}(\sigma)\partial_{\mathbf{y}}^{\sigma}\left(\Delta_K(\mathbf{y})g(\mathbf{y})\right),
\end{align*}
where the differential operator is the same as in \eqref{DifferentialOperator}.
 
As before, this weak-sense definition is consistent with the formal expansion of the determinant along rows and columns.  From now on, all appearances of these distributional determinants will be understood to be interpreted in the weak sense explained above when applied to $C_c^{\infty}(\mathbb{R}^K)$ functions (all functions we need to apply them to will always be in this space).

\subsubsection{Proof of Theorem \ref{thm-UnifApp}}
 For any fixed $K\in\mathbb{N}$, $N>K$, and $\mathbf{x}^{(N)}\in\mathbb{W}_N$, we consider the following variant of $\Lambda_K^{\infty}$,
 \begin{align}\label{tildeLambda}
\tilde{\Lambda}_K^{\infty}\left(\omega\left(\mathbf{x}^{(N)}\right), \mathrm{d}\mathbf{y}\right)&=
\left(\prod_{l=1}^{K-1}\frac{1}{l!}\right)
\det \left(\phi_{\omega\left(x_{1}^{(N)},\dots,x_{N-K+j}^{(N)}\right)}^{(j-1)}(\mathrm{d}y_{i})\right)_{i,j=1}^{K} \Delta_K(\mathbf{y}).
 \end{align}
We note that we only need (and in fact only define) $\tilde{\Lambda}_K^{\infty}$ for $\omega\in \Omega$ of the form $\omega\left(\mathbf{x}^{(N)}\right)$ for $\mathbf{x}^{(N)}\in \mathbb{W}_N$. Again the determinant in equation \eqref{tildeLambda} above is interpreted in the weak sense when tested against smooth functions of compact support as explained previously. We first prove the desired result with
$\Lambda_K^{\infty}$ replaced by
$\tilde{\Lambda}_K^{\infty}$ in \eqref{UniformApproximation}.

\begin{thm}\label{thm-tildeUnifApp}
In the setting of Theorem \ref{thm-UnifApp}, we have
    \begin{align*}
\sup_{\mathbf{x}^{(N)}\in \mathbb{W}_N}\left|\tilde{\Lambda}_K^{\infty}g\left(\omega\left(\mathbf{x}^{(N)}\right)\right)-\Lambda_K^{N}g\left(\mathbf{x}^{(N)}\right)\right|\to 0,\  \ 
 \textnormal{as } N\to\infty. 
 \end{align*} 
\end{thm}
After establishing Theorem \ref{thm-tildeUnifApp}, we prove that $\tilde{\Lambda}_K^{\infty}g\left(\omega\left(\mathbf{x}^{(N)}\right)\right)$  and $\Lambda_K^{\infty}g\left(\omega\left(\mathbf{x}^{(N)}\right)\right)$ are asymptotically uniformly close as $N\to\infty$, see Proposition \ref{prop-tildeLambda-Lambda} for the exact statement, which concludes the desired result.
The proof of Theorem \ref{thm-tildeUnifApp} is by an inductive argument in $K$. The variant $\tilde{\Lambda}_K^{\infty}$ of $\Lambda_K^\infty$ is very well-adapted to proving the rather complicated inductive step and this is its raison d'\^{e}tre.  Before giving the proof, we need  some more preliminaries and notations.

Denote the Fourier and inverse Fourier transforms of an integrable function $f$ on $\mathbb{R}^n$ by $\hat{f}$ and $\check{f}$ respectively given by:
\begin{align*}
    \hat{f}(\boldsymbol{\xi})=\int_{\mathbb{R}^n}\mathrm{e}^{\textnormal{i}\sum_{j=1}^ny_j\xi_j}f(\mathbf{y})\mathrm{d}\mathbf{y},\  \
    \check{f}(\boldsymbol{\xi})= (2\pi)^{-n}
    \int_{\mathbb{R}^n}\mathrm{e}^{-\textnormal{i}\sum_{j=1}^ny_j\xi_j}f(\mathbf{y})\mathrm{d}\mathbf{y},\  \ \boldsymbol{\xi}\in\mathbb{R}^n.
\end{align*}

\begin{defn} We define the following functions
 \begin{align}
        F_+(z;\omega)&\overset{\textnormal{def}}{=}\mathrm{e}^{\textnormal{i}\gamma z}\prod_{j=1}^{\infty}\frac{\mathrm{e}^{-\textnormal{i}x_jz}}{1-\textnormal{i}x_jz},\  \ \forall z\in\mathbb{R},\ 
 \ \omega\in\Omega_+,
        \\
        \label{F_omega}
        F(z;\omega)&\overset{\textnormal{def}}{=}\mathrm{e}^{\textnormal{i}\gamma z-\frac{\delta-\sum_{i=1}^\infty \left[(x_i^+)^2+(x_i^-)^2\right]}{2}z^2}\prod_{j=1}^{\infty}\frac{\mathrm{e}^{-\textnormal{i}x_j^+z}}{1-\textnormal{i}x_j^+z}
        \frac{\mathrm{e}^{\textnormal{i}x_j^-z}}{1+\textnormal{i}x_j^-z},\  \ \forall z\in\mathbb{R},\ 
 \ \omega\in\Omega. 
   \end{align}
Note that, the infinite products are well-defined by virtue of $\omega\in \Omega_+$ and $\omega \in \Omega$ respectively.
\end{defn}

It is easy to check that $F_+(\cdot;\omega)$ and $F(\cdot;\omega)$ are 
the characteristic functions (i.e. Fourier transforms) of the probability law $\phi_{\omega}$ when $\omega \in \Omega_+$ and $\Omega$, respectively. In particular, we have
\begin{equation}\label{FT-pairing}
\int_{\mathbb{R}} F(z;\omega)g(z)\mathrm{d}z=
     \int_{\mathbb{R}} \phi_{\omega}(\mathrm{d}z)\hat{g}(z),\  \ \forall g\in\mathcal{S}\left(\mathbb{R}\right),
\end{equation}
where $\mathcal{S}\left(\mathbb{R}\right)$ denotes the space of Schwartz functions on $\mathbb{R}$.  We will also need the following formula in the argument.
\begin{lem}
    Let $\mathbf{x}^{(N)}\in \mathbb{W}_N$. Then, we have
    \begin{align}\label{F-M}
    \prod_{j=1}^{N}\left(1-\frac{\textnormal{i}zx_j^{(N)}}{N}\right)^{-1}
    = \int_{-\infty}^{\infty}
    \left(1-\frac{\textnormal{i}zy}{N}\right)^{-N}\mathcal{M}\left(\mathrm{d}y;\mathbf{x}^{(N)}\right).
    \end{align}
\end{lem}
    \begin{proof}
    See Section 8 of \cite{Olshanski-Vershik}, stated there for $\mathbf{x}^{(N)}\in \mathbb{W}_N^{\circ}$ which can be extended to coinciding coordinates by continuity.
    \end{proof}

\begin{proof}[Proof of Theorem \ref{thm-tildeUnifApp} (the base case)]
    Let $K=1$ and fix $g\in C_c^\infty(\mathbb{R})$. Observe that $\tilde{\Lambda}_1^{\infty}=\Lambda_1^{\infty}$. 
    We have by definition that
    \begin{align*}
\sup_{\mathbf{x}^{(N)}\in \mathbb{W}_N}
\left|\tilde{\Lambda}_1^{\infty}g\left(\omega\left(\mathbf{x}^{(N)}\right)\right)-\Lambda_1^{N}g\left(\mathbf{x}^{(N)}\right)\right|&= \sup_{\mathbf{x}^{(N)}\in \mathbb{W}_N} 
\left|\int_{-\infty}^\infty\phi_{\omega\left(\mathbf{x}^{(N)}\right)}(\mathrm{d}y)g(y)-\int_{-\infty}^\infty\mathcal{M}\left(\mathrm{d}y;\mathbf{x}^{(N)}\right)g(y)\right|.
 \end{align*}
 We have, again by definition,
 \begin{align*}
\int_{-\infty}^\infty \phi_{\omega\left(\mathbf{x}^{(N)}\right)}(\mathrm{d}y)g(y)
= \frac{1}{2\pi}\int_{-\infty}^{\infty}F\left(z;\omega\left(\mathbf{x}^{(N)}\right)\right)(z)\check{g}(z)\mathrm{d}z.
 \end{align*}
Using \eqref{F_omega} and then \eqref{F-M} we obtain
 \begin{align*}
   \int_{-\infty}^\infty\phi_{\omega\left(\mathbf{x}^{(N)}\right)}(\mathrm{d}y)g(y)&=\frac{1}{2\pi}
\int_{-\infty}^{\infty}
\prod_{j=1}^{N}\left(1-\frac{\textnormal{i}zx_j^{(N)}}{N}\right)^{-1}
\check{g}(z)\mathrm{d}z\\&= \frac{1}{2\pi}\int_{-\infty}^{\infty}
\int_{-\infty}^{\infty}\left(1-\frac{\textnormal{i}zw}{N}\right)^{-N}\mathcal{M}\left(\mathrm{d}w;\mathbf{x}^{(N)}\right)
\check{g}(z)\mathrm{d}z. 
 \end{align*}
Note now that for any fix $w\in\mathbb{R}$, $\left(1-\frac{\textnormal{i}zw}{N}\right)^{-N}$ converges pointwise to $\mathrm{e}^{\textnormal{i}wz}$. Also, since
$g\in C_c^{\infty}(\mathbb{R})$,
the function $\check{g}$ is integrable on $\mathbb{R}$. Thus, by dominated convergence theorem,
\begin{align}\label{G_N}
\mathsf{G}_N(w)&\overset{\textnormal{def}}{=}\frac{1}{2\pi}\int_{-\infty}^{\infty}\left(1-\frac{\textnormal{i}zw}{N}\right)^{-N}\check{g}(z)\mathrm{d}z-g(w)
\\&=\frac{1}{2\pi}\int_{-\infty}^{\infty}\left[\left(1-\frac{\textnormal{i}zw}{N}\right)^{-N}-\mathrm{e}^{\textnormal{i}zw}\right]\check{g}(z)\mathrm{d}z\, \rightarrow 0, \  \ \textnormal{as } \;N\rightarrow\infty.\nonumber
 \end{align}
 We show in the Lemma \ref{lem UnifConv G_N} below that the above convergence is uniform over $\mathbb{R}$ which proves the result for the case $K=1$. This is because by the uniform convergence, we have for $N$ large enough that $\sup_{w\in\mathbb{R}}|\mathsf{G}_N(w)|<\varepsilon$, for a given $\varepsilon>0$,   
 and accordingly,
    \begin{align*}
     &\sup_{\mathbf{x}^{(N)}\in \mathbb{W}_N} \left|\int_{-\infty}^\infty\phi_{\omega\left(\mathbf{x}^{(N)}\right)}(\mathrm{d}y)g(y)-\int_{-\infty}^\infty\mathcal{M}\left(\mathrm{d}y;\mathbf{x}^{(N)}\right)g(y)\right|=\sup_{\mathbf{x}^{(N)}\in \mathbb{W}_N} \left|\int_{-\infty}^{\infty}\mathsf{G}_N(w)\mathcal{M}\left(\mathrm{d}w;\mathbf{x}^{(N)}\right)\right|\\&\  \ =\sup_{w\in\mathbb{R}}\left|\mathsf{G}_N(w)\right|\sup_{\mathbf{x}^{(N)}\in \mathbb{W}_N} \left|\int_{-\infty}^{\infty}\mathcal{M}\left(\mathrm{d}w;\mathbf{x}^{(N)}\right)\right|<\varepsilon.
 \end{align*}
 Here, we used the fact that $\int_{-\infty}^{\infty}\mathcal{M}\left(\mathrm{d}w;\mathbf{x}^{(N)}\right)=1$ for every $\mathbf{x}^{(N)}\in \mathbb{W}_N$.
\end{proof}

\begin{lem}\label{lem UnifConv G_N}
  Let $\mathsf{G}_N$ be the function defined in \eqref{G_N}. Then, we have,
  \begin{equation*}
 \sup_{w\in \mathbb{R}}|\mathsf{G}_N(w)|\overset{N \to \infty}{\longrightarrow} 0.
  \end{equation*}
  \end{lem}
\begin{proof}
    Suppose by contradiction that the assertion is not true. Thus, there exist $\delta>0$ and a sequence $(w_{N_n})_{n=1}^\infty$ such that 
    \begin{align}\label{>delta}
\left|\mathsf{G}_{N_n}(w_{N_n})\right|>\delta,\  \  \forall n\in\mathbb{N}.
    \end{align}
    There are two possibilities:
    \begin{itemize}
        \item [(i)]
        There exists some positive constant $L$
        such that $|w_{N_n}|\leq L$ for all $n\in\mathbb{N}$.
        \item [(ii)]
        The sequence $(w_{N_n})_{n=1}^\infty$ is unbounded.
    \end{itemize}
    We show both of these cases lead to a contradiction.
    First we assume (i) occurs. Consider the sequence of functions 
    $(\mathsf{G}_{N_n})_{n=1}^\infty$
    on the interval $[-L,L]$. One can see that this is a sequence of uniformly bounded and 
    equicontinuous functions
    on $[-L,L]$, and hence by the Arzelà–Ascoli theorem, there exists a subsequence $\left(\mathsf{G}_{N_{n_k}}\right)_{k\in\mathbb{N}}$ converging uniformly to 0 on $[-L,L]$. This implies that for large enough $k$ we have
    \begin{align*}
        \left|\mathsf{G}_{N_{n_k}}(w)\right|<\delta,\  \  \forall w\in [-L,L],
    \end{align*}
    which contradicts \eqref{>delta}.
    
    Next, we consider (ii), namely $|w_{N_n}|\to\infty$ as $n\to\infty$. 
    Then, there exists a non-decreasing, in absolute value, subsequence of $(w_{N_n})_{n=1}^\infty$  diverging to infinity. We denote this subsequence using the same index, and thus, we have 
    $|w_{N_1}|\leq|w_{N_2}|\leq \dots\longrightarrow\infty$.
    Suppose first that
    $\lim_{N_n\to\infty}\frac{|w_{N_n}|}{\sqrt{N_n}}=0$,
    and let
    \begin{align*}
       f_{n}(z)\overset{\textnormal{def}}{=}\left(
     1-\frac{\textnormal{i}zw_{N_n}}{N_n}\right)^{-N_n}-\mathrm{e}^{\textnormal{i}zw_{N_n}},\  \   z\in\mathbb{R}. 
    \end{align*}
    One can easily check that 
   $f_n$ converges to $0$ pointwise,
    and so does $\mathsf{G}_{N_n}(w_{N_n})$.  
    This again violates \eqref{>delta}.
    Finally, we assume that 
    $\lim_{N_n\to\infty}\frac{|w_{N_n}|}{\sqrt{N_n}}> 0$, namely it is equal to a positive constant $\ell$, or $\infty$. 
    Note that in this case, $g(w_{N_n})=0$ for large enough $n$.
Let 
\begin{align*}
\mathcal{I}_n\overset{\textnormal{def}}{=}\int_{-\infty}^{\infty}\left(1-\frac{\textnormal{i}zw_{N_n}}{N_n}\right)^{-N_n}\check{g}(z)\mathrm{d}z,\  \ n\in\mathbb{N},
 \end{align*}
    We show in the following that $\mathcal{I}_n\to 0$ as $n\to\infty$ in both cases.  
    As a result, we conclude in view of \eqref{G_N}, that $|\mathsf{G}_{N_n}(w_{N_n})|< \delta$ for sufficiently large $ n$, which contradicts \eqref{>delta} as well, and completes the proof of the lemma. 
    
    Assume first that $\lim_{n\to\infty}\frac{|w_{N_n}|}{\sqrt{N_n}}= \infty$.
     Observe that in this case, 
     $\left(1-\frac{\textnormal{i}zw_{N_n}}{N_n}\right)^{-N_n}\overset{n\to\infty}{\longrightarrow} 0$, 
   and hence, $\mathcal{I}_n\overset{n\to\infty}{\longrightarrow} 0$ (with the same reasoning as before) as desired. Next,  suppose $\lim_{n\to\infty}\frac{|w_{N_n}|}{\sqrt{N_n}}=\ell>0$. Since $\check{g}$ is integrable, one can choose $M_\delta>0$ such that
\begin{align*}\int_{\mathbb{R}\setminus (-M_\delta,M_\delta)}\check{g}(z)\mathrm{d}z \le \frac{\delta}{8}.  
\end{align*}
Note also that $\left|\left(1-\frac{\textnormal{i}zw_{N_n}}{N_n}\right)^{-N_n}\right|\leq 1$ for all $n\in\mathbb{N}$. So 
we have
 \begin{align*}
|\mathcal{I}_n|\leq\left|\int_{-M_\delta}^{M_\delta}\left(1-\frac{\textnormal{i}zw_{N_n}}{N_n}\right)^{-N_n}\check{g}(z)\mathrm{d}z\right|+\frac{\delta}{8}.
 \end{align*}
 Now, since $\lim_{n\to\infty}\frac{|w_{N_n}|}{\sqrt{N_n}}=\ell$, there exists $n_0\in\mathbb{N}$ such that $\frac{|w_{N_n}|}{\sqrt{N_n}}\leq \ell+1$ for all $n\geq n_0$. We pick such an $n_0$ so that $N_{n_0}> M_\delta^2(\ell+1)^2$ also holds, on account of which we can write a Taylor expansion for the function $\left(1-\frac{\textnormal{i}zw_{N_n}}{N_n}\right)^{-N_n}$ on the interval $(-M_\delta,M_\delta)$. 
 Accordingly, we have for all $z\in(-M_\delta,M_\delta)$ and $n\geq n_0$,
 \begin{align*}
    \left(1-\frac{\textnormal{i}zw_{N_n}}{N_n}\right)^{-N_n}= \mathrm{e}^{\textnormal{i}zw_{N_n}-\frac{z^2w_{N_n}^2}{N_n}+
    R_n(z)},
 \end{align*}
 where $
    R_n(z)= \frac{z^3w_{N_n}^3}{3N_n^2}+
    \frac{z^4w_{N_n}^4}{4N_n^3}+ \dots=
    O\left(\frac{1}{\sqrt{N_n}}\right)$.
 Then, 
\begin{align*}
|\mathcal{I}_n|&\leq\left|\int_{-M_\delta}^{M_\delta}\mathrm{e}^{\textnormal{i}zw_{N_n}-\frac{z^2w_{N_n}^2}{N_n}+R_n(z)}\check{g}(z)\mathrm{d}z\right|+\frac{\delta}{8}\\&\leq\left|\int_{-M_\delta}^{M_\delta}\mathrm{e}^{\textnormal{i}zw_{N_n}-\frac{z^2w_{N_n}^2}{N_n}
}\check{g}(z)\mathrm{d}z\right|+\left|\int_{-M_\delta}^{M_\delta}\mathrm{e}^{\textnormal{i}zw_{N_n}-\frac{z^2w_{N_n}^2}{N_n}}\left(\mathrm{e}^{R_n(z)}-1\right)\check{g}(z)\mathrm{d}z\right|+\frac{\delta}{8}.
 \end{align*}
 The second integral in the above expression converges to $0$ as $n\to\infty$, since $\left(\mathrm{e}^{R_n(z)}-1\right)\overset{n\to\infty}{\longrightarrow} 0$ and the integrand is bounded by $\check{g}$ up to a multiplicative constant on $[-M_\delta,M_\delta]$. Thus, for large enough $n\in\mathbb{N}$, we get  
 \begin{align*}
|\mathcal{I}_n|&\leq\left|\int_{-M_\delta}^{M_\delta}\mathrm{e}^{\textnormal{i}zw_{N_n}-\frac{z^2w_{N_n}^2}{N_n}}\check{g}(z)\mathrm{d}z\right|+\frac{2\delta}{8}
\leq\left|\int_{-\infty}^{\infty}\mathrm{e}^{\textnormal{i}zw_{N_n}-\frac{z^2w_{N_n}^2}{N_n}
}\check{g}(z)\mathrm{d}z\right|+\frac{3\delta}{8}.
 \end{align*}
Note now that for each $n\in\mathbb{N}$, the function 
 $ \mathrm{e}^{-\frac{z^2w_{N_n}^2}{N_n}
 }\check{g}(z)$
 is a Schwartz function and the Fourier transform maps the Schwartz space to itself.
It thus follows that for $n$ large enough, we have
\begin{align*}
\left|\int_{-\infty}^{\infty}\mathrm{e}^{\textnormal{i}zw_{N_n}-\frac{z^2w_{N_n}^2}{N_n}
}\check{g}(z)\mathrm{d}z\right|<\frac{\delta}{8}, 
\end{align*}
 and hence,
$\left|\mathcal{I}_n\right|<\frac{\delta}{2}$. 
This finally concludes the proof of the lemma.
 \end{proof}
  
\begin{proof}[Proof of Theorem \ref{thm-tildeUnifApp} (the inductive step)]
Suppose \eqref{UniformApproximation} holds for $K\in\mathbb{N}$, namely, for all $g\in C_{c}^{\infty}(\mathbb{R}^{K})$, 
we have 
    \begin{align*}
\sup_{\mathbf{x}^{(N)}\in \mathbb{W}_N}\left|\tilde{\Lambda}_K^{\infty}g\left(\omega\left(\mathbf{x}^{(N)}\right)\right)-\Lambda_K^{N}g\left(\mathbf{x}^{(N)}\right)\right|\longrightarrow 0, \  \ \textnormal{as } N\to\infty.
 \end{align*} 
 We need to show that for all $g\in C_{c}^{\infty}(\mathbb{R}^{K+1})$,
 
    \begin{align*}
\sup_{\mathbf{x}^{(N)}\in \mathbb{W}_N}\left|\tilde{\Lambda}_{K+1}^{\infty}g\left(\omega\left(\mathbf{x}^{(N)}\right)\right)-\Lambda_{K+1}^{N}g\left(\mathbf{x}^{(N)}\right)\right|\longrightarrow 0, \  \ \textnormal{as } N\to\infty.
 \end{align*} 
 Let $g\in C_{c}^{\infty}(\mathbb{R}^{K+1})$. We have by definition that 
 
 \begin{align*}
\tilde{\Lambda}_{K+1}^{\infty}g\left(\omega\left(\mathbf{x}^{(N)}\right)\right)&=
\left(\prod_{l=1}^{K}\frac{1}{l!}\right)
\int_{\mathbb{R}^{K+1}}
\det \left(\phi_{\omega\left(x_{1}^{(N)},\dots,x_{N-K-j+1}^{(N)}\right)}^{(j-1)}\left(\mathrm{d}y_{i}\right)\right)_{i,j=1}^{K+1} \Delta_{K+1}(\mathbf{y})g(\mathbf{y})
\\&= \left(\prod_{l=1}^{K}\frac{1}{l!}\right)
\int_{\mathbb{R}^{K+1}}
\det \left(\phi_{\omega\left(x_{1}^{(N)},\dots,x_{N+1-i}^{(N)}\right)}^{(K+1-i)}(\mathrm{d}y_{K+2-j})\right)_{i,j=1}^{K+1} \Delta_{K+1}(\mathbf{y})g(\mathbf{y}).
\end{align*}
Also, from \eqref{Lambda_K^N-M} we have
\begin{align*}
\Lambda_{K+1}^{N}g\left(\mathbf{x}^{(N)}\right)
=\left(\prod_{l=1}^{K}\frac{1}{l!}\right)\int_{\mathbb{R}^{K+1}}
\det \left(\mathcal{M}^{(K+1-i)}\left(\mathrm{d}y_{K+2-j};x_1^{(N)},\dots,x_{N+1-i}^{(N)}\right)\right)_{i,j=1}^{K+1}\Delta_{K+1}(\mathbf{y})g(\mathbf{y}).
\end{align*}
Thus, we obtain
\begin{align*}
\Big|&\tilde{\Lambda}_{K+1}^{\infty}g\left(\omega\left(\mathbf{x}^{(N)}\right)\right)-\Lambda_{K+1}^{N}g\left(\mathbf{x}^{(N)}\right)\Big|
=
\Bigg|\int_{\mathbb{R}^{K+1}}
\left(\prod_{l=1}^{K}\frac{1}{l!}\right)\Bigg[\det \left(\phi_{\omega\left(x_{1}^{(N)},\dots,x_{N+1-i}^{(N)}\right)}^{(K+1-i)}(\mathrm{d}y_{K+2-j})\right)_{i,j=1}^{K+1}\\&\hspace{4.35cm}-
\det \left(\mathcal{M}^{(K+1-i)}\left(\mathrm{d}y_{K+2-j};x_{1}^{(N)},\dots,x_{N+1-i}^{(N)}\right)\right)_{i,j=1}^{K+1}\Bigg]
\Delta_{K+1}(\mathbf{y})g(\mathbf{y})\Bigg|
\end{align*}
Now, by expanding both determinants along the first row (recall the definition of the determinant in the weak sense is consistent with expanding along a row or column), we get 
\begin{align*}
\Big|&\tilde{\Lambda}_{K+1}^{\infty}g\left(\omega\left(\mathbf{x}^{(N)}\right)\right)-\Lambda_{K+1}^{N}g\left(\mathbf{x}^{(N)}\right)\Big|\\
&=\Bigg|\int_{\mathbb{R}^{K+1}}
\sum_{k=1}^{K+1}\left(\prod_{l=1}^{K}\frac{1}{l!}\right)(-1)^{1+k}\Bigg[\phi_{\omega\left(\mathbf{x}^{(N)}\right)}^{(K)}(\mathrm{d}y_{K+2-k})\det \left(\phi_{\omega\left(x_{1}^{(N)},\dots,x_{N+1-i}^{(N)}\right)}^{(K+1-i)}(\mathrm{d}y_{K+2-j})\right)_{i=2,j=1,j\neq k}^{K+1}
\\&\hspace{.5cm}-\mathcal{M}^{(K)}\left(\mathrm{d}y_{K+2-k};\mathbf{x}^{(N)}\right)
\det \left(\mathcal{M}^{(K+1-i)}\left(\mathrm{d}y_{K+2-j};x_{1}^{(N)},\dots,x_{N+1-i}^{(N)}\right)\right)_{i=2,j=1,j\neq k}^{K+1}\Bigg]\Delta_{K+1}(\mathbf{y})g(\mathbf{y})\Bigg|
\\
&=\frac{1}{K!}\Bigg|
\sum_{k=1}^{K+1}(-1)^{1+k}\int_{\mathbb{R}^K}\left(\prod_{l=1}^{K-1}\frac{1}{l!}\right)
\Bigg[
\det \left(\phi_{\omega\left(x_{1}^{(N)},\dots,x_{N+1-i}^{(N)}\right)}^{(K+1-i)}(\mathrm{d}y_{K+2-j})\right)_{i=2,j=1,j\neq k}^{K+1}
\\&\hspace{3cm}-
\det \left(\mathcal{M}^{(K+1-i)}\left(\mathrm{d}y_{K+2-j};x_{1}^{(N)},\dots,x_{N+1-i}^{(N)}\right)\right)_{i=2,j=1,j\neq k}^{K+1}\Bigg]
\prod_{\underset{i,j\neq K+2-k}{1\leq i<j\leq K+1}}(y_i-y_j)\\&\hspace{4.8cm}\times(-1)^{K+1-k}\underbrace{\int_{\mathbb{R}}\phi_{\omega\left(\mathbf{x}^{(N)}\right)}^{(K)}(\mathrm{d}y_{K+2-k})\prod_{\underset{j\neq K+2-k}{1\leq j\leq K+1}}(y_{K+2-k}-y_j)g(\mathbf{y})
}_{(-1)^K\int_{\mathbb{R}}\phi_{\omega\left(\mathbf{x}^{(N)}\right)}(\mathrm{d}y_{K+2-k})h_k(\mathbf{y})}
\\
&\;\;\;+\sum_{k=1}^{K+1}(-1)^{1+k}
\int_{\mathbb{R}^K}\left(\prod_{l=1}^{K}\frac{1}{l!}\right)
\det \left(\mathcal{M}^{(K+1-i)}\left(\mathrm{d}y_{K+2-j};x_{1}^{(N)},\dots,x_{N+1-i}^{(N)}\right)\right)_{\underset{j\neq k}{i=2,j=1}}^{K+1}\prod_{\underset{i,j\neq K+2-k}{1\leq i<j\leq K+1}}(y_i-y_j)
\\&\hspace{1.2cm}\times(-1)^{K+1-k}
\underbrace{
\int_{\mathbb{R}}\left[\mathcal{M}^{(K)}\left(\mathrm{d}y_{K+2-k};\mathbf{x}^{(N)}\right)-\phi_{\omega\left(\mathbf{x}^{(N)}\right)}^{(K)}(\mathrm{d}y_{K+2-k})\right]
\prod_{\underset{j\neq K+2-k}{1\leq j\leq K+1}}(y_{K+2-k}-y_j)g(\mathbf{y})
}_{(-1)^K\int_{\mathbb{R}}\left[\mathcal{M}\left(\mathrm{d}y_{K+2-k};\mathbf{x}^{(N)}\right)-\phi_{\omega\left(\mathbf{x}^{(N)}\right)}(\mathrm{d}y_{K+2-k})\right]h_k(\mathbf{y})}\Bigg|, 
\end{align*} 
 where $h_k\in C_c^{\infty}(\mathbb{R}^{K+1})$ is defined by
 \begin{align*}
     h_k(\mathbf{y})=\frac{\mathrm{d}^K}{\mathrm{d}y_{K+2-k}^K}\left[\prod_{\underset{j\neq K+2-k}{1\leq j\leq K+1}}(y_{K+2-k}-y_j)g(\mathbf{y})\right].
 \end{align*}
Thus, we have
\begin{align*}
\Big|&\tilde{\Lambda}_{K+1}^{\infty}g\left(\omega\left(\mathbf{x}^{(N)}\right)\right)-\Lambda_{K+1}^{N}g\left(\mathbf{x}^{(N)}\right)\Big|\leq\frac{1}{K!}
\sum_{k=1}^{K+1}\int_{\mathbb{R}}\phi_{\omega\left(\mathbf{x}^{(N)}\right)}(\mathrm{d}y_{K+2-k})
\\&\quad \Bigg|\int_{\mathbb{R}^K}\left(\prod_{l=1}^{K-1}\frac{1}{l!}\right)\Bigg[
\det \left(\phi_{\omega\left(x_{1}^{(N)},\dots,x_{N+1-i}^{(N)}\right)}^{(K+1-i)}(\mathrm{d}y_{K+2-j})\right)_{i=2,j=1,j\neq k}^{K+1}
\\&\hspace{2cm}-
\det \left(\mathcal{M}^{(K+1-i)}\left(\mathrm{d}y_{K+2-j};x_{1}^{(N)},\dots,x_{N+1-i}^{(N)}\right)\right)_{i=2,j=1,j\neq k}^{K+1}\Bigg]
\prod_{\underset{i,j\neq K+2-k}{1\leq i<j\leq K+1}}(y_i-y_j)h_k(\mathbf{y})\Bigg|
\\&\  \ +\frac{1}{K!}\sum_{k=1}^{K+1}\left(\prod_{l=1}^{K}\frac{1}{l!}\right)
\int_{\mathbb{R}^K}\det \left(\mathcal{M}^{(K+1-i)}\left(\mathrm{d}y_{K+2-j};x_{1}^{(N)},\dots,x_{N+1-i}^{(N)}\right)\right)_{i=2,j=1,j\neq k}^{K+1}\prod_{\underset{i,j\neq K+2-k}{1\leq i<j\leq K+1}}(y_i-y_j)
\\&\hspace{5.5cm}\times\left|\int_{\mathbb{R}}\left[\mathcal{M}\left(\mathrm{d}y_{K+2-k};\mathbf{x}^{(N)}\right)-\phi_{\omega\left(\mathbf{x}^{(N)}\right)}(\mathrm{d}y_{K+2-k})\right]h_k(\mathbf{y})\right|.
 \end{align*} 
To complete the proof, we need to show that  
 the above expression becomes arbitrarily small, uniformly in $\mathbf{x}^{(N)}$, for large enough $N$.
Clearly, it suffices to prove this for each term of the above two series individually. We prove that this is indeed the case for the first series. A completely analogous argument works for the second too.

Let $\varepsilon>0$ be given.
Fix $1\le k\le K+1$ and consider the $k$-th term of the first series. For each $a\in\mathbb{R}$, let $\mathsf{h}_{a}\in C_{c}^{\infty}(\mathbb{R}^{K})$  be the function defined by
\begin{align*}
    \mathsf{h}_a(\mathbf{y})\overset{\textnormal{def}}{=}h_k(y_1,\dots,y_{k-1},a,y_k,\dots,y_K), \  \  \mathbf{y}\in\mathbb{R}^K.
\end{align*}

Since $h_k\in C_{c}^{\infty}(\mathbb{R}^{K+1})$, there exists a compact interval $\mathfrak{I}_k$ such that $\mathsf{h}_{a}\equiv 0$ 
for all $a\notin \mathfrak{I}_k$.
Moreover, by virtue of the induction hypothesis, for every $a\in \mathfrak{I}_k$, there exist a smallest $\mathsf{N}(a)\in\mathbb{N}$ (depending on $a$), such that for all $N\ge \mathsf{N}(a)$, we have 
\begin{align}\label{kth term}
\sup_{\mathbf{x}^{(N)}\in \mathbb{W}_N}\left|\tilde{\Lambda}_K^{\infty}\mathsf{h}_{a}\left(\omega\left(\mathbf{x}^{(N)}\right)\right)-\Lambda_K^{N}\mathsf{h}_{a}\left(\mathbf{x}^{(N)}\right)\right|
<\varepsilon. 
\end{align}
Note now that if $\mathsf{M}_k\overset{\textnormal{def}}{=}\sup_{a\in \mathfrak{I}_k}\mathsf{N}(a)$ is finite,
then the above inequality holds for all $a\in \mathfrak{I}_k$ and $N\ge \mathsf{M}_k$, and so, the result follows immediately owing to the fact that $\phi_{\omega\left(\mathbf{x}^{(N)}\right)}$ is a probability measure on $\mathbb{R}$ for every $\mathbf{x}^{(N)}\in \mathbb{W}_N$.

We thus wish to show that $\mathsf{M}_k<\infty$. Suppose by contradiction that it is not true. Accordingly, one can find a sequence $\left(\mathfrak{a}_n\right)_{n=1}^{\infty}\in \mathfrak{I}_k$ with $\mathfrak{a}_n\overset{n\to\infty}{\longrightarrow} \mathfrak{a}_*\in \mathfrak{I}_k$ such that 
$\mathsf{N}(n)\nearrow\infty$.
Let $\mathsf{h}_*\overset{\textnormal{def}}{=}\mathsf{h}_{\mathfrak{a}_*}$ and
$\mathsf{N}_*\in\mathbb{N}$ 
be such that 
\begin{align*}\sup_{\mathbf{x}^{(N)}\in \mathbb{W}_N}\left|\tilde{\Lambda}_K^{\infty}\mathsf{h}_*\left(\omega\left(\mathbf{x}^{(N)}\right)\right)-\Lambda_K^{N}\mathsf{h}_*\left(\mathbf{x}^{(N)}\right)\right|< \frac{\varepsilon}{2}
\end{align*} 
for all $N\ge \mathsf{N}_*$ (note that $\mathsf{N}_{*}$ is also well-defined by the induction hypothesis).

Now, since $h_k\in C_c^{\infty}(\mathbb{R}^{K+1})$ and $\mathfrak{a}_n\overset{n\to\infty}{\longrightarrow}\mathfrak{a}_*$, we have $\mathsf{h}_n\overset{\textnormal{def}}{=}\mathsf{h}_{\mathfrak{a}_n}\overset{n\to\infty}{\longrightarrow} \mathsf{h}_*$ pointwise. Moreover, 
it follows from the Arzelà–Ascoli theorem that there exists a uniformly convergent sub-sequence of this sequence which we denote by the same notation. 
Hence, we can pick $\mathsf{m}\in\mathbb{N}$ large enough so that for all $n\ge \mathsf{m}$, we have $\mathsf{N}(n)> \mathsf{N}_*$ and 
\begin{align*}\left|\mathsf{h}_n(\mathbf{y})-\mathsf{h}_*(\mathbf{y})\right|<\frac{\varepsilon}{4},\  \ \forall \mathbf{y}\in \mathbb{R}^K.
\end{align*} 
Therefore, we obtain 
 \begin{align*}
     \left|\tilde{\Lambda}_K^{\infty}\mathsf{h}_{\mathsf{m}}\left(\omega\left(\mathbf{x}^{(N)}\right)\right)-\Lambda_K^{N}\mathsf{h}_{\mathsf{m}}\left(\mathbf{x}^{(N)}\right)\right|&\le  \left|\tilde{\Lambda}_K^{\infty}(\mathsf{h}_{\mathsf{m}}-\mathsf{h}_*)\left(\omega\left(\mathbf{x}^{(N)}\right)\right)\right|+\left|\Lambda_K^{N}(\mathsf{h}_{\mathsf{m}}-\mathsf{h}_*)\left(\mathbf{x}^{(N)}\right)\right|\\
     &\  \ +\left|\tilde{\Lambda}_K^{\infty}\mathsf{h}_{*}\left(\omega\left(\mathbf{x}^{(N)}\right)\right)-\Lambda_K^{N}\mathsf{h}_{*}\left(\mathbf{x}^{(N)}\right)\right|\\&<
     \frac{\varepsilon}{4}+\frac{\varepsilon}{4}+\frac{\varepsilon}{2}=\varepsilon,
 \end{align*}
 for all $N\ge \mathsf{N}_*$, and $\mathbf{x}^{(N)}\in \mathbb{W}_N$. This implies that
 \begin{align*}
     \sup_{\mathbf{x}^{(N)}\in \mathbb{W}_N}\left|\tilde{\Lambda}_K^{\infty}\mathsf{h}_{\mathsf{m}}\left(\omega\left(\mathbf{x}^{(N)}\right)\right)-\Lambda_K^{N}\mathsf{h}_{\mathsf{m}}\left(\mathbf{x}^{(N)}\right)\right|<\varepsilon,\  \  \forall N\ge \mathsf{N}_*.
 \end{align*}
 Thus, we must have $\mathsf{N}(\mathsf{m})\le \mathsf{N}_*$ which is a contradiction in view of the particular choice of $\mathsf{N}(\mathsf{m})$ and $\mathsf{m}$.
\end{proof}
We finally conclude Theorem \ref{thm-UnifApp}  by proving the following proposition.

\begin{prop}\label{prop-tildeLambda-Lambda}
Let $K\in \mathbb{N}$. Then, for all $g\in C_c^\infty(\mathbb{R}^K)$ we have:
\begin{align*}
\sup_{\mathbf{x}^{(N)}\in \mathbb{W}_N}\left|\tilde{\Lambda}_K^{\infty}g\left(\omega\left(\mathbf{x}^{(N)}\right)\right)-\Lambda_K^{\infty}g\left(\omega\left(\mathbf{x}^{(N)}\right)\right)\right|\to 0,
\;\; \textnormal{as } N\to\infty. 
 \end{align*}
\end{prop}

\begin{proof}

    We have by definition
    \begin{align}\label{tildeLambda-Lambda}
&\sup_{\mathbf{x}^{(N)}\in \mathbb{W}_N}\left|\tilde{\Lambda}_K^{\infty}g\left(\omega\left(\mathbf{x}^{(N)}\right)\right)-\Lambda_K^{\infty}g\left(\omega\left(\mathbf{x}^{(N)}\right)\right)\right|\nonumber\\ &\; =\sup_{\mathbf{x}^{(N)}\in \mathbb{W}_N}\Bigg|
\int_{\mathbb{R}^K}
\left(\prod_{l=1}^{K-1}\frac{1}{l!}\right)
\left(\det \left(\phi_{\omega\left(x_{K-j+1}^{(N)},\dots,x_{N}^{(N)}\right)}^{(j-1)}(\mathrm{d}y_{K+1-i})\right)_{i,j=1}^{K} -
\det \left(\phi_{\omega\left(\mathbf{x}^{(N)}\right)}^{(j-1)}(\mathrm{d}y_{K+1-i})\right)_{i,j=1}^{K} \right)h\left(\mathbf{y}\right)\Bigg|,
\end{align}
with 
$h(\mathbf{y})=\Delta_K(\mathbf{y})g(\mathbf{y})$.
We show that the above expression is arbitrarily small for large enough $N$. Define the following $K \times K$ matrices,
\begin{align*}
\tilde{\mathbf{F}}\left(\mathbf{z};\mathbf{x}^{(N)}\right)=
    \begin{bmatrix}
       (-\textnormal{i}z_K)^{K-1}\prod_{j=1}^N\left(1-\frac{\textnormal{i}z_Kx_j^{(N)}}{N}\right)^{-1}&\dots&(-\textnormal{i}z_1)^{K-1}\prod_{j=1}^N\left(1-\frac{\textnormal{i}z_1x_j^{(N)}}{N}\right)^{-1}\\
       (-\textnormal{i}z_K)^{K-2}\prod_{j=2}^N\left(1-\frac{\textnormal{i}z_Kx_j^{(N)}}{N-1}\right)^{-1}&\dots&(-\textnormal{i}z_1)^{K-2}\prod_{j=2}^N\left(1-\frac{\textnormal{i}z_1x_j^{(N)}}{N-1}\right)^{-1}\\\vdots& \ddots & \vdots \\\prod_{j=K}^N\left(1-\frac{\textnormal{i}z_Kx_j^{(N)}}{N-K+1}\right)^{-1}&\dots&\prod_{j=K}^N\left(1-\frac{\textnormal{i}z_1x_j^{(N)}}{N-K+1}\right)^{-1}
    \end{bmatrix},\\
\mathbf{F}\left(\mathbf{z};\mathbf{x}^{(N)}\right)=
    \begin{bmatrix}
       (-\textnormal{i}z_K)^{K-1}\prod_{j=1}^N\left(1-\frac{\textnormal{i}z_Kx_j^{(N)}}{N}\right)^{-1}&\dots&(-\textnormal{i}z_1)^{K-1}\prod_{j=1}^N\left(1-\frac{\textnormal{i}z_1x_j^{(N)}}{N}\right)^{-1}\\
       (-\textnormal{i}z_K)^{K-2}\prod_{j=2}^N\left(1-\frac{\textnormal{i}z_Kx_j^{(N)}}{N}\right)^{-1}&\dots&(-\textnormal{i}z_1)^{K-2}\prod_{j=2}^N\left(1-\frac{\textnormal{i}z_1x_j^{(N)}}{N}\right)^{-1}\\\vdots& \ddots & \vdots \\
    \prod_{j=K}^N\left(1-\frac{\textnormal{i}z_Kx_j^{(N)}}{N}\right)^{-1}&\dots&\prod_{j=K}^N\left(1-\frac{\textnormal{i}z_1x_j^{(N)}}{N}\right)^{-1}
    \end{bmatrix}.
\end{align*}
We now use \eqref{FT-pairing} repeatedly to rewrite the determinants in \eqref{tildeLambda-Lambda} in terms of the determinants of the corresponding Fourier transforms.  
Then, we do elementary row operations repeatedly for the second determinant.
We finally use \eqref{F-M} and expand both determinants in order to obtain
\begin{align*}
\int_{\mathbb{R}^K}&
\left[\det \left(\phi_{\omega\left(x_{K-j+1}^{(N)},\dots,x_{N}^{(N)}\right)}^{(j-1)}(\mathrm{d}y_{K+1-i})\right)_{i,j=1}^{K} -
\det \left(\phi_{\omega\left(\mathbf{x}^{(N)}\right)}^{(j-1)}(\mathrm{d}y_{K+1-i})\right)_{i,j=1}^{K} \right]\Delta_K(\mathbf{y})g\left(\mathbf{y}\right)
\\
&=(2\pi)^{-K}\int_{\mathbb{R}^K}\left[\det \left(\tilde{\mathbf{F}}\left(\mathbf{z};\mathbf{x}^{(N)}\right)\right)
-\det \left(\mathbf{F}\left(\mathbf{z};\mathbf{x}^{(N)}\right)\right)\right]
\check{h}\left(\mathbf{z}\right)\mathrm{d}\mathbf{z}
\\&=(2\pi)^{-K}\sum_{\sigma\in\mathbb{S}_K}\textnormal{sgn}(\sigma)\int_{\mathbb{R}^K}\int_{\mathbb{R}^K}
\left[\prod_{l=1}^K\left(1-\frac{\textnormal{i}z_{\sigma(l)}w_{\sigma(l)}}{N-l+1}\right)^{-(N-l+1)}-\prod_{l=1}^K\left(1-\frac{\textnormal{i}z_{\sigma(l)}w_{\sigma(l)}}{N}\right)^{-(N-l+1)}\right] 
\\&\hspace{4.5cm}\times\left(\prod_{l=1}^K(-\textnormal{i}z_{\sigma(l)})^{K-l}\check{h}(\mathbf{z})\mathrm{d}\mathbf{z}\right)\prod_{l=1}^K\mathcal{M}\left(\mathrm{d}w_{\sigma(l)};x_l^{(N)},\dots,x_N^{(N)}\right).
\end{align*}
Note now that for any fixed $w_{\sigma(l)}$, the expression in the brackets on the RHS goes to $0$ pointwise in $z_{\sigma(l)}$, and moreover, it is bounded. Therefore, since the function in the parentheses is integrable, the inner integral converges to $0$ by the dominated convergence theorem. One can then use a similar argument as in the proof of Lemma \ref{lem UnifConv G_N} to verify that the convergence is indeed uniform 
and hence, the expression can be made arbitrarily small by picking large enough $N$. Finally, the conclusion follows by fact that $\mathcal{M}(\mathrm{d}w;\mathbf{x})$ is a probability measure on $\mathbb{R}$.  
\end{proof}

\section{Properties of the finite-dimensional diffusions}\label{SectionFiniteDimDiffusions}  
\subsection{Well-posedness and Feller property}\label{SubsectionWellposedness}

In this subsection we present several, for the most part standard, results on the dynamics (\ref{FSDE}) that we will need.

\begin{lem} \label{lem-wellPosednessN}
The SDE (\ref{FSDE}) has a unique strong non-exploding solution $\left(\mathfrak{x}^{(N)}(t);t\ge 0\right)$ for any initial condition $\mathfrak{x}^{(N)}(0)=\mathbf{x}\in \mathbb{W}_{N,+}$. Moreover, almost surely, for all $t>0$, $\mathfrak{x}^{(N)}(t)\in \mathbb{W}_{N,+}^>$, where
\begin{equation*}
\mathbb{W}_{N,+}^>=\left\{\mathbf{x}\in \mathbb{R}^N:x_1>x_2>\cdots >x_N\ge 0\right\}.
\end{equation*}
Finally, if $\mathfrak{x}^{(N)}(0)=\mathbf{x}\in \mathbb{W}_{N,+}^\circ$, then almost surely for all $t\ge 0$, $\mathfrak{x}^{(N)}(t)\in \mathbb{W}_{N,+}^\circ$.
\end{lem}

\begin{proof}
All statements, except the last one, follow from the main theorem of \cite{Graczyk-Malecki} by verifying the elementary conditions therein. To establish the last statement, observe that since $\mathbf{x}\in \mathbb{W}_{N,+}^\circ$, using It\^{o}'s formula we have, for all $t\ge 0$,
\begin{align*}
\log\mathfrak{x}_N^{(N)}(t) &= \log x_N+ \mathsf{w}_N(t) -\frac{1+\eta}{2}t+\int_0^t\frac{1}{2\mathfrak{x}_N^{(N)}(s)}\mathrm{d}s+\int_0^t\sum_{j=1}^{N-1}\frac{\mathfrak{x}_j^{(N)}(s)}{\mathfrak{x}_N^{(N)}(s)-\mathfrak{x}_j^{(N)}(s)}\mathrm{d}s\\
&\geq \log x_N+ \mathsf{w}_N(t) -\frac{1+\eta}{2}t-(N-1)t\left(\inf_{s\in[0,t]}\left|1-\frac{\mathfrak{x}_{N}^{(N)}(s)}{\mathfrak{x}_{N-1}^{(N)}(s)}\right|
\right)^{-1}.
\end{align*}
The conclusion then follows since the RHS is finite.
\end{proof}

 The Markov semigroup associated with $\left(\mathfrak{x}^{(N)}(t);t\ge 0\right)$, the unique strong solution of \eqref{FSDE}, is denoted by $\left(\mathfrak{P}_N(t)\right)_{t\geq 0}$.

We next introduce certain matrix-valued dynamics $\left(\mathbf{H}_t^{(N)};t\ge 0\right)$ from \cite{Rider-Valko} whose eigenvalue evolution is described by the SDE \eqref{FSDE}.
Recall that $ \mathbb{H}_+(N) $ denotes the space of $ N\times N$ non-negative definite Hermitian matrices.
We denote by $\textnormal{Tr}(\mathbf{A})$ the trace of  a matrix $\mathbf{A}$. Consider the following $\mathbb{H}_+(N)$-valued SDE introduced by Rider-Valko in \cite{Rider-Valko},
\begin{align}\label{MSDE}
	\mathrm{d}\mathbf{H}_t^{(N)}=\frac{1}{2}\left(\mathrm{d}\boldsymbol{\Gamma}_t^{(N)}\mathbf{H}_t^{(N)}+\mathbf{H}_t^{(N)}\mathrm{d}\left(
 \boldsymbol{\Gamma}_t^{(N)}\right)^{\dag}\right)+\left(-\frac{\eta+N}{2}\mathbf{H}_t^{(N)}+\frac{1}{2}\left(1+\textnormal{Tr}\left(\mathbf{H}_t^{(N)}\right)\right)\mathbf{I}\right)\mathrm{d}t,
\end{align}
where 
   $\left(\boldsymbol{\Gamma}_t^{(N)};t \geq 0\right)$ is a matrix-valued complex Brownian motion; namely, an $N\times N$ matrix with the entries independent complex Brownian motions (the real and imaginary parts are independent standard real Brownian motions), and
$\mathbf{I}$ is the $N\times N$ identity matrix.
\begin{lem}
For any $\mathbf{H}_0^{(N)}\in\mathbb{H}_+(N)$, the SDE \eqref {MSDE} has a unique strong solution.
\end{lem}
\begin{proof}
    It is easy to see that the coefficients of the equation are Lipschitz continuous in  matrix norm, and thus existence and uniqueness of the strong solution follows from well-known results, see for example \cite{Multidim-Yamada-Watanabe}. 
\end{proof}

\begin{lem}\label{LemMatrixEvalRelation} 
Let $\mathbf{x}\in \mathbb{W}_{N,+}$ and $\mathbf{H} \in 
\mathbb{H}_+(N)$ be such that $\mathsf{eval}_N\left(\mathbf{H}\right)=\mathbf{x}$. 
Let $\left(\mathbf{H}_t^{(N)};t\ge 0\right)$ be the solution to \eqref{MSDE} starting from $\mathbf{H}$ and $\left(\mathfrak{x}^{(N)}(t);t\ge 0\right)$ be the solution to \eqref{FSDE} starting from $\mathbf{x}$. Then,
    \begin{equation}
\left(\mathsf{eval}_N\left(\mathbf{H}^{(N)}_t\right);t\geq 0\right)\overset{\textnormal{d}}{=} \left(\mathfrak{x}^{(N)}(t);t\ge 0\right).
    \end{equation}
\end{lem}

\begin{proof}
The statement is a consequence of the general result of \cite{Multidim-Yamada-Watanabe}, which relates matrix diffusions solving SDE generalising \eqref{MSDE} to their eigenvalues evolutions given in terms of SDE generalising \eqref{FSDE}. The computation is also implicit in \cite{Rider-Valko}.
\end{proof}

\begin{rmk}
By looking at the stochastic equation satisfied by the process $t\mapsto \det\big(\mathbf{H}_t^{(N)}\big)$, see proof of Proposition 12 in 
\cite{Rider-Valko}, we can improve the statement of Lemma \ref{lem-wellPosednessN} regarding never reaching the origin, to initial conditions $\mathbf{x}$ with coinciding coordinates.
\end{rmk}

\begin{prop}\label{prop-FellerProperty} Let $\eta \in \mathbb{R}$ and $N\in \mathbb{N}$. The semigroup $\left(\mathfrak{P}_N(t)\right)_{t\ge 0}$ is a Feller semigroup.
\end{prop}

\begin{proof}
    The semigroup associated to $\left(\mathbf{H}_t^{(N)};t\ge 0\right)$, as a unique strong solution to a SDE with Lipschitz coefficients, is Feller by classical results, see \cite{Revuz-Yor}. The Feller property of the eigenvalue process follows from that of the matrix process. More precisely,
    let $\left(\mathsf{Q}_N(t)\right)_{t\ge 0}$ denote the semigroup associated to $\left(\mathbf{H}_t^{(N)};t \ge 0\right)$. Note that, if $\mathbf{x}=\mathsf{eval}_N(\mathbf{H})$, one has from Lemma \ref{LemMatrixEvalRelation},
    \begin{equation*}
        \left(\mathfrak{P}_N(t)f\right)(\mathbf{x})=\left(\mathsf{Q}_N(t)f\circ\mathsf{eval}_N\right)\left(\mathsf{diag}(\mathbf{x})\right),
    \end{equation*}
   and moreover, note that if $f \in C_0(\mathbb{W}_{N,+})$ then $f \circ \mathsf{eval}_N \in C_0(\mathbb{H}_+(N))$. From this, and the Feller property of $\left(\mathsf{Q}_N(t)\right)_{t\ge 0}$, one immediately checks the Feller property for $\left(\mathfrak{P}_N(t)\right)_{t\ge 0}$. 
\end{proof}

\begin{rmk}
By hypoellipticity \cite{StroockPDE} of the generator of $(\mathsf{Q}_N(t))_{t\ge 0}$, shown in Proposition 12 of \cite{Rider-Valko}, we have that $\mathsf{Q}_N(t)(\mathbf{A},\mathrm{d}\mathbf{B})=\mathsf{Q}_N(t)(\mathbf{A},\mathbf{B})\mathrm{d}\mathbf{B}$, with $\mathrm{d}\mathbf{B}$ Lebesgue measure on $\mathbb{H}_+(N)$, where the density $\mathsf{Q}_N(t)(\mathbf{A},\mathbf{B})$ is smooth, viewed as a function in the matrix entries, see \cite{Rider-Valko}.  
\end{rmk}

Our next goal is to give a description of the $N$-dimensional dynamics $\left(\mathfrak{P}_N(t)\right)_{t\ge 0}$, in some sense, in terms of one-dimensional diffusions. For this, we define for fixed $N\in\mathbb{N}$ and $\eta\in\mathbb{R}$, consider unique strong solution to the following SDE in $[0,\infty)$,
\begin{align}\label{IGBM}
	\mathrm{d}\mathsf{z}^{(N)}(t) &= \mathsf{z}^{(N)}(t) \mathrm{d}\mathsf{w}(t) + \left[\left(1-\frac{\eta}{2}-N\right)\mathsf{z}^{(N)}(t)+\frac{1}{2}\right]\mathrm{d}t,
\end{align}
with $\mathsf{w}$ as a usual a standard Brownian motion, with infinitesimal generator,
\begin{align}\label{LN}
 \mathsf{L}_x^{(N)} = \frac{x^2}{2}\partial_{x}^2+\left(\left(1-\frac{\eta}{2}-N\right)x+\frac{1}{2}\right)\partial_{x}
\end{align} 
and transition density with respect to Lebesgue measure in $[0,\infty)$ that we denote by $\mathfrak{p}^{(N)}_t(x,y)$. It is known that $\infty$ is a natural boundary and $0$ an entrance boundary in the terminology of Feller \cite{ItoMcKean,Ethier-Kurtz}. In particular, almost surely, for all strictly positive times the diffusion $\mathsf{z}^{(N)}$ is in $(0,\infty)$. 

We then have the following proposition. Analogous results are well-known for certain generalisations of Dyson Brownian motion: the radial Dunkl \cite{Dunkl} and Heckman-Opdam \cite{HeckmanOpdam} processes.

\begin{prop}\label{prop-Core+TransitionDens}  Let $\eta \in \mathbb{R}$ and $N\in \mathbb{N}$. Then, $C^{\infty}_{c,\textnormal{sym}}\left(\mathbb{W}_{N,+}\right)$ forms a core for the generator of the semigroup $\left(\mathfrak{P}_N(t)\right)_{t\ge 0}$  and it is moreover invariant under the action of the generator. Finally, the kernel $\mathfrak{P}_N(t)(\mathbf{x},\mathrm{d}\mathbf{y})$ of $\mathfrak{P}_N(t)$ has an explicit density with respect to the Lebesgue measure in $\mathbb{W}_{N,+}$ given by (abusing notation), for $\mathbf{x}\in \mathbb{W}_{N,+}^\circ$, $t>0$,
\begin{align}\label{transitioDensity}
	\mathfrak{P}_N(t)(x,y)= \mathrm{e}^{-\lambda_Nt}\frac{\Delta_N(\mathbf{y})}{\Delta_N(\mathbf{x})} \det\left(\mathfrak{p}_t^{(N)}(x_i,y_j)\right)_{i,j=1}^{N},
\end{align}
with $\lambda_N=N(N-1)\left(1-\frac{3\eta}{2}-2N\right)/6$.
   
\end{prop}

\begin{proof} Let us denote by $(\mathsf{L}_N,\mathcal{D}(\mathsf{L}_N))$ the generator of $(\mathfrak{P}_N(t))_{t\ge 0}$. Define the formal differential operator $\mathbf{L}_N$ by,
\begin{equation}\label{InfinitesimalGen}
\mathbf{L}_N=\sum_{i=1}^N \frac{x_i^2}{2}\partial_{x_i}^2+\sum_{i=1}^N\left[-\frac{\eta}{2}x_i+\frac{1}{2}+\sum_{j=1,j\neq i}^N\frac{x_ix_j}{x_i-x_j}\right]\partial_{x_i}.
\end{equation}
We will identify shortly $\mathsf{L}_N$ with $\mathbf{L}_N$ on $C^{\infty}_{c,\textnormal{sym}}\left(\mathbb{W}_{N,+}\right)$.

We first check that $C_{c,\textnormal{sym}}^\infty(\mathbb{W}_{N,+})$ is invariant under $\mathbf{L}_N$. Let $g\in C_{c,\textnormal{sym}}^\infty(\mathbb{W}_{N,+})$ be arbitrary and let $f$ be a function in $C_{c,\textnormal{sym}}^\infty(\mathbb{R}^N)$ that restricts to $g\in C_{c,\textnormal{sym}}^\infty(\mathbb{W}_{N,+})$. Consider $\mathbf{L}_Nf$, where we view $\mathbf{L}_N$ as acting on functions on $\mathbb{R}^N$ in the obvious way, and observe that on $\mathbb{W}_{N,+}$ this restricts to $\mathbf{L}_N g$. It then suffices to show $\mathbf{L}_Nf \in C_{c,\textnormal{sym}}^\infty(\mathbb{R}^N)$. $\mathbf{L}_Nf$ is clearly of compact support and moreover it is symmetric in its arguments. Showing smoothness boils down to showing that the only part which is possibly singular,
\begin{equation*}
 (x_1,\dots, x_N)   \mapsto \sum_{1\le i< j\le N} x_i x_j \frac{\partial_{x_i}f(x_1,\dots,x_N)-\partial_{x_j}f(x_1,\dots,x_N)}{x_i-x_j},
\end{equation*}
is actually smooth on the whole of $\mathbb{R}^N$. This moreover boils down to proving the following claim. For a smooth symmetric function $\mathrm{H}$ on $\mathbb{R}^2$, the set of which we denote by $ C_{\textnormal{sym}}^\infty(\mathbb{R}^2)$, the function 
\begin{equation*}
(x,y) \mapsto \frac{\partial_x \mathrm{H}(x,y)-\partial_y\mathrm{H}(x,y)}{x-y}
\end{equation*}
is smooth on the whole of $\mathbb{R}^2$, which is a direct consequence of known results. Namely, for $\mathrm{H}\in C_{\textnormal{sym}}^\infty(\mathbb{R}^2)$ let $\mathfrak{H}$ be the unique function such that 
\begin{equation*}
\mathrm{H}(x,y)=\mathfrak{H}\left(\mathrm{S}_1(x,y),\mathrm{S}_2(x,y)\right),
\end{equation*}
where $\mathrm{S}_1(x,y)=-(x+y)$ and $\mathrm{S}_2(x,y)=xy$, see \cite{Ball}.
Then, we have, see \cite{Ball},
\begin{align*}
\frac{\partial_x \mathrm{H}(x,y)-\partial_y\mathrm{H}(x,y)}{x-y}=-\frac{\partial\mathfrak{H}}{\partial\mathrm{S}_2}.
\end{align*}
Finally, from Corollary 3.3 of \cite{Ball}, $\mathfrak{H}$ is a smooth function and this establishes the claim. 

We now show that $\mathsf{L}_N$ extends $\mathbf{L}_N$ on $C_0(\mathbb{W}_{N,+})$ and $C^\infty_{c,\textnormal{sym}}(\mathbb{W}_{N,+})$ is a core for it. First, for $f$ so that $\mathbf{L}_Nf$ makes sense, define the process $t\mapsto \mathsf{M}_f(t)$ by, 
\begin{equation}\label{martingaleproblem}
\mathsf{M}_f(t)\overset{\textnormal{def}}{=}f(\mathsf{z}(t))-f(\mathsf{z}(0))-\int_0^t \mathbf{L}_N f(\mathsf{z}(s))\mathrm{d}s,
\end{equation}
where $\mathsf{Law}(\mathsf{z})=\mathsf{Law}(\mathfrak{x})$, with $\mathfrak{x}$ the unique strong solution to \eqref{FSDE}, namely the process with generator $\mathsf{L}_N$. Observe that, for $f \in C^\infty_{c,\textnormal{sym}}(\mathbb{W}_{N,+})$, $\mathsf{M}_f$ is bounded since $\mathbf{L}_N C^\infty_{c,\textnormal{sym}}(\mathbb{W}_{N,+}) \subset C^\infty_{c,\textnormal{sym}}(\mathbb{W}_{N,+})$. Hence, an application of  It\^{o}'s formula \cite{Revuz-Yor,Kallenberg} to the solution of the SDE \eqref{FSDE} gives that $t\mapsto \mathsf{M}_f(t)$ is a martingale (with respect to the natural filtration of the coordinate process). In particular, from Theorem 17.23 of \cite{Kallenberg} we obtain that $\mathsf{L}_N$ extends (we do not know yet if this extension is unique) $\mathbf{L}_N\big|_{C^\infty_{c,\textnormal{sym}}(\mathbb{W}_{N,+})}$ to $C_0(\mathbb{W}_{N,+})$.  

It remains to show that $C^\infty_{c,\textnormal{sym}}(\mathbb{W}_{N,+})$ is actually a core for the generator $\mathsf{L}_N$. This is done by connecting the SDE \eqref{FSDE} to a well-posed martingale problem. Define $C_\mathbf{x}(\mathbb{R}_+,\mathbb{W}_{N,+})=\left\{f\in C(\mathbb{R}_+,\mathbb{W}_{N,+}):f(0)=\mathbf{x}\right\}$.
We say that a family of probability measures $\{\mathbf{P}_\mathbf{x}\}_{\mathbf{x}\in \mathbb{W}_{N,+}}$ on $C(\mathbb{R}_+,\mathbb{W}_{N,+})$, endowed with the filtration generated by the coordinate process, is a solution to the martingale problem for $(\mathbf{L}_N,C^\infty_{c,\textnormal{sym}}(\mathbb{W}_{N,+}))$ if $\mathbf{P}_\mathbf{x}(C_\mathbf{x}(\mathbb{R}_+,\mathbb{W}_{N,+}))=1$, and under $\mathbf{P}_\mathbf{x}$, for every $f\in C^\infty_{c,\textnormal{sym}}(\mathbb{W}_{N,+})$, the process $t\mapsto \mathsf{M}_f(t)$ defined in \eqref{martingaleproblem}, where $\mathsf{Law}(\mathsf{z})=\mathbf{P}_\mathbf{x}$,
 is a martingale. We now show that this martingale problem is well-posed, namely has a unique solution. First, the law of any weak solution, to the SDE \eqref{FSDE} starting from $\mathbf{x}$, gives rises to a solution to the martingale problem by an application of It\^{o}'s formula \cite{Revuz-Yor,Kallenberg}. We then show that this solution to the martingale problem must be unique. Consider an arbitrary solution to the martingale problem. For $R>0$, define the stopping times,
 \begin{equation*}
\zeta_R=\inf\left\{t\ge 0;||\mathsf{z}(t)||\ge R \ \ \textnormal{or}  \ \sum_{1=i\neq j=N}\left|\int_0^t\frac{\mathsf{z}_i(s)\mathsf{z}_j(s)}{\mathsf{z}_i(s)-\mathsf{z}_j(s)}\mathrm{d}s \right|\ge R \right\},
 \end{equation*}
 where, for $\mathbf{x}\in \mathbb{R}^N$, $||\mathbf{x}||$ is its Euclidean norm. For $i,j=1,\dots,N$, we choose $\mathrm{f}_i,\mathrm{f}_{ij}\in C_c^\infty(\mathbb{W}_{N,+})$ satisfying $\mathrm{f}_i(\mathbf{x})=x_i, \mathrm{f}_{ij}(\mathbf{x})=x_ix_j$, for $||\mathbf{x}||\le R$. We now claim that, for all $i,j=1,\dots,N$, the processes defined by,
\begin{equation*}
t \mapsto \mathsf{N}_i(t\wedge \zeta_R)\overset{\textnormal{def}}{=}\mathsf{M}_{\mathrm{f}_i^R}(t\wedge \zeta_R)=\mathsf{z}_i(t\wedge \zeta_R)-\mathsf{z}_i(0)-\int_0^{t\wedge \eta_R} \left(-\frac{\eta}{2}\mathsf{z}_i(s)+\frac{1}{2}+\sum_{j=1,j\neq i}^N\frac{\mathsf{z}_i(s) \mathsf{z}_j(s)}{\mathsf{z}_i(s)-\mathsf{z}_j(s)}\right)\mathrm{d}s
\end{equation*}
and $t \mapsto \mathsf{N}_{ij}(t\wedge \zeta_R)\overset{\textnormal{def}}{=}\mathsf{M}_{\mathrm{f}_{ij}^R}(t\wedge \zeta_R)$ are martingales. We approximate $\mathrm{f}_i^R$ and $\mathrm{f}_{ij}^R$ by functions in $C_{c,\textnormal{sym}}^\infty(\mathbb{W}_{N,+})$, for which we know the martingale property of $\mathsf{M}_f$. By the Stone-Weirstrass theorem the space $C_{c,\textnormal{sym}}^\infty(\mathbb{W}_{N,+})$ is dense in $C_0(\mathbb{W}_{N,+})$ but we need slightly more in convergence of the derivatives.
More generally, let us take a function $g\in C_{c}^\infty(\mathbb{W}_N)$ (note that we can view the functions $\mathrm{f}_i^R,\mathrm{f}_{ij}^R$ as in $C_{c}^\infty(\mathbb{W}_N)$). Then, extend $g$ by symmetry to the whole of $\mathbb{R}^N$ to obtain a function $\mathsf{G}$. Observe that, by construction $\mathsf{G}$ is symmetric, has compact support, is smooth except on the hyperplanes $\{x_{i_1}=x_{i_2}=\cdots=x_{i_\ell}\}$, is Lipschitz on the whole of $\mathbb{R}^N$ and $\mathsf{G}|_{\mathbb{W}_N}=g$. Let $\mathfrak{u}$ be a symmetric approximation of the identity and define, for each $m \in \mathbb{N}$,
\begin{equation*}
\mathfrak{g}_m(\mathbf{x})\overset{\textnormal{def}}{=}\int_{\mathbb{R}^N} \mathsf{G}(\mathbf{y}) m^N\mathfrak{u}\left(m(\mathbf{x}-\mathbf{y})\right)\mathrm{d}\mathbf{y}.
\end{equation*}
Then, observe that, $\mathfrak{g}_m \in C^\infty_{c,\textnormal{sym}}(\mathbb{R}^N)$ and that $\mathfrak{g}_m$ approximates, as $m \to \infty$, $g$ in $C_0(\mathbb{W}_N)$ and all its derivatives approximate uniformly on compact sets in $\mathbb{W}_N^\circ$ the derivatives of $g$. From the above we get that there exist sequences $(\mathfrak{f}_i^m)_{m=1}^\infty, (\mathfrak{f}_{ij}^{m})_{m=1}^\infty$ in $C_{c,\textnormal{sym}}^\infty(\mathbb{W}_{N,+})$ so that, as $m \to \infty $, for all $i,j=1,\dots,N$,
\begin{equation*}
\mathsf{M}_{\mathfrak{f}_{i}^{m}}(\cdot \wedge \zeta_R) \longrightarrow \mathsf{M}_{\mathrm{f}_{i}^{R}}(\cdot \wedge \zeta_R), \ \mathsf{M}_{\mathfrak{f}_{ij}^{m}}(\cdot \wedge \zeta_R) \longrightarrow \mathsf{M}_{\mathrm{f}_{ij}^{R}}(\cdot \wedge \zeta_R),
\end{equation*}
uniformly on compact sets, in distribution under $\mathbf{P}_\mathbf{x}$. Hence, we easily obtain that for all $i,j=1,\dots,N$, the processes $t \mapsto  \mathsf{M}_{\mathrm{f}_{i}^{R}}(t \wedge \zeta_R)$, $t \mapsto  \mathsf{M}_{\mathrm{f}_{ij}^{R}}(t \wedge \zeta_R)$ are martingales as desired. Then, we can follow the computation in Theorem 32.7 of \cite{Kallenberg} or Proposition 5.4.6 of \cite{KaratzasShreve} to obtain independent standard Brownian motions $(\mathsf{w}_i(\cdot))_{i=1}^N$ (with respect to possibly an extension of the original filtration), such that, for all $i=1,\dots,N$, 
 \begin{equation*}
 \mathsf{N}_i(t\wedge \zeta_R)=\int_0^{t\wedge \zeta_R} \mathsf{z}_i(s)\mathrm{d}\mathsf{w}_i(s), \ \ \forall t\ge 0.
 \end{equation*}
 In particular, we obtain that $(\mathsf{z}_i(\cdot\wedge \zeta_R),\mathsf{w}_i(\cdot\wedge \zeta_R))_{i=1}^N$ is a weak solution to the SDE \eqref{FSDE} up until time $\zeta_R$. But by pathwise uniqueness shown in \cite{Graczyk-Malecki}, this solution must be equal to the unique strong solution of \eqref{FSDE} up until $\zeta_R$. Observe that, under the law of the solution to \eqref{FSDE}, $\zeta_R\to \infty$, as $R \to \infty$. In particular, $\zeta_R \to \infty$, as $R\to \infty$, under $\mathbf{P}_\mathbf{x}$ and from a solution to the martingale problem we constructed a solution to \eqref{FSDE} for all times. Hence, the solution to the martingale problem for $(\mathbf{L}_N,C^\infty_{c,\textnormal{sym}}(\mathbb{W}_{N,+}))$ must be unique (for otherwise we would have constructed two different in law solutions to \eqref{FSDE}). The associated diffusion process to it, see \cite{Kallenberg}, is clearly nothing but the solution to the SDE \eqref{FSDE}, namely the Feller process with semigroup $(\mathfrak{P}_N(t))_{t\ge 0}$. By virtue of Theorem 2.5 of \cite{VanCasteren}, since $C_{c,\textnormal{sym}}^\infty(\mathbb{W}_{N,+})$ is dense $C_0(\mathbb{W}_{N,+})$ and the martingale problem for $(\mathbf{L}_N,C^\infty_{c,\textnormal{sym}}(\mathbb{W}_{N,+}))$ is well-posed, we obtain that operator $\mathbf{L}_N$ is closable in $C_0(\mathbb{W}_{N,+})$ and the closure is nothing but the generator $\mathsf{L}_N$ of the semigroup $(\mathfrak{P}_{N}(t))_{t \ge 0}$. Namely $\overline{\mathsf{L}_N|_{C_{c,\textnormal{sym}}^\infty(\mathbb{W}_{N,+})}}=\overline{\mathbf{L}_N|_{C_{c,\textnormal{sym}}^\infty(\mathbb{W}_{N,+})}}=\mathsf{L}_N$ or in other words, $C_{c,\textnormal{sym}}^\infty(\mathbb{W}_{N,+})$ is a core for $(\mathsf{L}_N,\mathcal{D}(\mathsf{L}_N))$.

The final statement in the proposition is a direct consequence of the following observation. The generator $\mathbf{L}_N$ acting on $g\in C_{c,\textnormal{sym}}^\infty(\mathbb{W}_{N,+})$ can be written, after an elementary computation, as
\begin{align*}
\mathbf{L}_Ng(x_1,\dots,x_N)= \left[\Delta_N(\mathbf{x})^{-1}\circ\left(\sum_{i=1}^N \mathsf{L}_{x_i}^{(N)}\right)\circ\Delta_N(\mathbf{x})-\lambda_N\right]g(x_1,\dots,x_N),
\end{align*}
which is exactly the Doob $h$-transform, see \cite{Revuz-Yor,PinskyBook,Doob}, of $N$ independent one-dimensional diffusions with generator $\mathsf{L}^{(N)}$ killed when they intersect, by the Vandermonde determinant $\Delta_N$. This is the Markov process which has transition kernel having density (\ref{transitioDensity}), by virtue of the Karlin-McGregor formula \cite{KarlinMcGregor,ItoMcKean}, and the conclusion follows. 
\end{proof}

\begin{rmk}
For $N=1$, the fact that $C_{c}^\infty(\mathbb{R}_+)$ is a core for the $\mathsf{L}^{(1)}$-diffusion is well-known and a very special case of general results, see Chapter 8 of \cite{Ethier-Kurtz}.     
\end{rmk}

\subsection{Consistency}\label{SubSec-Consistency}

The following consistency or intertwining relation between $\left(\mathfrak{P}_{N}(t)\right)_{t\ge 0}$ and $\left(\mathfrak{P}_{N+1}(t)\right)_{t\ge0}$ is  the main algebraic ingredient for the method of intertwiners.
\begin{thm}\label{thm-intertwiningSemigroup}
         We have the intertwining relations:
         \begin{align}\label{intertwining-semigroup}
           \mathfrak{P}_{N+1}(t)\Lambda_N^{N+1}=\Lambda_N^{N+1}\mathfrak{P}_{N}(t),\ \ \forall t\geq0,\ \forall N\in\mathbb{N}.
        \end{align}
     \end{thm}

We present two arguments for a proof. The first argument is a little more sophisticated and uses matrix stochastic calculus. It is conceptually more natural as it somehow explains the origin of the intertwining: the intertwining (\ref{intertwining-semigroup}) simply follows from consistency with respect to corners projection at the level of the matrix diffusions. On the other hand, we only give a full proof for the special initial condition $\mathbf{H}^{(N)}_0=\mathbf{0}$ which allows for important simplifications (in the case of Dyson Brownian motion which is also consistent with $\Lambda_N^{N+1}$ the adaptation of the argument works for any initial condition).  The second argument is a direct verification, which works for all initial conditions, and thus is a complete proof, using the explicit formula from (\ref{transitioDensity}).

We begin with a little formalism giving consistency at the level of the eigenvalues from consistency at the level of the matrix process.

    \begin{prop}
        Let $\left(\mathbf{C}_t^{(N)};t\geq 0\right)$ and $\left(\mathbf{C}_t^{(N+1)};t\geq 0\right)$ be stochastic processes on $\mathbb{H}(N)$ and $\mathbb{H}(N+1)$ respectively. Suppose that $\left(\mathsf{eval}_{N}\left(\mathbf{C}_t^{(N)}\right);t\geq 0\right)$ and $\left(\mathsf{eval}_{N+1}\left(\mathbf{C}_t^{(N+1)}\right);t\geq 0\right)$ are Markovian with semigroups $\left(\mathsf{P}_{N}(t)\right)_{t\geq 0}$ and $\left(\mathsf{P}_{N+1}(t)\right)_{t\geq 0}$ respectively. Assume that, for all $t\geq 0$,
       \begin{align}\label{U-inv}
          \left(\mathbf{U}^{(K)}\right)^{\dag}\mathbf{C}_t^{(K)}
          \mathbf{U}^{(K)}\overset{\textnormal{d}}{=}\mathbf{C}_t^{(K)},\ \ \forall\mathbf{U}^{(K)}\in\mathbb{U}(K),\  \  K=N,N+1.
        \end{align}
        Suppose moreover that there exists a set $\mathscr{A}\subseteq\mathbb{H}(N+1)$ such that for all $\mathbf{C}_0^{(N+1)}\in\mathscr{A}$,
        \begin{equation}\label{ProjProp}
\Pi_N^{N+1}\left(\mathbf{C}_t^{(N+1)}\right)\overset{\textnormal{d}}{=}\mathbf{C}_t^{(N)}, \ \ \forall t \ge 0.
        \end{equation}
        Then, we have
        \begin{align}\label{P-tLambda=LambdaP-t}
           \mathsf{P}_{N+1}(t)\Lambda_N^{N+1}(\mathbf{x},\cdot)=\Lambda_N^{N+1}\mathsf{P}_{N}(t)(\mathbf{x},\cdot),\ \ \forall \mathbf{x}\in\mathsf{eval}_{N+1}(\mathscr{A}),\ \ \forall t\geq 0.
        \end{align}
    \end{prop} 
    \begin{proof}
    Fix $\mathbf{x}$ so that $\mathbf{x}=\mathsf{eval}_{(N+1)}\left(\mathbf{C}_0^{(N+1)}\right)$, with $\mathbf{C}_0^{(N+1)}\in\mathscr{A}$. We claim that, the LHS of \eqref{P-tLambda=LambdaP-t} is  given by,
    \begin{equation*}
\mathsf{Law}\left(\mathsf{eval}_{N}\left(\Pi_N^{N+1}\left(\mathbf{C}_t^{(N+1)}\right)\right)\right),
    \end{equation*}
     while the RHS by,
     \begin{equation*}
\mathsf{Law}\left(\mathsf{eval}_{N}\left(\mathbf{C}_t^{(N)}\right)\right).
     \end{equation*}
     These are equal by virtue of \eqref{ProjProp} and the conclusion follows. To see the claim, first observe that, 
     \begin{align*}
    \mathsf{Law}\left(\mathsf{eval}_{N}\left(\mathbf{C}_t^{(N)}\right)\right)(\mathrm{d}\mathbf{y})=\mu\mathsf{P}_N(t)(\mathrm{d}\mathbf{y}), 
     \end{align*}
     where, by unitary invariance \eqref{U-inv}, $\mu$ is given by,
     \begin{equation*}
   \mu(\mathrm{d}\mathbf{y})\overset{\textnormal{def}}{=}\mathsf{Law}\left(\mathsf{eval}_N\left(\mathbf{C}_0^{(N)}\right)\right)(\mathrm{d}\mathbf{y})=\mathsf{Law}\left(\mathsf{eval}_N\left(\Pi_N^{N+1}\left(\mathbf{C}_0^{(N+1)}\right)\right)\right)(\mathrm{d}\mathbf{y})=\Lambda_N^{N+1}\left(\mathbf{x},\mathrm{d}\mathbf{y}\right).
     \end{equation*}
     On the other hand, again by unitary invariance \eqref{U-inv}, we also have, 
\begin{equation*}
\mathsf{Law}\left(\mathsf{eval}_{N}\left(\Pi_N^{N+1}\left(\mathbf{C}_t^{(N+1)}\right)\right)\right)(\mathrm{d}\mathbf{y})=\mathsf{Law}\left(\mathsf{eval}_{N+1}\left(\mathbf{C}_t^{(N+1)}\right)\right)\Lambda_N^{N+1}(\mathrm{d}\mathbf{y})=\nu\mathsf{P}_{N+1}(t)\Lambda_N^{N+1}(\mathrm{d}\mathbf{y}),
\end{equation*}
where $\nu$ is given by,
\begin{equation*}
\nu(\mathrm{d}\mathbf{y})\overset{\textnormal{def}}{=}\mathsf{Law}\left(\mathsf{eval}_{N+1}\left(\mathbf{C}_0^{(N+1)}\right)\right)=\mathbf{1}_{\mathbf{x}}(\mathrm{d}\mathbf{y}).
\end{equation*}
 This completes the proof.
    \end{proof}
    
     We will apply this proposition (with the obvious notational identifications) to prove Theorem \ref{thm-intertwiningSemigroup} in the special case $\mathbf{H}_0^{(N)}=\mathbf{0}$, namely for $\mathbf{x}=\mathbf{0}$.
     
     \begin{thm}\label{thm-matrix consistency}
     Let $\left(\mathbf{H}_t^{(N)}; t\geq 0\right)$ be the unique strong solution of the equation \eqref{MSDE} starting from $\mathbf{H}_0^{(N)}=\mathbf{0}$, where $\mathbf{0}$ denotes $N\times N$ zero matrix. We have
    \begin{align}\label{HN=corner HN+1}
\mathbf{H}_t^{(N)}\overset{\textnormal{d}}{=}\Pi_N^{N+1}\left(\mathbf{H}_t^{(N+1)}\right),\  \  \forall t\geq 0.
\end{align}
\end{thm}

We first  show that the solution to the SDE \eqref{MSDE} can be solved explicitly in terms of the matrix exponential Brownian motion $\left(\mathbf{M}_t^{(\theta)};t\ge 0\right)$, which we introduce next. This is given by the unique strong solution of the matrix SDE: 
    \begin{align}\label{Mt SDE}
    \mathrm{d}\mathbf{M}_t^{(\theta)} =\frac{1}{2} \mathbf{M}_t^{(\theta)}\mathrm{d}\mathbf{W}_t^{(N)} -\theta\mathbf{M}_t^{(\theta)}\mathrm{d}t,
\end{align}
where $\theta \in \mathbb{R}$ and $\left(\mathbf{W}^{(N)}_t;t\ge 0\right)$ is an $N\times N$ complex matrix-valued Brownian motion.  Below we use the notation $\mathbf{A}^{-\dag}$ for $\left(\mathbf{A}^\dag\right)^{-1}$.

\begin{prop}\label{prop-Ht-Mt}
The unique strong solution $\left(\mathbf{H}_t^{(N)};t\ge 0\right)$ of \eqref{MSDE}, starting from $\mathbf{H}^{(N)}_0\in\mathbb{H}_+(N)$ is given by
\begin{align}\label{Ht-H0Mt}
  \left(\left(\mathbf{M}_t^{(-\nu_N)}\right)^{-1} \left(\mathbf{H}_0 + \frac{1}{2} \int_{0}^{t}\mathbf{M}_s^{(-\nu_N)}\left(\mathbf{M}_s^{(-\nu_N)}\right)^{\dag}\mathrm{d}s\right)\left(\mathbf{M}_t^{(-\nu_N)}\right)^{-\dag};t\ge 0\right).
\end{align}
where $\nu_N=\frac{1}{4}(\eta+N)$, with $\eta$ being the same parameter as throughout this paper.
\end{prop}

\begin{proof}

First, note that from \eqref{Mt SDE} we have
    \begin{align*}
    \mathrm{d}\left(\mathbf{M}_t^{(\theta)}\right)^{\dag} =\frac{1}{2} \mathrm{d}\left(\mathbf{W}_t^{(N)}\right)^{\dag}\left(\mathbf{M}_t^{(\theta)}\right)^{\dag} -\theta \left(\mathbf{M}_t^{(\theta)}\right)^{\dag}\mathrm{d}t.
\end{align*}
Moreover, using It\^{o}'s formula we obtain 
    \begin{align*}
    \mathrm{d}\left(\mathbf{M}_t^{(\theta)}\right)^{-1} &= - \frac{1}{2}\mathrm{d}\mathbf{W}_t^{(N)}\left(\mathbf{M}_t^{(\theta)}\right)^{-1} +\theta \left(\mathbf{M}_t^{(\theta)}\right)^{-1}\mathrm{d}t\\
    \mathrm{d}\left(\mathbf{M}_t^{(\theta)}\right)^{-\dag} &= -\frac{1}{2}\left(\mathbf{M}_t^{(\theta)}\right)^{-\dag}\mathrm{d}\left(\mathbf{W}_t^{(N)}\right)^{\dag} +\theta \left(\mathbf{M}_t^{(\theta)}\right)^{-\dag}\mathrm{d}t.
\end{align*}
One can then check directly that \eqref{Ht-H0Mt} solves the SDE \eqref{MSDE} by applying It\^{o}'s formula to the expression in \eqref{Ht-H0Mt} and using the equations above.
\end{proof}

\begin{prop}\label{prop H-MM*}
In the setting of Theorem \ref{thm-matrix consistency}, we have, for fixed $t\ge 0$,
\begin{align}\label{H=int MM*}
    \mathbf{H}_t^{(N)}\overset{\textnormal{d}}{=}\frac{1}{2} \int_{0}^{t}\mathbf{M}_s^{(\nu_N)}\left(\mathbf{M}_s^{(\nu_N)}\right)^{\dag}\mathrm{d}s,
\end{align}
where $\nu_N=\frac{1}{4}(\eta+N)$, with $\eta$ being the same parameter as throughout this paper.
\end{prop}

\begin{proof}
For any fixed $t\geq 0$, define 
\begin{equation*}
 \left(\mathbf{N}_s^{(\nu_N)};0\leq s\leq t\right)\overset{\textnormal{def}}{=}\left(\left(\mathbf{M}_t^{(-\nu_N)}\right)^{-1}\mathbf{M}_{t-s}^{(-\nu_N)};0\leq s\leq t\right).   
\end{equation*}
Then, making the change of variable $u\to t-u$ in \eqref{Ht-H0Mt}, one can write
\begin{align}\label{Ht-Nt}
   \mathbf{H}_t^{(N)}\overset{\textnormal{d}}{=}\left(\mathbf{M}_t^{(-\nu_N)}\right)^{-1} \mathbf{H}_0 \left(\mathbf{M}_t^{(-\nu_N)}\right)^{-\dag}+ \frac{1}{2}\int_{0}^{t}\mathbf{N}_u^{(\nu_N)}\left(\mathbf{N}_u^{(\nu_N)}\right)^{\dag}\mathrm{d}u.
\end{align}
From \eqref{Ht-Nt}, we have when $\mathbf{H}_0^{(N)}=\mathbf{0}$,
\begin{align}
        \mathbf{H}_t^{(N)}\overset{\textnormal{d}}{=} \frac{1}{2}\int_{0}^{t}\mathbf{N}_u^{(\nu_N)}\left(\mathbf{N}_u^{(\nu_N)}\right)^{\dag}\mathrm{d}u.
\end{align}
Moreover, it is shown in Lemma 2.1 of \cite{Bougerol} (essentially the matrix analogue of the time-reversal property of Brownian motion) that,
\begin{align*}
\left(\mathbf{N}_u^{(\nu_N)};0\leq u\leq t\right)\overset{\textnormal{d}}{=}\left(\mathbf{M}_u^{(\nu_N)};0\leq u\leq t\right),
    \end{align*}
    which implies the result.
\end{proof}

\begin{proof}[Proof of Theorem \ref{thm-matrix consistency}]
By virtue of \eqref{H=int MM*}, it suffices to show that 
\begin{align*}
\int_{0}^{t}\mathbf{M}_s^{(\nu_N)}\left(\mathbf{M}_s^{(\nu_N)}\right)^{\dag}\mathrm{d}s\overset{\textnormal{d}}{=}\Pi_N^{N+1}\left(\int_{0}^{t}\mathbf{M}_s^{(\nu_{N+1})}\left(\mathbf{M}_s^{(\nu_{N+1})}\right)^{\dag}\mathrm{d}s\right),\  \  \forall t\geq 0,\  \  N\in\mathbb{N}.
\end{align*}
By applying It\^{o}'s formula we have from \eqref{Mt SDE} with $\theta=\nu_N$,
\begin{align}\label{SDE for N-dim process}
   \mathrm{d}\left(\mathbf{M}_t^{(\nu_N)}\left(\mathbf{M}_t^{(\nu_N)}\right)^{\dag}\right)&= \mathrm{d}\mathbf{M}_t^{(\nu_N)}\left(\mathbf{M}_t^{(\nu_N)}\right)^{\dag}+\mathbf{M}_t^{(\nu_N)}\mathrm{d}\left(\mathbf{M}_t^{(\nu_N)}\right)^{\dag}+
   \mathrm{d}\mathbf{M}_t^{(\nu_N)}\mathrm{d}\left(\mathbf{M}_t^{(\nu_N)}\right)^{\dag}\nonumber\\
&=\frac{1}{2}\mathbf{M}_t^{(\nu_N)}\left(\mathrm{d}\mathbf{W}_t^{(N)}+\mathrm{d}\left(\mathbf{W}_t^{(N)}\right)^{\dag}\right)\left(\mathbf{M}_t^{(\nu_N)}\right)^{\dag}+\left(\frac{N}{2}-2\nu_N\right)\mathbf{M}_t^{(\nu_N)}\left(\mathbf{M}_t^{(\nu_N)}\right)^{\dag}\mathrm{d}t.
\end{align}
Now, using Lévy's characterization theorem, see \cite{Revuz-Yor}, one 
can easily check that
\begin{align}\label{hat{W}}
    \mathrm{d}\hat{\mathbf{W}}_t^{(N)}\overset{\textnormal{def}}{=}\left(\mathbf{M}_t^{(\nu_N)}\left(\mathbf{M}_t^{(\nu_N)}\right)^{\dag}\right)^{-\frac{1}{2}}\mathbf{M}_t^{(\nu_N)}\mathrm{d}\mathbf{W}_t^{(N)}\left(\mathbf{M}_t^{(\nu_N)}\right)^{\dag}\left(\mathbf{M}_t^{(\nu_N)}\left(\mathbf{M}_t^{(\nu_N)}\right)^{\dag}\right)^{-\frac{1}{2}},
\end{align}
 is an $N\times N$ matrix complex Brownian motion. Thus, the equation (\ref{SDE for N-dim process}) can be written in the following closed form:
\begin{align}\label{closed SDE for N-dim process}
   \mathrm{d}\left(\mathbf{M}_t^{(\nu_N)}\left(\mathbf{M}_t^{(\nu_N)}\right)^{\dag}\right)&=\frac{1}{2} \sqrt{\left(\mathbf{M}_t^{(\nu_N)}\left(\mathbf{M}_t^{(\nu_N)}\right)^{\dag}\right)}\left(\mathrm{d}\hat{\mathbf{W}}_t^{(N)}+\mathrm{d}\left(\hat{\mathbf{W}}_t^{(N)}\right)^{\dag}\right)\sqrt{\mathbf{M}_t^{(\nu_N)}\left(\mathbf{M}_t^{(\nu_N)}\right)^{\dag}}\nonumber
   \\& \  \ +\left(\frac{N}{2}-2\nu_N\right)\mathbf{M}_t^{(\nu_N)}\left(\mathbf{M}_t^{(\nu_N)}\right)^{\dag}\mathrm{d}t.
\end{align}
We next consider the SDE for the $(N+1) \times (N+1)$ matrix process, that we can write as follows
 \begin{align}\label{N+1-dim GBm Matrix}
   \mathrm{d}\mathbf{M}_t^{(\nu_{N+1})} &= \frac{1}{2}\mathbf{M}_t^{(\nu_{N+1})}\mathrm{d}\mathbf{W}_t^{(N+1)} -\nu_{N+1} \mathbf{M}_t^{(\nu_{N+1})}\mathrm{d}t \nonumber\\
    &= \frac{1}{2}\begin{pmatrix}
        \tilde{\mathbf{M}}_t^{(N)} & \mathbf{a}_t\\
        \mathbf{b}_t^\textnormal{T} & c_t
    \end{pmatrix} \mathrm{d}\begin{pmatrix}
        \Tilde{\mathbf{W}}_t^{(N)} & \mathbf{u}_t\\
        \mathbf{v}_t^\textnormal{T} & w_t
    \end{pmatrix} 
    - \nu_{N+1}\begin{pmatrix}
        \tilde{\mathbf{M}}_t^{(N)} & \mathbf{a}_t\\
        \mathbf{b}_t^\textnormal{T} & c_t
    \end{pmatrix} \mathrm{d}t,
\end{align}
where $\tilde{\mathbf{M}}_t^{(N)}$ is the $N\times N$ top-left corner of $\mathbf{M}_t^{(\nu_{N+1})}$, $\mathbf{a}_t$ and $\mathbf{b}_t$ are 
$\mathbb{C}^{N}$-valued, and $c_t\in\mathbb{C}$. Here,
$\tilde{\mathbf{W}}_t^{(N)}$, $\mathbf{u}_t,\mathbf{v}_t $, and $ w_t$ are complex matrix, complex vector, and $\mathbb{C}$-valued Brownian motions and $\mathbf{A}^\textnormal{T}$ denotes the transpose of $\mathbf{A}$.  One can see from the above equation that
\begin{align}\label{(tilde M_t a_t)}
    \mathrm{d}\begin{pmatrix}
        \tilde{\mathbf{M}}_t^{(N)} & \mathbf{a}_t
    \end{pmatrix}=\frac{1}{2}\begin{pmatrix}
        \tilde{\mathbf{M}}_t^{(N)} & \mathbf{a}_t
    \end{pmatrix}
    \mathrm{d}\mathbf{W}_t^{(N+1)}
    - \nu_{N+1}\begin{pmatrix}
        \tilde{\mathbf{M}}_t^{(N)} & \mathbf{a}_t
    \end{pmatrix} \mathrm{d}t.
\end{align}
We also have that for the $N\times N$ top-left corner of $\mathbf{M}_t^{(\nu_{N+1})}\left(\mathbf{M}_t^{(\nu_{N+1})}\right)^{\dag}$: 
\begin{align}
\Pi_N^{N+1}\left(\mathbf{M}_t^{(\nu_{N+1})}\left(\mathbf{M}_t^{(\nu_{N+1})}\right)^{\dag}\right)= \begin{pmatrix}
        \tilde{\mathbf{M}}_t^{(N)} & \mathbf{a}_t
    \end{pmatrix}\begin{pmatrix}
        \tilde{\mathbf{M}}_t^{(N)} & \mathbf{a}_t
    \end{pmatrix}^{\dag}. \label{ProjMatrixSDE}
\end{align}

Therefore,
by
It\^{o}'s formula we get 
\begin{align}
   \mathrm{d}\Pi_N^{N+1}\left(\mathbf{M}_t^{(\nu_{N+1})}\left(\mathbf{M}_t^{(\nu_{N+1})}\right)^{\dag}\right)&= 
    \mathrm{d}\begin{pmatrix}
        \tilde{\mathbf{M}}_t^{(N)} & \mathbf{a}_t
    \end{pmatrix}\begin{pmatrix}
        \tilde{\mathbf{M}}_t^{(N)} & \mathbf{a}_t
    \end{pmatrix}^{\dag}+
  \begin{pmatrix}
        \tilde{\mathbf{M}}_t^{(N)} & \mathbf{a}_t
    \end{pmatrix}\mathrm{d}\begin{pmatrix}
        \tilde{\mathbf{M}}_t^{(N)} & \mathbf{a}_t
    \end{pmatrix}^{\dag} \nonumber
   \\&\  \  + \mathrm{d}\begin{pmatrix}
        \tilde{\mathbf{M}}_t^{(N)} & \mathbf{a}_t
    \end{pmatrix}\mathrm{d}\begin{pmatrix}
        \tilde{\mathbf{M}}_t^{(N)} & \mathbf{a}_t
    \end{pmatrix}
    ^{\dag}.
\end{align}
Using \eqref{(tilde M_t a_t)}, we then obtain
\begin{align}\label{cornerSDE}
   \mathrm{d}\Pi_N^{N+1}\left(\mathbf{M}_t^{(\nu_{N+1})}\left(\mathbf{M}_t^{(\nu_{N+1})}\right)^{\dag}\right) & = \frac{1}{2}\begin{pmatrix}
        \tilde{\mathbf{M}}_t^{(N)} & \mathbf{a}_t
    \end{pmatrix} 
        \left( \mathrm{d}\mathbf{W}_t^{(N+1)}+ \mathrm{d}\left(\mathbf{W}_t^{(N+1)}\right)
^{\dag}\right)
    \begin{pmatrix}
        \tilde{\mathbf{M}}_t^{(N)} & \mathbf{a}_t
    \end{pmatrix}^{\dag}\nonumber\\
    &\  \ + \left(\frac{N+1}{2}-2\nu_{N+1}\right)\begin{pmatrix}
        \tilde{\mathbf{M}}_t^{(N)} & \mathbf{a}_t
    \end{pmatrix}\begin{pmatrix}
        \tilde{\mathbf{M}}_t^{(N)} & \mathbf{a}_t
    \end{pmatrix}^{\dag}\mathrm{d}t.
\end{align}
Similarly to \eqref{hat{W}}, by virtue of Levy's characterization theorem, we can define the $N\times N$ matrix-valued complex Brownian motion $\left(\hat{\mathbf{W}}_t^{(N)}\right)_{t\ge 0}$ (different from \eqref{hat{W}}) by the formula
\begin{align*}
\mathrm{d}\hat{\mathbf{W}}_t^{(N)}\overset{\textnormal{def}}{=}\nonumber \left(\begin{pmatrix}
        \tilde{\mathbf{M}}_t^{(N)} & \mathbf{a}_t
    \end{pmatrix}\begin{pmatrix}
        \tilde{\mathbf{M}}_t^{(N)} & \mathbf{a}_t
    \end{pmatrix}^{\dag}\right)^{-\frac{1}{2}}\begin{pmatrix}
        \tilde{\mathbf{M}}_t^{(N)} & \mathbf{a}_t
\end{pmatrix}\mathrm{d}\mathbf{W}_t^{(N+1)}\times \\ \begin{pmatrix}
        \tilde{\mathbf{M}}_t^{(N)} & \mathbf{a}_t
    \end{pmatrix}^{\dag}\left(\begin{pmatrix}
        \tilde{\mathbf{M}}_t^{(N)} & \mathbf{a}_t
    \end{pmatrix}\begin{pmatrix}
        \tilde{\mathbf{M}}_t^{(N)} & \mathbf{a}_t
    \end{pmatrix}^{\dag}\right)^{-\frac{1}{2}}.
\end{align*}
Using these we represent the SDE \eqref{cornerSDE} in the following closed form, recalling (\ref{ProjMatrixSDE}),
\begin{align}\label{closed SDE of the corner}
   &\mathrm{d}\Pi_N^{N+1}\left(\mathbf{M}_t^{(\nu_{N+1})}\left(\mathbf{M}_t^{(\nu_{N+1})}\right)^{\dag}\right) = \frac{1}{2}
        \sqrt{\begin{pmatrix}
        \tilde{\mathbf{M}}_t^{(N)} & \mathbf{a}_t
    \end{pmatrix}\begin{pmatrix}
        \tilde{\mathbf{M}}_t^{(N)} & \mathbf{a}_t
    \end{pmatrix}^{\dag}}
        \left(\mathrm{d}\hat{\mathbf{W}}_t^{(N)}+ \mathrm{d}\left(\hat{\mathbf{W}}_t^{(N)}\right)
^{\dag}\right)\times
    \nonumber \\ &\hspace{1.5cm} \sqrt{\begin{pmatrix}
        \tilde{\mathbf{M}}_t^{(N)} & \mathbf{a}_t
    \end{pmatrix}\begin{pmatrix}
        \tilde{\mathbf{M}}_t^{(N)} & \mathbf{a}_t
    \end{pmatrix}^{\dag}}
    + \left(\frac{N+1}{2}-2\nu_{N+1}\right)\begin{pmatrix}
        \tilde{\mathbf{M}}_t^{(N)} & \mathbf{a}_t
    \end{pmatrix}\begin{pmatrix}
        \tilde{\mathbf{M}}_t^{(N)} & \mathbf{a}_t
    \end{pmatrix}^{\dag}\mathrm{d}t.
\end{align}
Now, comparing the SDEs \eqref{closed SDE for N-dim process} and 
\eqref{closed SDE of the corner}
and considering the definition of the parameter $\nu_N$, it follows that, both 
$\left(\mathbf{M}_t^{(\nu_N)}\left(\mathbf{M}_t^{(\nu_N)}\right)^{\dag};t \ge 0\right)$
and
$\left(\Pi_N^{N+1}\left(\mathbf{M}_t^{(\nu_{N+1})}\left(\mathbf{M}_t^{(\nu_{N+1})}\right)^{\dag}\right);t \ge 0\right)$
satisfy the same equation, having a unique weak solution, and hence, these processes have the same distribution. This in particular implies 
that 
\begin{align*}
    \int_{0}^{t}\mathbf{M}_s^{(\nu_N)}\left(\mathbf{M}_s^{(\nu_N)}\right)^{\dag}\mathrm{d}s\overset{\textnormal{d}}{=}\Pi_N^{N+1}\left(\int_{0}^{t}\mathbf{M}_s^{(\nu_{N+1})}\left(\mathbf{M}_s^{(\nu_{N+1})}\right)^{\dag}\mathrm{d}s\right),\  \  N\in\mathbb{N},
\end{align*}
 for any fixed $t>0$, as desired.
\end{proof}
We now give a direct proof for Theorem \ref{thm-intertwiningSemigroup} for any initial condition.

\begin{proof}[Proof of Theorem \ref{thm-intertwiningSemigroup}]
      We will show  
        \begin{align}\label{P-Lambda-intertwining}
           \mathfrak{P}_{N+1}(t)\Lambda_N^{N+1}(\mathbf{x},\cdot)=\Lambda_N^{N+1}\mathfrak{P}_{N}(t)(\mathbf{x},\cdot),\  \ \forall \mathbf{x}\in\mathbb{W}_{N+1,+}^{\circ},\  \ \forall t\geq 0.
        \end{align}
        The case of general $\mathbf{x}\in \mathbb{W}_{N+1,+}$ then follows by the Feller property, see Proposition \ref{prop-FellerProperty}. We have using the explicit expressions in \eqref{transitioDensity} and \eqref{link}, that (\ref{P-Lambda-intertwining}) boils down to showing:
\begin{align}\label{P-Lambda-intertwining formula}
    \int_{\mathbb{W}_{N+1,+}}\mathrm{e}^{-\lambda_{N+1}t}\det\left(\mathfrak{p}^{N+1}_t(x_i,y_j)\right)_{i,j=1}^{N+1}\mathbf{1}_{\mathbf{z}\prec\mathbf{y}}\mathrm{d}\mathbf{y}=
        \int_{\mathbb{W}_{N,+}}\mathrm{e}^{-\lambda_Nt}\mathbf{1}_{\mathbf{y}\prec\mathbf{x}}\det\left(\mathfrak{p}^{N}_t(y_i,z_j)\right)_{i,j=1}^{N}\mathrm{d}\mathbf{y}.
\end{align}
We write $\mathbf{x}\tilde{\prec}\mathbf{y}$ to denote $ y_1\geq x_1> y_2\geq x_2>\dots\geq x_N> y_{N+1}$, when $\mathbf{x}\in\mathbb{W}_N$ and $\mathbf{y}\in\mathbb{W}_{N+1}$.
Then, we recall the well-known identity
$\mathbf{1}_{\mathbf{x}\tilde{\prec}\mathbf{y}}= \det\left(\varphi(x_i,y_j)\right)_{i,j=1}^{N+1}$ with 

\begin{align*}
\varphi(x,y)=\begin{cases}
-\mathbf{1}_{y>x}, \  \  & x,y\in\mathbb{R}_+,\\
 1, \  \  & x=\textnormal{virt},
\end{cases}
\end{align*}
where we use the convention $x_{N+1}=\textnormal{virt}$ for a ``virtual variable". Note that, we can replace $\mathbf{1}_{\mathbf{y}\prec\mathbf{x}}$ with $\mathbf{1}_{\mathbf{y}\tilde{\prec}\mathbf{x}}$ in \eqref{P-Lambda-intertwining formula}.
Accordingly, with this replacement, the LHS of \eqref{P-Lambda-intertwining formula}, by the Andreief identity, is equal to 
\begin{align*}
    \mathrm{e}^{-\lambda_{N+1}t}\det\begin{pmatrix}
        [\mathfrak{p}_t^{(N+1)}\varphi^{*}](x_1,z_1)& [\mathfrak{p}_{t}^{(N+1)}\varphi^{*}](x_1,z_2)& \cdots & 1\\
         [\mathfrak{p}_{t}^{(N+1)}\varphi^{*}](x_2,z_1)& [\mathfrak{p}_{t}^{(N+1)}\varphi^{*}](x_2,z_2)& \cdots & 1\\
        \vdots&\vdots& &\vdots\\  [\mathfrak{p}_{t}^{(N+1)}(t)\varphi^{*}](x_{N+1},z_1)& [\mathfrak{p}_t^{(N+1)}\varphi^{*}](x_{N+1},z_2)& \cdots & 1
    \end{pmatrix},
\end{align*}
where $\varphi^{*}(x,y)=\varphi(y,x)$, and by definition $[\mathfrak{p}_t^{(N+1)}\varphi^*](x,y)$ denotes compositition of kernels so that the entries of the matrix above are as follows:
\begin{align*}
    [\mathfrak{p}_t^{(N+1)}\varphi^*](x_i,z_j)=\int_{0}^{\infty} \mathfrak{p}_t^{(N+1)}(x_i,y)\varphi^{*}(y,z_j)\mathrm{d}y=-\int_{z_j}^{\infty}\mathfrak{p}_t^{(N+1)}(x_{i},y)\mathrm{d}y,\  \ \forall i,j=1,\dots,N+1.
\end{align*}
For the RHS, we first expand the $\left(N+1\right)\times \left(N+1\right)$ determinant corresponding to $\mathbf{1}_{\mathbf{x}\tilde{\prec}\mathbf{y}}$. We next use the Andreief identity, and then again recombine terms to obtain 
\begin{align}\label{RHS-intertwinig}
   \sum_{l=0}^{N+1}(-1)^{N+1-l} \mathrm{e}^{-\lambda_{N}t}&\det\left([\varphi^{*}\mathfrak{p}_t^{(N)}](x_i,z_j)\right)_{i,j=1}^{N+1}=
   \nonumber
   \\&  \mathrm{e}^{-\lambda_{N}t}
    \det\begin{pmatrix}
        [\varphi^{*}\mathfrak{p}_t^{(N)}](x_1,z_1)& [\varphi^{*}\mathfrak{p}_t^{(N)}](x_1,z_2)& \cdots & 1\\
        [\varphi^{*}\mathfrak{p}_t^{(N)}](x_2,z_1)& [\varphi^{*}\mathfrak{p}_t^{(N)}](x_2,z_2)& \cdots & 1\\
        \vdots&\vdots& &\vdots\\  [\varphi^{*}\mathfrak{p}_t^{(N)}](x_{N+1},z_1)& [\varphi^{*}\mathfrak{p}_t^{(N)}](x_{N+1},z_2)& \cdots & 1
    \end{pmatrix},
\end{align}
where 
\begin{align*}
    [\varphi^{*}\mathfrak{p}_t^{(N)}](x_{i},z_j)=\int_{0}^{\infty} \varphi^{*}(x_i,y)\mathfrak{p}_t^{(N)}(y,z_j)\mathrm{d}y=-\int_{0}^{x_i}\mathfrak{p}_t^{(N)}(y,z_j)\mathrm{d}y,\  \ \forall i,j=1,\dots,N+1.
\end{align*}
We now claim that
\begin{align}\label{pt^N=int qt}
    \mathfrak{p}_t^{(N)}(x,z)=\mathrm{e}^{\left(\frac{\eta}{2}+N\right)t}\partial_x\int_z^{\infty}\mathfrak{p}_t^{(N+1)}(x,y)\mathrm{d}y\overset{\textnormal{def}}{=}q_t(x,y).
\end{align}
 We first note, the key algebraic consistency relation
$\partial_x\mathsf{L}_x^{(N+1)}=\mathsf{L}_x^{(N)}\partial_x+(-\frac{\eta}{2}-N)\partial_x$.
Making use of this, it can be checked that $q_t(x,y)$ is the fundamental solution to the Kolmogorov backward equation associated to the one-dimensional diffusion
$\mathsf{L}_x^{(N)}$, which is unique, and hence, \eqref{pt^N=int qt} holds. 
Thus, we have
\begin{align}
    -\int_0^{x_i}\mathfrak{p}_t^{(N)}(y,z_j)\mathrm{d}y=\mathrm{e}^{\left(\frac{\eta}{2}+N\right)t}\left[\int_{z_j}^{\infty}\mathfrak{p}_t^{(N+1)}(0,y)\mathrm{d}y-\int_{z_j}^{\infty}\mathfrak{p}_t^{(N+1)}(x_i,y)\mathrm{d}y\right],\  \forall i,j=1,\dots,N+1.
\end{align}
One can now substitute this in the determinant in \eqref{RHS-intertwinig} and use column operations to show the desired identity, by noting that $\lambda_{N+1}-\lambda_N=N(-\frac{\eta}{2}-N)$.
\end{proof}

\subsection{Invariant measure and convergence to equilibrium}\label{SubSec-Ergodicity}
The inverse Laguerre ensemble $\mathfrak{M}_N^{\eta}$
of parameter $\eta>-1$ is the probability measure on $\left(\mathbb{W}_{N,+}\right)_{N=1}^{\infty}$ given by 
	\begin{align}\label{ILM}
		\mathfrak{M}_N^{\eta}(\mathrm{d}\mathbf{x})=\frac{1}{\mathcal{Z}_N^{(\eta)}}\Delta_N^2(\textbf{x})\prod_{i=1}^{N}x_i^{-\eta-2N}\mathrm{e}^{-\frac{1}{x_i}}\mathrm{d}\mathbf{x},
	\end{align}
for some normalisation constant $\mathcal{Z}_N^{(\eta)}$. We now show that the solution $(\mathfrak{x}^{(N)}(t);t\geq 0)$ of (\ref{FSDE}) started from arbitrary initial condition, converges to $\mathfrak{M}_N^{\eta}$.

\begin{prop}\label{prop-ConvToEquilibrium}
    Let $\eta>-1$ and consider the unique strong solution $(\mathfrak{x}^{(N)}(t);t\geq 0)$ of (\ref{FSDE}). Suppose $\mathfrak{K} \in \mathscr{M}_\textnormal{p}(\mathbb{W}_{N,+})$  and that $\mathsf{Law}(\mathfrak{x}^{(N)}(0))=\mathfrak{K}$. Then,  as $t\to\infty$,
\begin{align}\label{ConvToEqFiniteN}
\mathfrak{x}^{(N)}(t) \overset{\textnormal{d}}{\longrightarrow} \mathbf{Z}, \ \ \textnormal{where } \mathsf{Law}(\mathbf{Z})=\mathfrak{M}_N^{\eta}.
\end{align} 
In particular, $\mathfrak{M}_N^{\eta}$ is the unique invariant measure of $(\mathfrak{P}_N(t))_{t\ge 0}$.
\end{prop}
\begin{proof}
We use the corresponding matrix process to prove the result.
Recall that from \eqref{Ht-Nt}, the unique strong solution $(\mathbf{H}_t^{(N)};t\ge 0)$ 
starting from $\mathbf{H}_0^{(N)}\in\mathbb{H}_+(N)$ can be written as 
\begin{align}\label{Ht-Mt}
\left(\mathbf{M}_t^{(\nu_N)}\right)^{-1} \mathbf{H}^{(N)}_0 \left(\mathbf{M}_t^{(-\nu_N)}\right)^{-\dag}+ \frac{1}{2}\int_{0}^{t}\mathbf{N}_u^{(\nu_N)}\left(\mathbf{N}_u^{(\nu_N)}\right)^{\dag}\mathrm{d}u.
\end{align}
Let us take $\mathbf{H}_0^{(N)}$ such that $\mathsf{Law}\left(\mathsf{eval}_N\left(\mathbf{H}_0^{(N)}\right)\right)=\mathfrak{K}$ (and independent of the driving Brownian motions). Then, from Lemma 2.2 in \cite{Bougerol} the first term converges to $0$ almost surely as $t\to\infty$. Hence, we obtain
\begin{align}
   \mathbf{H}_t^{(N)}\overset{\textnormal{d}}{\longrightarrow} \frac{1}{2}\int_{0}^{\infty}\mathbf{N}_u^{(\nu_N)}\left(\mathbf{N}_u^{(\nu_N)}\right)^{\dag}\mathrm{d}u,\  \  \text{as } t\to \infty.
\end{align}
Now, by the matrix Dufresne identity proven in \cite{Rider-Valko}, see also \cite{OConnellMatrixDiffusions}, we know that
\begin{equation*}
\mathsf{Law}\left(\frac{1}{2}\int_{0}^{\infty}\mathbf{N}_u^{(\nu_N)}\left(\mathbf{N}_u^{(\nu_N)}\right)^{\dag}\mathrm{d}u\right) =\mathfrak{IL}_N^{\eta},
\end{equation*}
where $\mathfrak{IL}_N^{\eta}$ is the inverse Laguerre measure on $\mathbb{H}_+(N)$  given by the formula,
\begin{equation*}
  \mathfrak{IL}_N^\eta(\mathrm{d}\mathbf{A})=\det\left(\mathbf{A}\right)^{-\eta-2N}\exp\left(-\textnormal{Tr}\left(\mathbf{A}^{-1}\right)\right)  \mathrm{d}\mathbf{A},
\end{equation*}
with $\mathrm{d}\mathbf{A}$ denoting the Lebesgue measure on $\mathbb{H}_+(N)$, for some normalisation constant $\tilde{\mathcal{Z}}_N^{(\eta)}$.
By Weyl's integration formula we know that 
\begin{equation*}
\left(\mathsf{eval}_N\right)_*\mathfrak{IL}_N^\eta=\mathfrak{M}_N^\eta,
\end{equation*}
and this concludes the proof of (\ref{ConvToEqFiniteN}). The final statement follows by a standard argument.
\end{proof}

As a consequence of Theorem \ref{thm-intertwiningSemigroup} and uniqueness of the invariant measures from Proposition \ref{prop-ConvToEquilibrium} we obtain the consistency relation between the inverse Laguerre ensembles; see also \cite{InvWishart} for a direct proof, using orthogonal polynomials.

     \begin{prop}\label{prop-intertwining measures}
        Let $\eta>-1$. Then, we have
         \begin{align}
\mathfrak{M}^{\eta}_{N+1}\Lambda_N^{N+1}=\mathfrak{M}^{\eta}_{N},\ \ \forall N\in\mathbb{N}.
         \end{align}
\end{prop}

We also have the following equivalent abstract description of the measure $\mathfrak{M}^\eta$ from the introduction.

\begin{prop}\label{prop-MeasDescr}
Let $\eta>-1$. The unique $\mathfrak{m} \in \mathscr{M}_\textnormal{p}(\Omega_+)$ satisfying:
\begin{equation*}
\mathfrak{m}\Lambda_N^\infty=\mathfrak{M}_N^\eta, \ \ \forall N\in\mathbb{N},
\end{equation*}
is given by $\mathfrak{M}^\eta$.
\end{prop}

\begin{proof}
This is the main result of \cite{InvWishart}.
\end{proof}

\subsection{Proof of Theorem \ref{MainThm1Intro} and Theorem \ref{MainThmConvEq}}

\begin{proof}[Proof of Theorem \ref{MainThm1Intro}]
This follows by taking the semigroups $\left(\mathsf{P}_N(t)\right)_{t\ge 0}$ to be $\left(\mathfrak{P}_N(t)\right)_{t\ge 0}$ in Theorem \ref{thm-MarkovProcessonBoundary}, Theorem \ref{thm-ConvMarkovProc+} and Proposition \ref{Prop-Conv in l^2} all of whose conditions are satisfied by virtue of Proposition \ref{prop-FellerProperty} and Proposition \ref{prop-Core+TransitionDens} and Theorem \ref{thm-intertwiningSemigroup} and Theorem \ref{thm-Boundary+}.
\end{proof}

\begin{proof}[Proof of Theorem \ref{MainThmConvEq}]
By taking $\left(\mathsf{P}_N(t)\right)_{t\ge 0}$ to be $\left(\mathfrak{P}_N(t)\right)_{t\ge 0}$ this is a consequence of Theorem \ref{thm-ergodicity} by virtue of Proposition \ref{prop-ConvToEquilibrium} and Proposition \ref{prop-MeasDescr}. 
\end{proof}

\subsection{On the Gibbs property}\label{SectionGibbs}

We briefly discuss a new Gibbs resampling property, see \cite{CorwinHammond}, underlying our model. We assume the reader is familiar with the basic terminology of Gibbsian line ensembles, see \cite{CorwinHammond}.  This subsection is not part of our argument and can be skipped. 

For $\eta>-1$, there exists an explicit spectral expansion of $\mathfrak{p}_t^{(N)}(x,y)$ in terms of special functions, see \cite{Wong}. This expansion is well-adapted to looking at the $t\to \infty $ asymptotics of $\mathfrak{p}_t^{(N)}(x,y)$ and using it one can prove that the semigroup $\left(\mathfrak{P}_N(t)\right)_{t\ge 0}$ is exactly the transition semigroup of $N$ independent one-dimensional diffusions with generator $\mathsf{L}^{(N)}$ from \eqref{LN} conditioned to never intersect\footnote{We note that this statement does not follow directly from the Doob $h$-transform structure of $\left(\mathfrak{P}_N(t)\right)_{t\ge 0}$ in \eqref{transitioDensity}. One needs to compute the asymptotics of the first collision time of the $N$ independent one-dimensional diffusions $\mathsf{L}^{(N)}$ in order show that these two structures coincide and this is in general a delicate problem.}. In particular, $(\mathfrak{x}^{(N)}(t);t\ge 0)$, viewed as a line ensemble with $N$ lines, see \cite{CorwinHammond}, enjoys a certain Gibbs resampling property with respect to the law of conditioned $\mathsf{L}^{(N)}$-bridges, see \cite{Fitzsimmons-Yor} for the rigorous definition of diffusion bridges; in the framework of \cite{CorwinHammond} we replace Brownian bridges with $\mathsf{L}^{(N)}$-bridges. Similarly, the rescaled process $(\mathsf{x}^{(N)}(t);t\ge 0)$ from \eqref{rescaledSDE} enjoys a Gibbs resampling property with respect to $\mathsf{L}_{\textnormal{rsc}}^{(N)}$-bridges, where $\mathsf{L}_{\textnormal{rsc}}^{(N)}$ is the generator of the rescaled by $N^{-1}$ diffusion $\mathsf{L}^{(N)}$ given by:
\begin{equation*}
\mathsf{L}_{\textnormal{rsc}}^{(N)}=\frac{x^2}{2}\partial_{x}^2+\left(\left(1-\frac{\eta}{2}-N\right)x+\frac{1}{2N}\right)\partial_{x},
\end{equation*}
with transition density denoted by $\mathfrak{q}_t^{(N)}(x,y)$. It is then a natural question whether, the limiting paths $((\mathsf{x}_i(t))_{i=1}^\infty;t\ge 0)$, which we do not yet know if they are non-intersecting, enjoy a Gibbs resampling property as well. If they did this would give an alternative, to the one in Section \ref{SebSec-NonIntersection} below, albeit more complicated, way to prove non-intersection.

One can show, that given the convergence of the paths, that we already know from Theorem \ref{MainThm1Intro}, this question really boils down, modulo assumptions of lesser importance to avoid pathological situations, to convergence of the free $\mathsf{L}_{\textnormal{rsc}}^{(N)}$-bridge, whose time-inhomogeneous generator, in case of the bridge ending at point $y$ at time $T>0$, is given by, for $0\le t<T$, see \cite{Fitzsimmons-Yor},
\begin{equation*}
\frac{x^2}{2}\partial_{x}^2+\left(\left(1-\frac{\eta}{2}-N\right)x+\frac{1}{2N}+x^2\partial_x \log \mathfrak{q}_{T-t}^{(N)}(x,y)\right)\partial_{x},
\end{equation*}
to a limiting diffusion bridge as $N \to \infty$. This question turns out to be complicated. The spectral expansion in \cite{Wong} is not well-adapted for looking at the $N \to \infty$ limit of the above expression.  Nevertheless, using some judicious transformations, it is possible to connect this problem to a problem of asymptotics of the logarithmic derivative of the transition kernel of a Schr\"{o}dinger semigroup with an explicit $N$-dependent Morse potential $\mathsf{V}_N(x)$, see \cite{MatsumotoYor1,MatsumotoYor2}. What we gain by doing this is that all $N$-dependence of the problem is now in the potential $\mathsf{V}_N(x)$ and it is  easy to see that it converges to an explicit potential $\mathsf{V}_\infty(x)$. Alas, the limiting potential $\mathsf{V}_\infty(x)$ has the ``wrong" sign and the associated Schr\"{o}dinger semigroup blows up! In particular, the transition kernel of the Schr\"{o}dinger semigroup with potential $\mathsf{V}_N(x)$ diverges as $N\to \infty$. However, its logarithmic derivative may still well converge, due a subtle cancellation, and so this does not contradict the  existence of a limiting Gibbs property. 

The naive guess is that we have an exponential Brownian Gibbs property in the limit. The simple intuition comes from the fact that if the vanishing $1/2N$ drift term in $\mathsf{L}_{\textnormal{rsc}}^{(N)}$ was absent then $\mathsf{L}_{\textnormal{rsc}}^{(N)}$ would have been an exponential Brownian motion with $N$-dependent drift and the corresponding bridge would have been an exponential Brownian bridge identically for all $N$.

\section{The ISDE}\label{Section-ISDE}

In this section we first prove that $\left(\mathsf{x}_i(\cdot)\right)_{i=1}^\infty$ almost surely consists of non-intersecting paths in Theorem \ref{thm-NonCollision,>0} and then in Theorem \ref{thm-ISDE} that it solves the ISDE \eqref{ISDEintro}. The main idea will be to use various ``characteristic polynomial"-like objects associated to the finite-dimensional dynamics $(\mathsf{x}^{(N)}_i(\cdot))_{i=1}^N$, first as Lyapunov functions, to control uniformly in $N$, the times two paths come close and then to show convergence of the singular drift term. We need some preliminaries.

We recall a result from \cite{RandomEntireFun} on convergence of polynomials with real zeros to analytic functions that we use as one of the ingredients in the proofs of 
Theorem \ref{thm-NonCollision,>0} and Theorem \ref{thm-ISDE} below. Given, $\mathbf{x}^{(N)} \in \mathbb{W}_{N}$ define the polynomial
\begin{align*}
\mathfrak{U}_N\left(z;\mathbf{x}^{(N)}\right)\overset{\textnormal{def}}{=}\prod_{j=1}^{N}\left(1-x_j^{(N)}z\right).
\end{align*}
This is the reverse characteristic polynomial of a matrix with eigenvalues $\mathbf{x}^{(N)}$. We also have the following definition.

\begin{defn}
Define the entire functions $\mathsf{E}(z;\omega)$ and $\mathsf{E}_+(z;\omega)$, indexed by $\omega\in \Omega$ and $\omega\in \Omega_+$ respectively,
given by their Hadamard products:
\begin{align*}
\mathsf{E}(z;\omega)&\overset{\textnormal{def}}{=}\mathrm{e}^{-\gamma z-\frac{\delta-\sum_{i=1}^\infty \left[(x_i^+)^2+(x_i^-)^2\right]}{2}z^2}\prod_{j=1}^{\infty}\mathrm{e}^{x_j^+ z}\left(1-x_j^+z\right)\prod_{j=1}^{\infty}\mathrm{e}^{-x_j^- z}\left(1+x_j^- z\right),\  \ \omega\in\Omega,\\
\mathsf{E}_+(z;\omega)&\overset{\textnormal{def}}{=}\mathrm{e}^{-\gamma z}\prod_{j=1}^{\infty}\mathrm{e}^{x_j z}\left(1-x_jz\right),\  \ \omega\in\Omega_+.
\end{align*}
  
\end{defn}

Then, we have the following result, see
Proposition 2.5 and Remark 2.7 of \cite{RandomEntireFun}.

\begin{prop}\label{prop-CharPolyConv}
Suppose $\left(\mathbf{x}^{(N)}\right)_{N=1}^{\infty}\in \left(\mathbb{W}_N\right)_{N=1}^{\infty}$ such that $\omega\left(\mathbf{x}^{(N)}\right)\overset{N\to\infty}{\longrightarrow}\omega\in\Omega$. Then, the following convergence holds uniformly on compact sets in $\mathbb{C}$:
\begin{align*}
\mathfrak{U}_N\left(z;\mathbf{x}^{(N)}\right)\to
\mathsf{E}(z;\omega), \ \ \textnormal{ as }\;N\to\infty.
\end{align*}
If instead $\left(\mathbf{x}^{(N)}\right)_{N=1}^{\infty}\in \left(\mathbb{W}_{N,+}\right)_{N=1}^{\infty}$ such that $\omega\left(\mathbf{x}^{(N)}\right)\overset{N\to\infty}{\longrightarrow}\omega\in\Omega_+$, then the following convergence holds uniformly on compact sets in $\mathbb{C}$:
\begin{align}\label{CharPolyConv+}
\mathfrak{U}_N\left(z;\mathbf{x}^{(N)}\right)\to
\mathsf{E}_+(z;\omega), \  \ \text{ as }\;N\to\infty.
\end{align}
\end{prop}
Note that this convergence statement is consistent with the embedding \eqref{omega_+ToOmega} of $\Omega_+$ into $\Omega$. Below we will only use the result for $\mathbb{W}_{N,+}$ and $\Omega_+$.

\subsection{Non-intersection via characteristic polynomials}\label{SebSec-NonIntersection}

We use the following notation convention throughout. For $(\mathbf{x},\gamma)\in \Omega_+$ we write $\mathbb{P}_{(\mathbf{x},\gamma)}$ for the law of $\left(\mathbf{X}_t^{\Omega_+};t\ge 0\right)$ from Theorem \ref{MainThm1Intro} starting from $(\mathbf{x},\gamma)$ and abusing notation (as we do not assume that all these processes are defined on the same probability space unless otherwise stated), for $\mathbf{x}^{(N)}\in \mathbb{W}_{N,+}$, we write $\mathbb{P}_{\mathbf{x}^{(N)}}$ for the law of the process solving \eqref{rescaledSDE} starting from $\mathbf{x}^{(N)}$. We write $\mathbb{E}_{(\mathbf{x},\gamma)}$ and $\mathbb{E}_{\mathbf{x}^{(N)}}$ for the corresponding expectations.

\begin{thm}\label{thm-NonCollision,>0}
   Let $\eta \in \mathbb{R}$, $\mathbf{x} \in \mathbb{W}_{\infty,+}^\circ$, $\gamma \in \mathbb{R}_+$, with $\sum_{i=1}^\infty x_i\le \gamma$. Consider the Feller process on $\Omega_+$, $\left(\mathbf{X}_t^{\Omega_+};t\ge 0\right)=\left(\left(\left(\mathsf{x}_i(t)\right)_{i=1}^\infty,\boldsymbol{\gamma}(t)\right);t\ge 0\right)$ from Theorem \ref{MainThm1Intro} with initial condition $\mathbf{X}_0^{\Omega_+}=(\mathbf{x},\gamma)$. Then, for all $T\ge 0$ and $n\in\mathbb{N}$ we have
    \begin{align*}
\mathbb{P}_{(\mathbf{x},\gamma)}\left(\inf_{t\in[0,T]}\mathsf{x}_{n}(t)>0
\right)&=1,\\
\mathbb{P}_{(\mathbf{x},\gamma)}\left(\inf_{t\in[0,T]}\left|
\mathsf{x}_{n}(t)-\mathsf{x}_{n+1}(t)
\right|>0\right)&=1.
\end{align*}

\end{thm}
\begin{proof}
  We prove the result using an inductive argument. For each $N\in\mathbb{N}$, let $\mathbf{x}^{(N)}\in \mathbb{W}_{N,+}^\circ$ converge to $(\mathbf{x},\gamma)\in \Omega_+$ under the embedding \eqref{embedding+}. 
Let $\mathsf{x}^{(N)}$ be the solution to the equation \eqref{rescaledSDE} started at $\mathbf{x}^{(N)}$.
We will make use of the $\log$ transform of \eqref{rescaledSDE}, which we consider in its integral form:
\begin{align}\label{logX^N}
\log\mathsf{x}_i^{(N)}(t) &= \log x_i^{(N)}+ \mathsf{w}_i(t) -\frac{1+\eta}{2}t+\int_0^t\frac{1}{2N\mathsf{x}_i^{(N)}(s)}\mathrm{d}s\nonumber\\
&\  \ +
\int_0^t\sum_{j=1, j\neq i}^{N}\frac{\mathsf{x}_j^{(N)}(s)}{\mathsf{x}_i^{(N)}(s)-\mathsf{x}_j^{(N)}(s)}\mathrm{d}s,\  \
i=1,\dots,N.
\end{align}
Let us continue with some definitions. We always use the convention $\inf \emptyset =\infty$ in what follows. 
Define, 
for any $N\in\mathbb{N}$, the following stopping times (with respect to the natural filtration) associated to $\left(\mathsf{x}_i^{(N)}\right)_{i=1}^N$,
\begin{align*}
\sigma_n^N(\epsilon)&\overset{\textnormal{def}}{=}\inf\left\{t\geq 0;\mathsf{x}_n^{(N)}(t)\le \epsilon\right\},\  \  \epsilon>0,\  \ n=1,\dots,N\\
\tau_n^N(\delta)&\overset{\textnormal{def}}{=}\inf\left\{t\geq 0;\left|1-\frac{\mathsf{x}_{n+1}^{(N)} (t)}{\mathsf{x}_n^{(N)} (t)}\right|\le\delta\right\},\  \  \delta>0,\  \ n=1,\dots,N-1,\\
\theta_n^N(r)&\overset{\textnormal{def}}{=}\inf\left\{t\geq 0;\mathsf{x}_n^{(N)}(t)\ge r\right\},\  \ r>0,\  \ n=1,\dots,N,
\end{align*}
and 
analogously, the stopping times (with respect to the natural filtration) associated to $\left(\mathsf{x}_i\right)_{i=1}^\infty$,
\begin{align*}
\sigma_n(\epsilon)&\overset{\textnormal{def}}{=}\inf\left\{t\geq 0;\mathsf{x}_n(t)\le\epsilon\right\},\  \  \epsilon>0,\  \ n\in\mathbb{N},\\
\tau_n(\delta)&\overset{\textnormal{def}}{=}\inf\left\{t\geq 0;\left|1-\frac{\mathsf{x}_{n+1} (t)}{\mathsf{x}_n (t)}\right|\le\delta\right\},\  \  \delta>0,\  \ n\in\mathbb{N},\\
\theta_n(r)&\overset{\textnormal{def}}{=}\inf\left\{t\geq 0;\mathsf{x}_n(t)\ge r\right\},\  \ r>0,\  \ n\in\mathbb{N}.
\end{align*}
Note that, by Lemma \ref{lem-wellPosednessN}, $\tau_n^N(\delta)$ is well-defined for any $N\in\mathbb{N}$ and $n\leq N-1$. 
It will be justified later on that $\tau_n(\delta)$ is also well-defined for all $n\in\mathbb{N}$.

Now, let $N\in\mathbb{N}$. We define for any $\mathbf{x}^{(N)}\in\mathbb{W}_{N,+}^\circ$, the following variant of the characteristic polynomial associated to $\mathbf{x}^{(N)}$,
\begin{align*}
\Psi_n^N\left(\mathbf{x}^{(N)}\right)&\overset{\textnormal{def}}{=}\prod_{i=1}^{n}\prod_{j=n+1}^{N}\left(1-\frac{x_j^{(N)}}{x_i^{(N)}}\right),\  \  n=1,\dots,N-1,
    \end{align*} 
and consider the Lyapunov function
given as:
\begin{align}\label{f_n^N}
f_n^N\left(\mathbf{x}^{(N)}\right)\overset{\textnormal{def}}{=}-\log\left(\mathrm{\Psi}_n^N\left(\mathbf{x}^{(N)}\right)\right),\  \  n=1,\dots,N-1.
\end{align}
It is easy to see that $f_n^N\left(\mathbf{x}^{(N)}\right)>0$ for all $\mathbf{x}^{(N)}\in\mathbb{W}_{N,+}^\circ$ and $n=1,\dots,N-1$.
    We next define
 \begin{align*}
    \kappa_n^N(R)\overset{\textnormal{def}}{=}\inf\left\{t\geq 0;f_n^N\left(\mathsf{x}^{(N)}(t)\right)\ge R\right\},\  \   R>0,\  \  n=1,\dots,N-1.
    \end{align*}
    Observe that,
\begin{align}\label{T-tau}
        \left\{\tau_n^N(\delta)\le a\right\}\subset \left\{\kappa_n^N\left(\log(\delta^{-1})\right)\le a\right\},\  \ \textnormal{for } \delta<1.
    \end{align}
We are going to prove by induction on $n$ that 
for all $T\ge 0$ and $n\in\mathbb{N}$, 
\begin{align}
\mathbb{P}_{(\mathbf{x},\gamma)}\left(\inf_{t\in[0,T]}\mathsf{x}_{n}(t)>0
\right)&=1,\label{X_n>0}\\
\mathbb{P}_{(\mathbf{x},\gamma)}\left(\inf_{t\in[0,T]}\left|1-\frac{\mathsf{x}_{n+1}(t)}{\mathsf{x}_n(t)}\right|>0\right)&=1.\label{ratio_n<1}
\end{align}
Note that, this implies non-collision, since given \eqref{X_n>0} and \eqref{ratio_n<1}, we have almost surely,
\begin{align*}
\inf_{t\in[0,T]}\left|
\mathsf{x}_{n}(t)-\mathsf{x}_{n+1}(t)
\right|\geq \inf_{t\in[0,T]}\mathsf{x}_n(t)
\inf_{t\in[0,T]}\left|1-\frac{\mathsf{x}_{n+1}(t)}{\mathsf{x}_n(t)}\right|,
\end{align*}
and hence,
\begin{align*}
\mathbb{P}_{(\mathbf{x},\gamma)}\left(\inf_{t\in[0,T]}\left|
\mathsf{x}_{n}(t)-\mathsf{x}_{n+1}(t)
\right|>0\right)=1.
\end{align*}

\textbf{The base case.}
Using \eqref{logX^N}, we have for any $t\ge 0$ that
	\begin{align*}
		\log\mathsf{x}_{1}^{(N)}\left(t\wedge \sigma_{1}^N(\epsilon)\wedge \theta_{1}^N(r)\right)& \geq
    \log x_{1}^{(N)}
		-\left|\frac{1+\eta}{2}\right|t+\mathsf{w}_{1}\left(t\wedge \sigma_{1}^N(\epsilon)\wedge\theta_{1}^N(r)\right).
	\end{align*}
	Therefore, since $t \mapsto \mathsf{w}_{1}\left(t\wedge \sigma_{1}^N(\epsilon)\wedge\theta_{1}^N(r)\right)$ is a mean-zero martingale, taking expectation of both sides we obtain
	\begin{align*}
		\log r+\log \epsilon \mathbb{P}_{\mathsf{x}^{(N)}}\left(\sigma_{1}^N(\epsilon)< t\wedge \theta_{1}^N(r)\right)&\geq \log x_{1}^{(N)}-\left|\frac{1+\eta}{2}\right|t,
	\end{align*}
and thus, for any $\epsilon<1$,
\begin{align}\label{P(sigma<theta)}
\mathbb{P}_{\mathbf{x}^{(N)}}\left(\sigma_{1}^N(\epsilon)< t\wedge \theta_{1}^N(r)\right)&\leq \frac{\log r-\log x_{1}^{(N)}+\left|\frac{1+\eta}{2}\right|t}{|\log \epsilon|}.
	\end{align}
Observe that, the bound in \eqref{P(sigma<theta)} converges as $N\to\infty$ (recall that $x_1^{(N)} \to x_1>0$). We would like to conclude that the same bound (with $x_1$ in the formula instead) holds for the stopping times $\sigma_1(\epsilon)$ and $\theta_1(r)$ corresponding to the limit paths. A little argument is needed though. These stopping times, viewed as functionals on paths are not continuous, but fortunately they are not far from continuous. For more details see Chapter VI, Sections 2,3 of \cite{JacodShiryaev} which we follow closely. For $a \in \mathbb{R}_+$, $f\in C(\mathbb{R}_+,\mathbb{R}_+)$, we define
\begin{align*}
\mathsf{S}_{\ge}^a(f)=\inf\left\{t\ge 0; f(t)\ge a\right\},\  \ \mathsf{S}_{\le}^a(f)=\inf\left\{t\ge 0; f(t)\le a\right\}.
\end{align*}
 Write $\mathbb{R}_{+,\textnormal{ext}}=\mathbb{R}_+\cup \{\infty\}$. For fixed $f\in C(\mathbb{R}_+,\mathbb{R}_+)$ define the map $\mathsf{B}_f$ by,
 \begin{align*}
    \mathsf{B}_{f}:\mathbb{R}_+^2 &\longrightarrow \mathbb{R}_{+,\textnormal{ext}}^2\\
    (a,b)&\mapsto \left(\mathsf{S}_{\le }^a(f),\mathsf{S}_{\ge }^b(f)\right),
 \end{align*}
 which from Lemma 2.10 of Chapter VI of \cite{JacodShiryaev} only has jump discontinuities, while for fixed $(a,b)\in \mathbb{R}_+^2$ define the map $\mathsf{A}_{(a,b)}$ by,
 \begin{align*}
 \mathsf{A}_{(a,b)}:C(\mathbb{R}_+,\mathbb{R}_+) &\longrightarrow \mathbb{R}_{+,\textnormal{ext}}^2\\
 f&\mapsto \left(\mathsf{S}_{\le }^a(f),\mathsf{S}_{\ge }^b(f)\right).
 \end{align*}
 Now, if we let $\mathfrak{D}_{(a,b)}$ denote the set of discontinuities of the map $\mathsf{A}_{(a,b)}$ and $\mathsf{D}(f)$ denote the set of discontinuities of the map $\mathsf{B}_f$, then from Proposition 2.11 of Chapter VI of \cite{JacodShiryaev}, since $\mathsf{A}_{(a,b)}$ is continuous at every $f$ so that $(a,b)\notin \mathsf{D}(f)$, we obtain that,
 \begin{equation*}
  \mathfrak{D}_{(a,b)}\subset \left\{f\in C(\mathbb{R}_+,\mathbb{R}_+):(a,b)\in \mathsf{D}(f)\right\}.
 \end{equation*}
 Let us write $\mathbb{P}^{(1)}_{(\mathbf{x},\gamma)}$ for $\mathsf{Law}(\mathsf{x}_1)$ on $C(\mathbb{R}_+,\mathbb{R}_+)$ under $\mathbb{P}_{(\mathbf{x},\gamma)}$. Then, define
\begin{align*}
\mathcal{V}=\left\{(a,b)\in \mathbb{R}_+^2:\mathbb{P}_{(\mathbf{x},\gamma)}^{(1)}\left(\mathfrak{D}_{(a,b)}\right)>0\right\} \subset \left\{(a,b)\in \mathbb{R}_+^2:\mathbb{P}_{(\mathbf{x},\gamma)}^{(1)}\left((a,b)\in \mathsf{D}(\mathsf{x}_1)\right)>0\right\}.
\end{align*} 
Then, by virtue of the fact that the map $\mathsf{B}_f$ only has jump discontinuities, we obtain, see Lemma 3.12 of Chapter VI of \cite{JacodShiryaev}, that the sets $\mathcal{V}_1=\{a\in \mathbb{R}_+:(a,b)\in \mathcal{V}\}$ and $\mathcal{V}_2=\{b\in \mathbb{R}_+:(a,b)\in \mathcal{V}\}$ are at most countable. In particular, we can find sequences $(\epsilon_m)_{m=1}^\infty$ and $(r_\ell)_{\ell=1}^\infty$, with $\epsilon_m\to 0$ and $r_\ell \to \infty$, as $m,\ell \to \infty$, so that for any $m,\ell \in \mathbb{N}$, $\mathbb{P}_{(\mathbf{x},\gamma)}^{(1)}(\mathfrak{D}_{(\epsilon_m,r_{\ell})})=0$. Now, observe that,
\begin{equation*}
\left(\sigma_1^N(\epsilon),\theta_1^N(r)\right)=\mathsf{A}_{(\epsilon,r)}\left(\mathsf{x}_1^{(N)}\right)
, \ \ \left(\sigma_1(\epsilon),\theta_1(r)\right)=\mathsf{A}_{(\epsilon,r)}(\mathsf{x}_1).
\end{equation*}
Hence, by the continuous mapping theorem, we obtain that for all $m,l\in \mathbb{N}$,  as $N \to \infty$,
\begin{equation*}
\left(\sigma_1^N(\epsilon_m),\theta_1^N(r_\ell)\right)\overset{\textnormal{d}}{\longrightarrow} \left(\sigma_1(\epsilon_m),\theta_1(r_\ell)\right).
\end{equation*}
By the Portmanteau theorem for convergence in distribution, we get that, for any open set $\mathcal{O}\subseteq\mathbb{R}^2$,  
 \begin{align*}
     \mathbb{P}_{(\mathbf{x},\gamma)}\left(\left(\sigma_{1}(\epsilon_m),\theta_{1}(r_\ell)\right)\in\mathcal{O}\right)
     \le\liminf_{N\to \infty}\mathbb{P}_{\mathbf{x}^{(N)}}\left(\left(\sigma_{1}^N(\epsilon_m),\theta_{1}^N(r_\ell)\right)\in\mathcal{O}\right).   
 \end{align*}
 From this and  \eqref{P(sigma<theta)} we conclude that (for $\epsilon_m<1$),
\begin{align*}
\mathbb{P}_{(\mathbf{x},\gamma)}\left(\sigma_{1}(\epsilon_m)< t\wedge \theta_{1}(r_\ell)\right)\leq \liminf_{N\to \infty}\mathbb{P}_{\mathsf{x}^{(N)}}\left(\sigma_{1}^N(\epsilon_m)< t\wedge \theta_{1}^N(r_\ell)\right)\leq \frac{\log r_\ell-\log x_1+\left|\frac{1+\eta}{2}\right|t}{|\log \epsilon_m|},
\end{align*}
and hence, by the monotone convergence theorem,
\begin{align*}
\mathbb{P}_{(\mathbf{x},\gamma)}\left(\lim_{m\to \infty}\sigma_{1}(\epsilon_m)< t\wedge \theta_{1}(r_\ell)\right)=0.
\end{align*}
Therefore, by taking $\ell,t\to\infty$, it follows, by virtue of the fact that (recall that $r_\ell \to \infty$), $\theta_{1}(r_{\ell})\overset{\textnormal{a.s.}}{\longrightarrow}\infty$, that $\mathbb{P}_{(\mathbf{x},\gamma)}\left(\lim_{m \to \infty}\sigma_1(\epsilon_m)=\infty\right)=1$, namely, since $\epsilon_m \to 0$,
\begin{align}\label{X_1>0}
\mathbb{P}_{(\mathbf{x},\gamma)}\left(\inf_{t\in[0,T]}\mathsf{x}_{1}(t)>0
\right)=1.
\end{align}

We now proceed to prove the second assertion. 
Using It\^{o}'s formula, we have 
\begin{align} \mathrm{d}f_1^N\left(\mathsf{x}^{(N)}(t)\right)&=
    \mathrm{d}\mathsf{M}_1^{(N)}(t)+
    \frac{N-1}{2N\mathsf{x}_1^{(N)}(t)}\mathrm{d}t-\sum_{j=2}^{N}\frac{\left(\mathsf{x}_j^{(N)}(t)\right)^2}{\left(\mathsf{x}_1^{(N)}(t)-\mathsf{x}_j^{(N)}(t)\right)^2}\mathrm{d}t \nonumber \\ &\ 
 \ 
+\sum_{j=2}^{N}\sum_{\ell=2,\ell\neq j}^{N}\frac{\mathsf{x}_j^{(N)}(t)\mathsf{x}_\ell^{(N)}(t)}{\left(\mathsf{x}_1^{(N)}(t)-\mathsf{x}_\ell^{(N)}(t)\right)\left(\mathsf{x}_{j}^{(N)}(t)-\mathsf{x}_\ell^{(N)}(t)\right)}\mathrm{d}t. \label{LyapunovApp1}
\end{align}
where
\begin{align*}
\mathsf{M}_1^{(N)}(t)\overset{\textnormal{def}}{=}(N-1) \mathsf{w}_1(t)+\sum_{\ell=2}^{N}\int_0^t\frac{\mathsf{x}_1^{(N)}\mathrm{d}\mathsf{w}_1(s)-\mathsf{x}_{\ell}^{(N)} \mathrm{d}\mathsf{w}_{\ell}(s)}{\mathsf{x}_1^{(N)}-\mathsf{x}_{\ell}^{(N)}}, \ \ \forall t \ge 0,
\end{align*}
is a local martingale. Now, we observe, and this is key, that for any $t\ge 0$, the last term on the RHS of \eqref{LyapunovApp1} is non-positive, by virtue of the elementary inequality, for $\mathbf{y}\in \mathbb{W}_{N,+}^\circ$,
\begin{equation*}
\sum_{j=2}^{N}\sum_{\ell=2,\ell\neq j}^{N}\frac{y_j y_\ell}{(y_1-y_\ell)(y_j-y_\ell)}
=\sum_{\underset{j<\ell}{j,\ell=2}}^{N}\frac{y_j y_\ell}{y_j-y_\ell}\left(\frac{1}{y_1-y_\ell}-\frac{1}{y_1-y_j}\right)\le 0.    
\end{equation*}
Thus, we can drop this term and write 
\begin{align*}
    f_1^N\left(\mathsf{x}^{(N)}\left(t\wedge \kappa_1^N(R)\wedge \sigma_1^N(\epsilon)\right)\right)&\leq f_1^N\left(\mathbf{x}^{(N)}\right)+\mathsf{M}_1^{(N)}\left(t\wedge \kappa_1^N(R)\wedge \sigma_1^N(\epsilon)\right) \\ &\  \
 +\int_0^{t\wedge \kappa_1^N(R)\wedge \sigma_1^N(\epsilon)}\frac{1}{2\mathsf{x}_1^{(N)}(s)}\mathrm{d}s\\&\leq f_1^N\left(\mathbf{x}^{(N)}\right)+\mathsf{M}_1^{(N)}\left(t\wedge \kappa_1^N(R)\wedge \sigma_1^N(\epsilon)\right)+ 
    \frac{t}{\epsilon}.
\end{align*}
Now, since the stopped process $t \mapsto \mathsf{M}_1^{(N)}\left(t\wedge \kappa_1^N(R)\wedge \sigma_1^N(\epsilon)\right)$ is a mean-zero martingale, we get
\begin{align*}
\mathbb{E}_{\mathbf{x}^{(N)}}\left[f_1^N\left(\mathsf{x}^{(N)}\left(t\wedge \kappa_1^N(R)\wedge \sigma_1^N(\epsilon)\right)\right)\right]&\leq f_1^N\left(\mathbf{x}^{(N)}\right)+\frac{t}{\epsilon}.
\end{align*}
Therefore, we have
\begin{align*}
    R\mathbb{P}_{\mathbf{x}^{(N)}}\left(\kappa_1^N(R)< t\wedge \sigma_1^N(\epsilon)\right)\leq \sup_{N\in\mathbb{N}}f_1^N\left(\mathbf{x}^{(N)}\right)+
    \frac{t}{\epsilon}.
\end{align*}
Note that, since
$\omega\left(\mathbf{x}^{(N)}\right)\overset{N\to\infty}{\longrightarrow}(\mathbf{x},\gamma)$, it follows from Proposition \ref{prop-CharPolyConv} that 
$\Psi_1^N\left(\mathbf{x}^{(N)}\right)$ converges to a strictly positive limit, 
and hence, 
$C\overset{\textnormal{def}}{=}\sup_{N\in\mathbb{N}}f_1^N\left(\mathbf{x}^{(N)}\right)<\infty$. We thus have,
\begin{align*}
\mathbb{P}_{\mathbf{x}^{(N)}}\left(\kappa_1^{N}(R)< t\wedge \sigma_1^N(\epsilon)\right)\leq \frac{C+\frac{t}{\epsilon}}{R}.
\end{align*}
Let $\delta<1$ be arbitrary. We have in view of \eqref{T-tau} that 
\begin{align}\label{P(tau<sigma)}
\mathbb{P}_{\mathbf{x}^{(N)}}\left(\tau_1^{N}(\delta)< t\wedge \sigma_1^N(\epsilon)\right)\leq \frac{C+\frac{t}{\epsilon}}{|\log\delta|},
\end{align}
and note that this bound is independent of $N$. Observe that,
\begin{equation*}
\tau_1^N(\delta)=\mathsf{S}_{\le}^\delta \left(\left|1-\frac{\mathsf{x}_2^{(N)}}{\mathsf{x}_1^{(N)}}\right|\right), \ \ \tau_1(\delta)=\mathsf{S}_{\le}^\delta \left(\left|1-\frac{\mathsf{x}_2}{\mathsf{x}_1}\right|\right),
\end{equation*}
and that by virtue of \eqref{X_1>0}, $t\mapsto |\mathsf{x}_1(t)|^{-1}|\mathsf{x}_1(t)-\mathsf{x}_2(t)|$ is almost surely, a well-defined continuous function. We now argue similarly to before to obtain the existence of sequences, abusing notation (the sequence $(\epsilon_\ell)_{\ell=1}^\infty$ here may have nothing to do with the sequence $(\epsilon_m)_{m=1}^\infty$ from earlier), $(\delta_m)_{m=1}^\infty$, $(\epsilon_\ell)_{\ell=1}^\infty$, with $\delta_m, \epsilon_\ell \to 0$, as $m,\ell \to \infty$, so that, for all $m,\ell \in \mathbb{N}$, as $N \to \infty$,
\begin{equation*}
\left(\tau_1^N(\delta_m),\sigma_1^N(\epsilon_\ell)\right)\overset{\textnormal{d}}{\longrightarrow} \left(\tau_1(\delta_m),\sigma_1(\epsilon_\ell)\right).
\end{equation*}
Hence, by taking $N\to\infty$ in \eqref{P(tau<sigma)}, one obtains, for $\delta_m<1$,
\begin{align*}
\mathbb{P}_{(\mathbf{x},\gamma)}\left(\tau_1(\delta_m)< t\wedge \sigma_1(\epsilon_\ell)\right)\leq \liminf_{N \to \infty } \mathbb{P}_{\mathbf{x}^{(N)}}\left(\tau_1^{N}(\delta_m)< t\wedge \sigma_1^N(\epsilon_\ell)\right)\le \frac{C+\frac{t}{\epsilon_\ell}}{|\log\delta_m|}.
\end{align*}
It thus follows that,
\begin{align*}
\mathbb{P}_{(\mathbf{x},\gamma)}\left(\lim_{m\to \infty}\tau_1(\delta_m)< t\wedge \sigma_1(\epsilon_\ell)\right)=0.
\end{align*}
Now, we take $t\to\infty$,
 and then $\ell\to \infty$, to finally get
\begin{align*}
\mathbb{P}_{(\mathbf{x},\gamma)}\left(\lim_{m\to \infty}\tau_1(\delta_m)\geq \lim_{\ell \to \infty}\sigma_1(\epsilon_\ell)\right)=1.
\end{align*}
But $\mathbb{P}_{(\mathbf{x},\gamma)}\left(\lim_{\ell \to \infty}\sigma_1(\epsilon_\ell )=\infty\right)=1$ by virtue of \eqref{X_1>0}.
This implies
\begin{align*}
\mathbb{P}_{(\mathbf{x},\gamma)}\left(\lim_{m \to \infty}\tau_1(\delta_m)=\infty\right)=1.
\end{align*}
That is, for all $T\ge 0$, we have
\begin{align}\label{ratio_1<1}
\mathbb{P}_{(\mathbf{x},\gamma)}\left(\inf_{t\in[0,T]}\left|1-\frac{\mathsf{x}_{2}(t)}{\mathsf{x}_1(t)}\right|>0\right)=1,
\end{align}
which completes the proof of the case $n=1$.

\textbf{The inductive step.} We next assume that \eqref{X_n>0} and  \eqref{ratio_n<1} hold for $n=k$, and prove them for $n=k+1$.
From \eqref{logX^N}, we can write
\begin{align*}
    \log\mathsf{x}_{k+1}^{(N)}\left(t\wedge \sigma_{k+1}^N(\epsilon)\wedge \tau_{k}^N(\delta)\wedge \theta_{k+1}^N(r)\right) &\geq
    \log x_{k+1}^{(N)}
    +\mathsf{w}_{k+1}\left(t\wedge \sigma_{k+1}^N(\epsilon)\wedge \tau_{k}^N(\delta)\wedge \theta_{k+1}^N(r)\right)\\ 
 -\left|\frac{1+\eta}{2}\right|&t+\int_{0}^{t\wedge \sigma_{k+1}^N(\epsilon)\wedge \tau_{k}^N(\delta)\wedge \theta_{k+1}^N(r)}
    \sum_{j=1, j\neq k+1}^{k}\frac{\mathsf{x}_j^{(N)}(t)}{\mathsf{x}_{k+1}^{(N)}(t)-\mathsf{x}_j^{(N)}(t)}\mathrm{d}t\\
\geq
    \log x_{k+1}^{(N)}&
    -\left|\frac{1+\eta}{2}\right|t+\mathsf{w}_{k+1}\left(t\wedge \sigma_{k+1}^N(\epsilon)\wedge \tau_{k}^N(\delta)\wedge \theta_{k+1}^N(r)\right)-\frac{kt}{\delta},
\end{align*}
 and thus, we have,
\begin{align*}
    \log r+\log \epsilon\mathbb{P}_{\mathsf{x}^{(N)}}\left(\sigma_{k+1}^N(\epsilon)< t\wedge \tau_{k}^N(\delta)\wedge \theta_{k+1}^N(r)\right)&\geq \log x_{k+1}^{(N)}-\left|\frac{1+\eta}{2}\right|t-\frac{kt}{\delta}.
\end{align*}
 One can check, by making use of the induction hypothesis, that as previously, there exist sequences $(\epsilon_m)_{m=1}^\infty$, $(\delta_\ell)_{\ell=1}^\infty$, $(r_q)_{q=1}^\infty$, with $\epsilon_m,\delta_\ell \to 0$, $r_q \to \infty$, as $m,\ell,q \to \infty$, so that the corresponding stopping times converge in distribution, as $N \to \infty$, to their limiting counterparts, and in particular we obtain,
 \begin{align*}
   \mathbb{P}_{(\mathbf{x},\gamma)}\left(\lim_{m\to \infty}\sigma_{k+1}(\epsilon_m)< t\wedge \tau_{k}(\delta_\ell)\wedge \theta_{k+1}(r_q)\right)=0.
\end{align*}
 Hence, it follows that $\mathbb{P}_{(\mathbf{x},\gamma)}\left(\lim_{m \to \infty}\sigma_{k+1}(\epsilon_m)=\infty\right)=1$, if we take $m \to \infty$ first, and then, $t,\ell,q\to\infty$ whereby $\tau_{k}(\delta_\ell),\theta_{k+1}(r_q)\overset{\textnormal{a.s.}}{\longrightarrow}\infty$. Equivalently,
\begin{align*}
\mathbb{P}_{(\mathbf{x},\gamma)}\left(\inf_{t\in[0,T]}\mathsf{x}_{k+1}(t)>0
\right)=1.
\end{align*}
 
We next apply It\^{o}'s formula to $f_{k+1}^N$ to obtain
\begin{align*}
   \mathrm{d}f_{k+1}^N\left(\mathsf{x}^{(N)}(t)\right)&=
    \mathrm{d}\mathsf{M}_{k+1}^{(N)}(t)+
    \sum_{i=1}^{k+1}\frac{N-k-1}{2N\mathsf{x}_i^{(N)}(t)}\mathrm{d}t-\sum_{i=1}^{k+1}\sum_{j=k+2}^{N}\frac{\left(\mathsf{x}_j^{(N)}(t)\right)^2}{\left(\mathsf{x}_i^{(N)}(t)-\mathsf{x}_j^{(N)}(t)\right)^2}\mathrm{d}t  \\&\  \ +
   \sum_{i=1}^{k+1}\sum_{j=k+2}^N\sum_{\underset{\ell\neq i,j}{\ell=1}}^N
\frac{\mathsf{x}_j^{(N)}(t)\mathsf{x}_\ell^{(N)}(t)}{\left(\mathsf{x}_i^{(N)}(t)-\mathsf{x}_\ell^{(N)}(t)\right)\left(\mathsf{x}_{j}^{(N)}(t)-\mathsf{x}_\ell^{(N)}(t)\right)}\mathrm{d}t,  
\end{align*}
where $t\mapsto \mathsf{M}_{k+1}^{(N)}(t)$ stands for the corresponding local martingale term.
We note, and this is again important, that the last term in the above expression is non-positive for any $t\ge 0$, by virtue of the inequality, for $\mathbf{y}\in \mathbb{W}_{N,+}^\circ$,
\begin{equation*}
   \sum_{i=1}^{k+1}\sum_{j=k+2}^N\sum_{\underset{\ell\neq i,j}{\ell=1}}^N \frac{y_j y_{\ell}}{(y_i-y_\ell)(y_j-y_\ell)}  \le 0,
\end{equation*}
which is easily seen to be true since the LHS can be written as,
\begin{align*}
&\sum_{i=1}^{k+1}\sum_{j=k+2}^N\sum_{\ell=1,\ell\neq i}^{k+1} \frac{y_j y_{\ell}}{(y_i-y_\ell)(y_j-y_\ell)} +\sum_{i=1}^{k+1}\sum_{j=k+2}^N\sum_{\underset{\ell\neq j}{\ell=k+2}}^N \frac{y_j y_{\ell}}{(y_i-y_\ell)(y_j-y_\ell)}
\\ &=
\sum_{j=k+2}^N\sum_{\underset{i<\ell}{i,\ell=1}}^{k+1}\frac{y_j}{y_i-y_\ell}\left(\frac{y_{\ell}}{y_j-y_\ell}-\frac{y_{i}}{y_j-y_i}\right) +\sum_{i=1}^{k+1}\sum_{\underset{j<\ell}{j,\ell=k+2}}^N\frac{y_j y_\ell}{y_j-y_\ell}\left(\frac{1}{y_i-y_\ell}-\frac{1}{y_i-y_j}\right) \le 0.
\end{align*}
Thus, by dropping it, we have
\begin{align*}
    f_{k+1}^N\left(\mathsf{x}^{(N)}\left(t\wedge \kappa_{k+1}^N(R)\wedge \sigma_{k+1}^N(\epsilon)\right)\right)& \leq f_{k+1}^N\left(\mathbf{x}^{(N)}\right)+\mathsf{M}_{k+1}^{(N)}\left(t\wedge \kappa_{k+1}^N(R)\wedge \sigma_{k+1}^N(\epsilon)\right)\\
&\ 
 \ +\sum_{i=1}^{k+1}\int_{0}^{t\wedge \kappa_{k+1}^N(R)\wedge \sigma_{k+1}^N(\epsilon)}
    \frac{1}{2\mathsf{x}_n^{(N)}(t)}\mathrm{d}t\\
&\leq  f_{k+1}^N\left(\mathbf{x}^{(N)}\right)+\mathsf{M}_{k+1}^{(N)}\left(t\wedge \kappa_{k+1}^N(R)\wedge \sigma_{k+1}^N(\epsilon)\right)+ 
    \frac{k}{\epsilon}t.
\end{align*}
One can now follow the argument for the case $n=1$ to establish the desired result in the exact same way.
\end{proof}

\begin{rmk}
We note that in the argument above we only needed to take as input convergence of the paths and not necessarily convergence in $C(\mathbb{R}_+,\Omega_+)$. Had we not even known convergence of the paths, the argument would in any case give that any possible subsequential limits would need to consist of non-intersecting paths.
\end{rmk}

\subsection{The ISDE via characteristic polynomials}

We now prove that $\left(\mathsf{x}_i(\cdot)\right)_{i=1}^\infty$ solves the ISDE \eqref{ISDEintro}. Beyond the use\footnote{We also present a slightly shorter argument (directly using convergence on $\Omega_+$ and non-intersection of the paths) for convergence of the singular drift term in Remark \ref{AlternativeConvergenceArgument} that does not involve the use of characteristic polynomials. Nevertheless, these two arguments are basically equivalent, since convergence on $\Omega_+$ implies convergence of the corresponding polynomials.} of characteristic polynomials to show convergence of the singular drift, it is interesting to note the following. In the construction of weak solutions to finite-dimensional SDE, see \cite{Stroock-Varadhan}, the driving Brownian motions can usually\footnote{The upshot is that one wants to construct the driving Brownian motions from the solution, see \cite{Stroock-Varadhan}. If the noise degenerates, one needs an auxiliary, independent from the solution, Brownian motion, see \cite{Stroock-Varadhan}. However, in our setting the noise does not degenerate as $\mathsf{x}_i>0$ and this is not the source of the additional randomness. } be taken to be measurable with respect to the natural filtration generated by the solution. Here, the driving Brownian motions $(\mathsf{w}_i(\cdot))_{i=1}^\infty$ will in fact be measurable not with the filtration generated by $\left(\mathsf{x}_i(\cdot)\right)_{i=1}^\infty$ on its own but rather with the filtration generated by the enhanced process $\mathbf{X}_\cdot^{\Omega_+}=\left((\mathsf{x}_i(\cdot))_{i=1}^\infty,\boldsymbol{\gamma}(\cdot)\right)$. This is another illustration of the phenomenon that in the infinite-dimensional limit additional information is created.

\begin{thm}\label{thm-ISDE}
    Let $\eta \in \mathbb{R}$, $\mathbf{x} \in \mathbb{W}_{\infty,+}^\circ$, $\gamma\in \mathbb{R}_+$ with $\sum_{i=1}^{\infty}x_i\le \gamma$.
    Let $\mathbf{X}_\cdot^{\Omega_+}$ and $\mathsf{x}(\cdot)=\left(\mathsf{x}_i(\cdot)\right)_{i=1}^\infty$ be as in Theorem \ref{thm-NonCollision,>0}. Then, $\mathsf{x}(\cdot)$ is a weak solution, starting from $\mathbf{x}$, to the ISDE \eqref{ISDEintro}. Moreover, we have that,
 \begin{equation}\label{Claim-gamma=sum}
\mathbb{P}_{(\mathbf{x},\gamma)}\left(\mathbf{X}^{\Omega_+}_t\in \Omega_+^{0},\;\forall t>0\right)=1, \  \ \forall (\mathbf{x},\gamma)\in \Omega_+.
    \end{equation}
\end{thm}

\begin{proof}[Proof of Theorem \ref{thm-ISDE}]
We first prove that there exists a filtered probability space, on which independent standard Brownian motions $(\tilde{\mathsf{w}}_i)_{i=1}^\infty$ adapted to the filtration and an equal in law copy $(\tilde{\mathsf{x}},\tilde{\boldsymbol{\gamma}})$ of $(\mathsf{x},\boldsymbol{\gamma})$, also adapted to the filtration, are defined, such that almost surely, for all $t\ge 0$:
   	\begin{align}\label{ISDE-Omega}	
    \tilde{\mathsf{x}}_i(t) = x_i+ \int_{0}^t\tilde{\mathsf{x}}_i(s) \mathrm{d}\tilde{\mathsf{w}}_i(s) +\int_0^t \left(-\frac{\eta}{2}\tilde{\mathsf{x}}_i(s)+\tilde{\boldsymbol{\gamma}}(s)-\sum_{j=1}^{\infty}\tilde{\mathsf{x}}_j(s)+\sum_{j=1,j\neq i}^{\infty}\frac{\tilde{\mathsf{x}}_i(s)\tilde{\mathsf{x}}_j(s)}{\tilde{\mathsf{x}}_i(s)-\tilde{\mathsf{x}}_j(s)}\right)\mathrm{d}s, \  \  i\in\mathbb{N}.
	\end{align}
The filtration will be the natural filtration of $(\tilde{\mathsf{x}},\tilde{\boldsymbol{\gamma}})$. Recall that the natural filtration $(\mathscr{F}_t)_{t\ge 0}$ of a process $(\mathfrak{R}_t;t\ge 0)$ is given by, for all $t\ge 0$,
\begin{equation*}
\mathscr{F}_t=\sigma\left(\mathfrak{R}_t; 0 \le s \le t\right),
\end{equation*}
 where for a collection of random variables $\mathcal{C}$, we write $\sigma (\mathcal{C})$ for the sigma algebra generated by them. 
 
We now prove this claim. For each $N\in\mathbb{N}$, let $\mathbf{x}^{(N)}\in \mathbb{W}_{N,+}^\circ$ be such that it converges, under the embedding \eqref{embedding+}, to $(\mathbf{x},\gamma)$.  Let $\mathsf{x}^{(N)}$ be the solution to the equation \eqref{rescaledSDE} starting at $\mathbf{x}^{(N)}$. 
Writing the equation in the integral form, one has for $i=1,2,\dots,N$,
\begin{align}\label{rescaledSDEintegralEq}
	\mathsf{x}_i^{(N)}(t)  =	x_i^{(N)} +\int_{0}^{t}\mathsf{x}_i^{(N)}(s) \mathrm{d}\mathsf{w}_i(s) - \frac{\eta}{2}\int_{0}^{t}\mathsf{x}_i^{(N)}(s)\mathrm{d}s
	+\frac{t}{2N}
	+\int_{0}^{t}\sum_{j=1,j\neq i}^{N}\frac{\mathsf{x}_i^{(N)}(s)\mathsf{x}_j^{(N)}(s)}{\mathsf{x}_i^{(N)}(s)-\mathsf{x}_j^{(N)}(s)}\mathrm{d}s. 
\end{align}   
  By Theorem \ref{MainThm1Intro}, $\mathbf{X}^{(N)}\overset{\textnormal{d}}{\longrightarrow}(\mathsf{x},\boldsymbol{\gamma})$ on $C(\mathbb{R}_+, \Omega_+)$ where as usual, $\mathbf{X}^{(N)}$ corresponds to the embedded process on $\Omega_+$. Now, by the Skorokhod
 representation theorem, there exists a probability space $(\tilde{\boldsymbol{\Omega}},\tilde{\mathbf{P}})$ on which 
 $\tilde{\mathsf{x}}^{(N)} \overset{\textnormal{d}}{=} \mathsf{x}^{(N)} $, i.e.  $\tilde{\mathbf{X}}^{(N)} \overset{\textnormal{d}}{=} \mathbf{X}^{(N)} $, and $\tilde{\mathbf{X}}^{\Omega_+}\overset{\textnormal{d}}{=}\mathbf{X}^{\Omega_+}$, i.e. $(\tilde{\mathsf{x}},\tilde{\boldsymbol{\gamma}})\overset{\textnormal{d}}{=}(\mathsf{x},\boldsymbol{\gamma})$, 
  are defined such that $\tilde{\mathbf{P}}$-a.s.,
  \begin{align}\label{convOnOmega}
\tilde{\mathbf{X}}^{(N)}\overset{N\to\infty}{\longrightarrow} (\tilde{\mathsf{x}},\tilde{\boldsymbol{\gamma}}),\  \  \textnormal{on}\  \  C(\mathbb{R}_+, \Omega_+).
  \end{align}
Let us fix $N \in \mathbb{N}$. Let us write $(\mathscr{F}_t^{(N)})_{t\ge 0}$ for the natural filtration of 
$\tilde{\mathsf{x}}^{(N)}$. We first show that there exist  $(\mathscr{F}_t^{(N)})_{t\ge 0}$-adapted independent standard Brownian motions $(\tilde{\mathsf{w}}_{i}^{(N)})_{i=1}^N$ so that $\tilde{\mathbf{P}}$-a.s., for all $t\ge 0$ and  $i=1,2,\dots,N$, we have
 \begin{align}\label{tilde-FSDE}
			\tilde{\mathsf{x}}_i^{(N)}(t)  =	x_i^{(N)}+\int_{0}^{t}\tilde{\mathsf{x}}_i^{(N)}(s) \mathrm{d}\tilde{\mathsf{w}}_i^{(N)}(s) - \frac{\eta}{2}\int_{0}^{t}\tilde{\mathsf{x}}_i^{(N)}(s)\mathrm{d}s
  +\frac{t}{2N}
   +\int_{0}^{t}\sum_{j=1,j\neq i}^{N}\frac{\tilde{\mathsf{x}}_i^{(N)}(s)\tilde{\mathsf{x}}_j^{(N)}(s)}{\tilde{\mathsf{x}}_i^{(N)}(s)-\tilde{\mathsf{x}}_j^{(N)}(s)}\mathrm{d}s. 
		\end{align} 
Let us define the process $\left(\left(\mathsf{Y}_i^{(N)}(t)\right)_{i=1}^N;t \ge 0\right)$ coordinate-wise, as follows,
\begin{align*}
\left(\mathsf{Y}_i^{(N)}(t);t\ge 0\right)&\overset{\textnormal{def}}{=}\left(\tilde{\mathsf{x}}_i^{(N)}(t) -	x_i^{(N)} + \frac{\eta}{2}\int_{0}^{t}\tilde{\mathsf{x}}_i^{(N)}(s)\mathrm{d}s
  -\frac{t}{2N}
   -\int_{0}^{t}\sum_{j=1,j\neq i}^{N}\frac{\tilde{\mathsf{x}}_i^{(N)}(s)\tilde{\mathsf{x}}_j^{(N)}(s)}{\tilde{\mathsf{x}}_i^{(N)}(s)-\tilde{\mathsf{x}}_j^{(N)}(s)}\mathrm{d}s;t\ge 0\right)\\
   &\overset{\textnormal{d}}{=}\left(\int_0^t \mathsf{x}_i^{(N)} (s)\mathrm{d}\mathsf{w}_i(s)\;t\ge 0\right).
\end{align*}
The equality in distribution is by virtue of the fact that $\tilde{\mathsf{x}}^{(N)} \overset{\textnormal{d}}{=} \mathsf{x}^{(N)} $ and that $\mathsf{x}^{(N)}$ is a solution to the SDE \eqref{rescaledSDE}. Observe that, $\mathsf{Y}^{(N)}$ is adapted with respect to $(\mathscr{F}_t^{(N)})_{t\ge 0}$. It is also easy to check that $\mathsf{Y}^{(N)}$ is an $N$-dimensional continuous local martingale with respect to 
$(\mathscr{F}_t^{(N)})_{t\ge 0}$, with quadratic variation\footnote{For a continuous $N$-dimensional local martingale $(\mathsf{z}_i(\cdot))_{i=1}^N$ we write $\left\langle \mathsf{z}_i, \mathsf{z}_j \right \rangle_t$, $i,j=1,\dots,N$ for its quadratic (co-)variation.}  \cite{Revuz-Yor,JacodShiryaev} given by,
\begin{equation}
\left\langle \mathsf{Y}_i^{(N)},\mathsf{Y}_j^{(N)}\right\rangle_t=\mathbf{1}_{i=j}\int_{0}^t \tilde{\mathsf{x}}_i^{(N)}(s)\tilde{\mathsf{x}}_j^{(N)}(s)\mathrm{d}s, \ \ \forall t \ge 0, \ i,j=1,2,\dots,N. 
\end{equation}
Now, we can define, for $i=1,\dots,N$, by virtue of the fact that $\tilde{\mathbf{P}}$-a.s. $\tilde{\mathsf{x}}_i^{(N)}>0$, the stochastic integral,
\begin{equation}\label{BM^Ndef}
\tilde{\mathsf{w}}_i^{(N)}(t)\overset{\textnormal{def}}{=} \int_{0}^t \frac{1}{\tilde{\mathsf{x}}_i^{(N)}(s)}\mathrm{d}\mathsf{Y}_i^{(N)}(s), \ \ \forall t \ge 0.
\end{equation}
Observe that, $(\tilde{\mathsf{w}}_i^{(N)})_{i=1}^N$ is adapted with respect to $(\mathscr{F}_t^{(N)})_{t\ge 0}$ and moreover it is in fact an $N$-dimensional continuous local martingale, with respect to this filtration, satisfying,
\begin{equation*}
\left\langle \tilde{\mathsf{w}}_i^{(N)},\tilde{\mathsf{w}}_j^{(N)}\right\rangle_t=\mathbf{1}_{i=j}t, \ \ \forall t \ge 0, \ i,j=1,2,\dots,N. 
\end{equation*}
Hence, by virtue of Levy's characterisation theorem \cite{Revuz-Yor}, $(\tilde{\mathsf{w}}_i^{(N)})_{i=1}^N$ is a sequence of independent standard Brownian motions. Moreover, from \eqref{BM^Ndef}, we have
\begin{equation*}
\mathsf{Y}_i^{(N)}(t)=\int_{0}^t \tilde{\mathsf{x}}_i^{(N)}(s)\mathrm{d}\tilde{\mathsf{w}}_i^{(N)}(s), \ \ \forall t \ge 0,\  \ i=1,2,\dots,N,
\end{equation*}
from which \eqref{tilde-FSDE} follows.

We next prove \eqref{ISDE-Omega} holds. 
Clearly, proving \eqref{ISDE-Omega} amounts to showing the convergence of the integrals to the corresponding terms in the limit. Let $i\in\mathbb{N}$ be fixed. We proceed to show the convergence of the drift terms first. Observe that, $\tilde{\mathbf{P}}$-a.s., for all $t\ge 0$,
\begin{align*}
\int_{0}^{t}\tilde{{\mathsf{x}}}_i^{(N)}(s)\mathrm{d}s\to \int_{0}^{t}\tilde{{\mathsf{x}}}_i(s)\mathrm{d}s, \  \ \text{ as } N\to \infty.
\end{align*}
We thus need to prove convergence of the interaction terms, namely, $\tilde{\mathbf{P}}$-a.s., for all $t \ge 0$,
\begin{align}\label{convInteractionTerm}
	\int_{0}^{t}\sum_{j=1,j\neq i}^{N}\frac{\tilde{\mathsf{x}}_i^{(N)}(s)\tilde{\mathsf{x}}_j^{(N)}(s)}{\tilde{\mathsf{x}}_i^{(N)}(s)-\tilde{\mathsf{x}}_j^{(N)}(s)}\mathrm{d}s\overset{N\to \infty}{\longrightarrow}
\int_{0}^{t}\left(\tilde{\boldsymbol{\gamma}}(s)-\sum_{j=1}^{\infty}\tilde{\mathsf{x}}_j(s)\right)\mathrm{d}s
   +\int_{0}^{t}\sum_{j=1,j\neq i}^{\infty}\frac{\tilde{\mathsf{x}}_i(s)\tilde{\mathsf{x}}_j(s)}{\tilde{\mathsf{x}}_i(s)-\tilde{\mathsf{x}}_j(s)}\mathrm{d}s.
		\end{align}
To this end, we will use Proposition \ref{prop-CharPolyConv} on convergence of the corresponding characteristic polynomials.
Let $N\in\mathbb{N}$ and $1\leq i\leq N$. We consider the following entire functions: \begin{align*}
	\Phi_i^N\left(z;\mathbf{x}^{(N)}\right)&\overset{\textnormal{def}}{=}\prod_{j=1,j\neq i}^{N}\left(1-x_j^{(N)}z\right)^2,\  \  \mathbf{x}^{(N)}\in\mathbb{W}_{N,+},\\
	\mathsf{E}_i(z;\omega)&\overset{\textnormal{def}}{=}\mathrm{e}^{-2(\gamma-x_i) z}\prod_{j=1,j\neq i}^{\infty}\mathrm{e}^{2x_jz}\left(1-x_jz\right)^2,\  \ \omega=(\mathbf{x},\gamma)\in\Omega_+.
\end{align*}
Observe that, the integrands in \eqref{convInteractionTerm} are determined in terms of the functions $\Phi_i^N\left(z;\mathbf{x}^{(N)}\right)$ and $\mathsf{E}_i(z;\omega)$ respectively. Indeed, 
the corresponding terms, at time $t\ge 0$, are equal to 
\begin{align*}
& \frac{1}{2}\frac{\mathrm{d}}{\mathrm{d} z}\log \Phi_i^N\left(z,\tilde{\mathsf{x}}^{(N)}(t)\right)\Big|_{z=\left(\tilde{\mathsf{x}}_i^{(N)}(t)\right)^{-1}},\\
& \frac{1}{2}\frac{\mathrm{d}}{\mathrm{d} z}\log 
\mathsf{E}_i\left(z;\tilde{\mathbf{X}}_t^{\Omega_+}\right)\Big|_{z=\left(\tilde{\mathsf{x}}_i(t)\right)^{-1}}.
\end{align*}
Note that, the expressions above are well-defined by virtue of the fact that the paths are $\tilde{\mathbf{P}}$-a.s. non-intersecting and never hit zero. We now prove that, $\tilde{\mathbf{P}}$-a.s., for all $t\ge 0$,
\begin{align}\label{LogDerConv}
\frac{\mathrm{d}}{\mathrm{d} z}\log \Phi_i^N\left(z,\tilde{\mathsf{x}}^{(N)}(t)\right)\Big|_{z=\left(\tilde{\mathsf{x}}_i^{(N)}(t)\right)^{-1}}\overset{N\to\infty}{\longrightarrow} \frac{\mathrm{d}}{\mathrm{d}z}\log 
\mathsf{E}_i\left(z;\tilde{\mathbf{X}}_t^{\Omega_+}\right)\Big|_{z=\left(\tilde{\mathsf{x}}_i(t)\right)^{-1}}.
\end{align}
 Since we have convergence on $\Omega_+$, as in \eqref{convOnOmega}, it follows from
 Proposition \ref{prop-CharPolyConv}, in particular equation \eqref{CharPolyConv+}, that $\tilde{\mathbf{P}}$-a.s., for all $t\ge 0$,
\begin{align*}
\Phi_i^N\left(z,\tilde{\mathsf{x}}^{(N)}(t)\right)\longrightarrow \mathsf{E}_i\left(z;\tilde{\mathbf{X}}_t^{\Omega_+}\right),\  \  \textnormal{ as } N\to\infty,
\end{align*}
uniformly on compact sets in $\mathbb{C}$, and the same is true for the corresponding derivatives by analyticity.
Note that, by virtue of Theorem \ref{thm-NonCollision,>0} we have $\tilde{\mathbf{P}}$-a.s., for all $t\ge 0$
\begin{align*}
\left(\tilde{\mathsf{x}}_i^{(N)}(t)\right)^{-1}\longrightarrow \left(\tilde{\mathsf{x}}_i(t)\right)^{-1}, \  \  \textnormal{as }N\to\infty.
\end{align*}
From the above, and the fact that all processes we are considering are $\tilde{\mathbf{P}}$-a.s. non-intersecting, see Lemma \ref{lem-wellPosednessN} and Theorem \ref{thm-NonCollision,>0}, the convergence in \eqref{LogDerConv} is concluded. This shows convergence of the integrands for all times $t\ge 0$. We now prove that the corresponding integrals also converge.
First, note that using the uniform convergence and the non-intersection property of the coordinates, we have, $\tilde{\mathbf{P}}$-a.s., for all $t\ge 0$,
   \begin{equation*}
	\int_{0}^{t}\sum_{j=1}^{i-1}\frac{\tilde{\mathsf{x}}_i^{(N)}(s)\tilde{\mathsf{x}}_j^{(N)}(s)}{\tilde{\mathsf{x}}_i^{(N)}(s)-\tilde{\mathsf{x}}_j^{(N)}(s)}\mathrm{d}s\overset{N\to \infty}{\longrightarrow}
\int_{0}^{t}\sum_{j=1}^{i-1}\frac{\tilde{\mathsf{x}}_i(s)\tilde{\mathsf{x}}_j(s)}{\tilde{\mathsf{x}}_i(s)-\tilde{\mathsf{x}}_j(s)}\mathrm{d}s.
		\end{equation*} 
Thus, it suffices to show that  $\tilde{\mathbf{P}}$-a.s., for all $t\ge 0$,
 \begin{align*}
	\int_{0}^{t}\sum_{j=i+1}^{N}\frac{\tilde{\mathsf{x}}_i^{(N)}(s)\tilde{\mathsf{x}}_j^{(N)}(s)}{\tilde{\mathsf{x}}_i^{(N)}(s)-\tilde{\mathsf{x}}_j^{(N)}(s)}\mathrm{d}s\overset{N\to \infty}{\longrightarrow}
\int_{0}^{t}\left(\tilde{\boldsymbol{\gamma}}(s)-\sum_{j=1}^{\infty}\tilde{\mathsf{x}}_j(s)\right)\mathrm{d}s
   +\int_{0}^{t}\sum_{j=i+1}^{\infty}\frac{\tilde{\mathsf{x}}_i(s)\tilde{\mathsf{x}}_j(s)}{\tilde{\mathsf{x}}_i(s)-\tilde{\mathsf{x}}_j(s)}\mathrm{d}s.
		\end{align*} 
     Observe that, as paths are strictly ordered,
 we have
 \begin{align*}
        \sum_{j= i+1}^{N}\frac{\tilde{\mathsf{x}}_i^{(N)}(s)\tilde{\mathsf{x}}_j^{(N)}(s)}{\tilde{\mathsf{x}}_i^{(N)}(s)-\tilde{\mathsf{x}}_j^{(N)}(s)}
        \leq  \frac{\tilde{\mathsf{x}}_i^{(N)}(s)}{\tilde{\mathsf{x}}_i^{(N)}(s)-\tilde{\mathsf{x}}_{i+1}^{(N)}(s)}\sum_{j=i+1}^{N}\tilde{\mathsf{x}}_j^{(N)}(s), \  \  \forall s \ge 0,
    \end{align*}
 where, by virtue of \eqref{convOnOmega} and the non-intersection property of the limiting paths, the right hand side converges $\tilde{\mathbf{P}}$-a.s., uniformly on $[0,t]$. The desired result then follows by applying the generalized Lebesgue dominated convergence theorem.

We now deal with the stochastic integral term. From now on, let us denote by $(\mathscr{F}_t)_{t\ge 0}$ the natural filtration of the process $\left(\tilde{\mathbf{X}}^{\Omega_+}_t; t\geq 0\right)$. We deduce from the convergence of both the LHS and the drift term in \eqref{tilde-FSDE} that $\tilde{\mathbf{P}}$-a.s., for all $i\in \mathbb{N}$,
$\mathsf{Y}_i^{(N)}(\cdot)$ converges, as $N \to \infty$, uniformly on compact sets,
to the continuous process $\mathsf{Y}_i(\cdot)$ given by,
\begin{align*}
\mathsf{Y}_i(t)= \tilde{\mathsf{x}}_i(t) -x_i-\int_0^t \left(-\frac{\eta}{2}\tilde{\mathsf{x}}_i(s)+\tilde{\boldsymbol{\gamma}}(s)-\sum_{j=1}^{\infty}\tilde{\mathsf{x}}_j(s)+\sum_{j=1,j\neq i}^{\infty}\frac{\tilde{\mathsf{x}}_i(s)\tilde{\mathsf{x}}_j(s)}{\tilde{\mathsf{x}}_i(s)-\tilde{\mathsf{x}}_j(s)}\right)\mathrm{d}s, \ \ \forall t \ge 0.
\end{align*}
Observe that, $\mathsf{Y}(\cdot)=\left(\mathsf{Y}_i(\cdot)\right)_{i=1}^\infty$ is adapted with respect to $(\mathscr{F}_t)_{t\ge 0}$. We now claim that for any $K \in \mathbb{N}$, $\left(\mathsf{Y}_i(\cdot)\right)_{i=1}^K$ is a $K$-dimensional continuous local martingale with quadratic variation given by,
\begin{equation*}
\langle \mathsf{Y}_i,\mathsf{Y}_j\rangle_t=\mathbf{1}_{i=j}\int_{0}^t \tilde{\mathsf{x}}_i(s)\tilde{\mathsf{x}}_j(s)\mathrm{d}s, \ \ \forall t \ge 0, \ i,j \in \mathbb{N}.
\end{equation*}
Let us assume this claim momentarily. Then, we can define, by virtue of the fact that $\tilde{\mathbf{P}}$-a.s. $\tilde{\mathsf{x}}_i>0$, for all $i \in \mathbb{N}$, the stochastic integral,
\begin{equation}\label{BMdef}
\tilde{\mathsf{w}}_i(t)\overset{\textnormal{def}}{=} \int_{0}^t \frac{1}{\tilde{\mathsf{x}}_i(s)}\mathrm{d}\mathsf{Y}_i(s), \ \ \forall t \ge 0.
\end{equation}
Observe that $\left(\tilde{\mathsf{w}}_i(\cdot)\right)_{i=1}^\infty$ is adapted with respect to $(\mathscr{F}_t)_{t\ge 0}$ and moreover, for each $K \in \mathbb{N}$, $(\tilde{\mathsf{w}}_i(\cdot))_{i=1}^K$ is a $K$-dimensional continuous local martingale with quadratic variation,
\begin{equation*}
\langle \tilde{\mathsf{w}}_i,\tilde{\mathsf{w}}_j\rangle_t=\mathbf{1}_{i=j}t, \ \ \forall t \ge 0, \ i,j \in \mathbb{N}. 
\end{equation*}
Thus, making use of Levy's characterisation theorem \cite{Revuz-Yor} we get that $(\tilde{\mathsf{w}}_i)_{i=1}^\infty$ is a sequence of independent standard Brownian motions. Finally, from \eqref{BMdef} we get, for all $i\in \mathbb{N}$,
\begin{equation*}
\mathsf{Y}_i(t)=\int_{0}^t \tilde{\mathsf{x}}_i(s)\mathrm{d}\tilde{\mathsf{w}}_i(s), \ \ \forall t \ge 0.
\end{equation*}
This concludes the proof of \eqref{ISDE-Omega} modulo the claim. We now prove it. Fix $K\in \mathbb{N}$. Define the stopping times $\zeta_K^N(R)$ and $\zeta_K(R)$, with respect to the corresponding filtrations $(\mathscr{F}_t^{(N)})_{t\ge 0}$ and $(\mathscr{F}_t)_{t\ge 0}$, for $R>0$, $N \ge K $,
\begin{align*}
\zeta_K^N(R)\overset{\textnormal{def}}{=}\inf\left\{t\geq 0;\sum_{i=1}^K \left|\mathsf{Y}_i^{(N)}(t)\right|\ge R\right\}, \ \ 
\zeta_K(R)\overset{\textnormal{def}}{=}\inf\left\{t\geq 0;\sum_{i=1}^K \left|\mathsf{Y}_i(t)\right|\ge R\right\}. 
\end{align*}
Observe that, $\tilde{\mathbf{P}}$-a.s., $\zeta_K(R) \overset{R \to \infty}{\longrightarrow} \infty $. Moreover, arguing similarly to the proof of Theorem \ref{thm-NonCollision,>0} we get that there exists a sequence $(R_m)_{m=1}^\infty$, with $R_m\to \infty$ so that for any fixed $m \in \mathbb{N}$, $\zeta_K^N(R_m) \longrightarrow \zeta_K(R_m)$, in distribution, as $N \to \infty$. Let $t\ge s$ and $R_m$ be fixed. Let $\mathfrak{G}$ be an arbitrary bounded continuous functional on $C([0,s],\Omega_+)$, the space of continuous functions on $[0,s]$ with values in $\Omega_+$. Note that, we have, as $N \to \infty$,
\begin{align*}
& \left(\mathsf{Y}^{(N)}_i\left(t\wedge \zeta^N_K(R_m)\right)-\mathsf{Y}^{(N)}_i\left(s\wedge \zeta^N_K(R_m)\right)\right)_{i=1}^K\mathfrak{G}\left[\left(\tilde{\mathbf{X}}^{(N)}(u);0\le u \le s\right)\right] \\ &\overset{\textnormal{d}}{\longrightarrow} \left(\mathsf{Y}_i\left(t\wedge \zeta_K(R_m)\right)-\mathsf{Y}_i\left(s\wedge \zeta_K(R_m)\right)\right)_{i=1}^K\mathfrak{G}\left[\left(\tilde{\mathbf{X}}^{\Omega_+}(u);0\le u \le s\right)\right].
\end{align*}
Moreover, observe that, if for fixed $N \in \mathbb{N}$, we define the functional $\mathsf{G}$ on $C([0,s],\mathbb{W}_{N,+})$ by,
\begin{equation*}
\mathsf{G}\left[\left(\left(\mathrm{f}_i(u)\right)_{i=1}^N;0\le u \le s \right)\right] =\mathfrak{G}\left[\left(\left(\left(N^{-1}\mathrm{f}_i(u)\right)_{i=1}^\infty,N^{-1}\sum_{i=1}^N\mathrm{f}_i(u)\right); 0 \le u \le s\right)\right],
\end{equation*}
with the convention $\mathrm{f}_i\equiv 0$, for $i>N$, then $\mathsf{G}$ is bounded and continuous. Hence, putting everything together, by virtue of the local martingale property of $\mathsf{Y}^{(N)}$, we obtain,
\begin{align*}
&\mathbb{E}_{(\mathbf{x},\gamma)}\left[\left(\mathsf{Y}_i(t\wedge \zeta_K(R_m))-\mathsf{Y}_i(s\wedge \zeta_K(R_m))\right)_{i=1}^K\mathfrak{G}\left[\left(\tilde{\mathbf{X}}^{\Omega_+}_u;0\le u \le s\right)\right]\right]\\
&=\lim_{N \to \infty} \mathbb{E}_{\mathbf{x}^{(N)}}\left[\left(\mathsf{Y}^{(N)}_i\left(t\wedge \zeta^N_K(R_m)\right)-\mathsf{Y}^{(N)}_i\left(s\wedge \zeta^N_K(R_m)\right)\right)_{i=1}^K\mathfrak{G}\left[\left(\tilde{\mathbf{X}}^{(N)}(u);0\le u \le s\right)\right]\right]\\
&=\lim_{N \to \infty} \mathbb{E}_{\mathbf{x}^{(N)}}\left[\left(\mathsf{Y}^{(N)}_i\left(t\wedge \zeta^N_K(R_m)\right)-\mathsf{Y}^{(N)}_i\left(s\wedge \zeta^N_K(R_m)\right)\right)_{i=1}^K\mathsf{G}\left[\left(\tilde{\mathsf{x}}^{(N)}(u);0\le u \le s\right)\right]\right]=(0)_{i=1}^K.
\end{align*}
This proves that $\left(\mathsf{Y}_i(\cdot)\right)_{i=1}^K$ is a continuous local martingale, with $(\zeta_K(R_m))_{m=1}^\infty$ a localising sequence, with respect to $\left(\mathscr{F}_t\right)_{t\ge 0}$. Finally note that, $\tilde{\mathbf{P}}$-a.s. the quadratic variations of the $\mathsf{Y}^{(N)}$'s also converge, as $N \to \infty$, to the desired expression and this concludes the proof of the claim and thus the proof of \eqref{ISDE-Omega}.

We now proceed to show that $\tilde{\mathsf{x}}$ actually solves \eqref{ISDEintro}. It will suffice to show that $\tilde{\mathbf{P}}$-a.s., for all $t\ge 0$, we have  
\begin{equation*}
\int_0^t\left(\tilde{\boldsymbol{\gamma}}(s)-\sum_{i=1}^{\infty}\tilde{\mathsf{x}}_i(s)\right)\mathrm{d}s= 0.
\end{equation*}
Observe that, for any $M\in\mathbb{N}$ and $t\ge 0$, we have from \eqref{ISDE-Omega},
     \begin{align*}
\sum_{i=1}^{M}\tilde{\mathsf{x}}_i(t)=&\sum_{i=1}^{M}x_i+\int_0^t\sum_{i=1}^{M}\tilde{\mathsf{x}}_i(s) \mathrm{d}\tilde{\mathsf{w}}_i(s)-\frac{\eta}{2}\int_0^t\sum_{i=1}^{M}\tilde{\mathsf{x}}_i(s)\mathrm{d}s+M\int_0^t\left(\tilde{\boldsymbol{\gamma}}(s)-\sum_{i=1}^{\infty}\tilde{\mathsf{x}}_i(s)\right)\mathrm{d}s\nonumber\\
&+\int_0^t\sum_{i=1}^{M}\sum_{j>M}\frac{\tilde{\mathsf{x}}_i(s)\tilde{\mathsf{x}}_j(s)}{\tilde{\mathsf{x}}_i(s)-\tilde{\mathsf{x}}_j(s)}\mathrm{d}s.
	\end{align*}
 Therefore, since the first and last terms in the above display are non-negative, one has
     \begin{align*}
M\int_0^t\left(\tilde{\boldsymbol{\gamma}}(s)-\sum_{i=1}^{\infty}\tilde{\mathsf{x}}_i(s)\right)\mathrm{d}s\leq
\sum_{i=1}^{M}\tilde{\mathsf{x}}_i(t)-\int_0^t\sum_{i=1}^{M}\tilde{\mathsf{x}}_i(s) \mathrm{d}\tilde{\mathsf{w}}_i(s)+\frac{\eta}{2}\int_0^t\sum_{i=1}^{M}\tilde{\mathsf{x}}_i(s)\mathrm{d}s.
	\end{align*}
  Note that, all the terms on the RHS are finite uniformly as $M \to \infty$. Thus, dividing both sides by $M$, and then, taking $M$ to infinity, one obtains that $\tilde{\mathbf{P}}$-a.s., for all $t\ge 0$,
 \begin{align*}
\int_0^t\left(\tilde{\boldsymbol{\gamma}}(s)-\sum_{i=1}^{\infty}\tilde{\mathsf{x}}_i(s)\right)\mathrm{d}s\leq 0.
	\end{align*}
 But clearly, $\int_0^t\left(\tilde{\boldsymbol{\gamma}}(s)-\sum_{i=1}^{\infty}\tilde{\mathsf{x}}_i(s)\right)\mathrm{d}s\geq 0$ and thus we get $\tilde{\mathbf{P}}$-a.s., $\int_0^t\left(\tilde{\boldsymbol{\gamma}}(s)-\sum_{i=1}^{\infty}\tilde{\mathsf{x}}_i(s)\right)\mathrm{d}s= 0$, as desired.

 We now want to prove the second statement in the theorem. First observe that, the above implies that $\tilde{\mathbf{P}}$-a.s., for all $t\ge 0$, \begin{align}\label{tildegamma=sum}
     \tilde{\boldsymbol{\gamma}}(s)=\sum_{i=1}^{\infty}\tilde{\mathsf{x}}_i(s), \ 
 \ \textnormal{for almost all } s\in[0,t]. 
 \end{align}
Furthermore, we have, $\tilde{\mathbf{P}}$-a.s., for all $M\in\mathbb{N}$ and $t\ge 0$,
  \begin{align}\label{FinalIntermediateEquation}
\sum_{i=1}^{M}\tilde{\mathsf{x}}_i(t)=&\sum_{i=1}^{M}x_i+\int_0^t\sum_{i=1}^{M}\tilde{\mathsf{x}}_i(s) \mathrm{d}\tilde{\mathsf{w}}_i(s)-\frac{\eta}{2}\int_0^t\sum_{i=1}^{M}\tilde{\mathsf{x}}_i(s)\mathrm{d}s+\int_0^t\sum_{i=1}^{M}\sum_{j>M}\frac{\tilde{\mathsf{x}}_i(s)\tilde{\mathsf{x}}_j(s)}{\tilde{\mathsf{x}}_i(s)-\tilde{\mathsf{x}}_j(s)}\mathrm{d}s.
\end{align}
Now, by making use of the Burkholder-Davis-Gundy inequality for local martingales, see \cite{KaratzasShreve,Revuz-Yor}, or directly using Problem 1.5.25 of \cite{KaratzasShreve}, by virtue of the fact that $\sum_{i=K}^\infty \tilde{\mathsf{x}}_i(\cdot)^2 \to 0$, uniformly on compact sets, as $K \to \infty$, we obtain the existence of a sequence $(M_{\ell})_{\ell=1}^\infty$ with $M_{\ell} \to \infty$, as $\ell \to \infty$, such that $\tilde{\mathbf{P}}$-a.s. the following limit exists uniformly on compact sets,
\begin{equation*}
\mathsf{U}(t)\overset{\textnormal{def}}{=}\lim_{\ell \to \infty} \int_0^t \sum_{i=1}^{M_\ell} \tilde{\mathsf{x}}_i(s)\mathrm{d}\tilde{\mathsf{w}}_i(s), \ \ \forall t \ge 0.
\end{equation*}
In particular, $\tilde{\mathbf{P}}$-a.s. the function $t \mapsto \mathsf{U}(t)$ is continuous. Then, putting $M=M_\ell$ and taking the limit $\ell \to\infty$ in both sides of \eqref{FinalIntermediateEquation}, we obtain that $\tilde{\mathbf{P}}$-a.s., for all $t\ge 0$,
 \begin{align}\label{sum}
\sum_{i=1}^{\infty}\tilde{\mathsf{x}}_i(t)=&\sum_{i=1}^{\infty}x_i+\mathsf{U}(t)-\frac{\eta}{2}\int_0^t\sum_{i=1}^{\infty}\tilde{\mathsf{x}}_i(s)\mathrm{d}s+\lim_{\ell\to\infty}\int_0^t\sum_{i=1}^{M_\ell}\sum_{j>M_\ell}\frac{\tilde{\mathsf{x}}_i(s)\tilde{\mathsf{x}}_j(s)}{\tilde{\mathsf{x}}_i(s)-\tilde{\mathsf{x}}_j(s)}\mathrm{d}s.
\end{align}
Observe now that the second and third terms on the RHS of \eqref{sum} are continuous functions of $t$, and thus, any possible discontinuity of $t \mapsto \sum_{i=1}^{\infty}\tilde{\mathsf{x}}_i(t)$ arises from the last term. Note also that since the integrand is strictly positive, one has 
 \begin{align*}
\lim_{\ell\to\infty}\int_0^\mathsf{u}\sum_{i=1}^{M_\ell}\sum_{j>M_\ell}\frac{\tilde{\mathsf{x}}_i(s)\tilde{\mathsf{x}}_j(s)}{\tilde{\mathsf{x}}_i(s)-\tilde{\mathsf{x}}_j(s)}\mathrm{d}s>\lim_{t\nearrow \mathsf{u}}\lim_{\ell\to\infty}\int_0^t\sum_{i=1}^{M_\ell}\sum_{j>M_\ell}\frac{\tilde{\mathsf{x}}_i(s)\tilde{\mathsf{x}}_j(s)}{\tilde{\mathsf{x}}_i(s)-\tilde{\mathsf{x}}_j(s)}\mathrm{d}s,
	\end{align*}
 at any point of discontinuity $\mathsf{u}>0$ of $t \mapsto \sum_{i=1}^{\infty}\tilde{\mathsf{x}}_i(t)$, and hence, from \eqref{sum} we get 
 \begin{align*}
   \sum_{i=1}^{\infty}\tilde{\mathsf{x}}_{i}(\mathsf{u})> \lim_{t_k\nearrow \mathsf{u}}\sum_{i=1}^{\infty}\tilde{\mathsf{x}}_i(t_k)=\lim_{t_k\nearrow \mathsf{u}} \tilde{\boldsymbol{\gamma}}(t)=\tilde{\boldsymbol{\gamma}}(\mathsf{u}),
\end{align*}
where, by virtue of \eqref{tildegamma=sum}, the sequence $\left(t_k\right)_{k=1}^\infty$ is picked so that $\tilde{\boldsymbol{\gamma}}(t_k)=\sum_{i=1}^{\infty}\tilde{\mathsf{x}}_i(t_k)$ holds for all $k\in\mathbb{N}$ and recall that $t\mapsto \tilde{\boldsymbol{\gamma}}(t)$ is continuous.
But, this is a contradiction as we know that
$\sum_{i=1}^{\infty}\tilde{\mathsf{x}}_i(t)\leq\tilde{\boldsymbol{\gamma}}(t)$, for all $t\geq 0$. Therefore, it follows that
 $t \mapsto \sum_{i=1}^{\infty}\tilde{\mathsf{x}}_i(t)$ is, in fact, a continuous function on $(0,\infty)$ and in particular, we have that $\tilde{\mathbf{P}}$-a.s.,
 \begin{equation*}
\sum_{i=1}^{\infty}\tilde{\mathsf{x}}_i(t)=\tilde{\boldsymbol{\gamma}}(t), \ \  \forall t>0,
 \end{equation*} which is equivalent to \eqref{Claim-gamma=sum}. This completes the proof.
 \end{proof}

 \begin{rmk}
 It is interesting to note that Theorem \ref{thm-ISDE} combined with Theorem \ref{MainThmConvEq} give a non-computational proof of the fact that if, instead of Definition \ref{DefInvariantMeasure}, we define $\mathfrak{M}^\eta$ as the unique probability measure on $\Omega_+$ satisfying $\mathfrak{M}^\eta\Lambda^\infty_N=\mathfrak{M}^\eta_N$ for all $N\in\mathbb{N}$, then $\mathfrak{M}^\eta$ must be supported on $\Omega_+^0$. The original proof in \cite{InvWishart} uses the underlying determinantal point process structure and some hard estimates on the correlation kernel of the Laguerre unitary ensemble. The dynamical proof we just mentioned avoids these considerations entirely.
 \end{rmk}

 \begin{rmk}\label{AlternativeConvergenceArgument}
 We note that it is possible to  obtain the convergence of the integrands in \eqref{convInteractionTerm} in a slightly different (but essentially equivalent) way to the one presented in the proof of Theorem \ref{thm-ISDE} without using characteristic polynomials. This is somewhat less intuitive but a little quicker. First, observe that, for all $N\in\mathbb{N}$,  $i\leq N$, $j\neq i$,
\begin{align}\label{drift_j-x_j}\frac{\tilde{\mathsf{x}}_i^{(N)}(s)\tilde{\mathsf{x}}_j^{(N)}(s)}{\tilde{\mathsf{x}}_i^{(N)}(s)-\tilde{\mathsf{x}}_j^{(N)}(s)}-\tilde{\mathsf{x}}_j^{(N)}(s)=
     \frac{\tilde{\mathsf{x}}_j^{(N)}(s)^2}{\tilde{\mathsf{x}}_i^{(N)}(s)-\tilde{\mathsf{x}}_j^{(N)}(s)} \le\frac{\tilde{\mathsf{x}}_j^{(N)}(s)^2}{\tilde{\mathsf{x}}_i^{(N)}(s)-\tilde{\mathsf{x}}_{i+1}^{(N)}(s)} , \  \ \forall s \ge 0.
\end{align}
 Next, using convergence on $\Omega_+$, we have that, $\tilde{\mathbf{P}}$-a.s.,
\begin{equation*}
\left(\tilde{\mathsf{x}}_i^{(N)}(s)-\tilde{\mathsf{x}}_{i+1}^{(N)}(s)\right)^{-1}\sum_{j=i+1}^{N}\tilde{\mathsf{x}}_j^{(N)}(s)^2\overset{N\to\infty}{\longrightarrow}
\left(\tilde{\mathsf{x}}_i(s)-\tilde{\mathsf{x}}_{i+1}(s)\right)^{-1}\sum_{j=i+1}^{\infty}\tilde{\mathsf{x}}_j(s)^2, \  \ \forall s\ge 0.
\end{equation*}
From this and \eqref{drift_j-x_j}, making use of the generalised dominated convergence theorem for the series, we conclude that (again using convergence on $\Omega_+$),
\begin{align}\label{limDrift}
    \lim_{N\to\infty}\sum_{j=1,j\neq i}^{N}
    \frac{\tilde{\mathsf{x}}_i^{(N)}(s)\tilde{\mathsf{x}}_j^{(N)}(s)}{\tilde{\mathsf{x}}_i^{(N)}(s)-\tilde{\mathsf{x}}_j^{(N)}(s)}= \sum_{j=1,j\neq i}^{\infty}\frac{\tilde{\mathsf{x}}_j(s)^2}{\tilde{\mathsf{x}}_i(s)-\tilde{\mathsf{x}}_j(s)}+\tilde{\boldsymbol{\gamma}}(s)-\tilde{\mathsf{x}}_i(s), \  \ \forall s \ge 0.
\end{align}
Finally, observe that we have
\begin{align*}\label{Drift-Sum}
\sum_{j=1,j\neq i}^{\infty}\frac{\tilde{\mathsf{x}}_j(s)^2}{\tilde{\mathsf{x}}_i(s)-\tilde{\mathsf{x}}_j(s)}=\sum_{j=1,j\neq i}^{\infty}\frac{\tilde{\mathsf{x}}_i(s)\tilde{\mathsf{x}}_j(s)}{\tilde{\mathsf{x}}_i(s)-\tilde{\mathsf{x}}_j(s)} -  \sum_{j=1,j\neq i}^{\infty}\tilde{\mathsf{x}}_j(s), \  \ \forall s\ge 0,
\end{align*}
which gives what we desired to prove. In order to show convergence of the actual integrals and that $\tilde{\boldsymbol{\gamma}}(s)=\sum_{i=1}^\infty \tilde{\mathsf{x}}_i(s)$, for $s>0$, we use the same argument as in the proof of Theorem \ref{thm-ISDE}.
 \end{rmk}

\begin{rmk} Arguing as in the proof of Theorem \ref{thm-ISDE}, we can show that, under the $\tilde{\mathbf{P}}$-coupling from \eqref{convOnOmega}, for any initial condition $\omega=(\mathbf{x},\gamma) \in \Omega_+$ (we do not need to use any non-intersection property of the paths),
    \begin{equation}\label{gammaSDE}
\mathrm{d}\tilde{\boldsymbol{\gamma}}(t)=\mathrm{d}\mathcal{N}(t)+\left(\frac{1}{2}-\frac{\eta}{2}\tilde{\boldsymbol{\gamma}}(t)\right)\mathrm{d}t,
\end{equation}
where $\mathcal{N}(\cdot)$ is given by the uniform limit (look at the corresponding equation for finite $N$ and solve for the martingale term),
\begin{equation*}
\mathcal{N}(t)\overset{\textnormal{def}}{=} \lim_{N \to \infty} \int_0^t \sum_{i=1}^N \tilde{\mathsf{x}}^{(N)}_i(s)\mathrm{d}\tilde{\mathsf{w}}_i^{(N)}(s), \ \forall t \ge 0,
\end{equation*}
and is in particular a one-dimensional martingale with respect to $(\mathscr{F}_t)_{t\ge 0}$ with quadratic variation, by virtue of the fact that $ \sum_{i=1}^N\tilde{\mathsf{x}}^{(N)}_i(\cdot)^2 \to \sum_{i=1}^{\infty}\tilde{\mathsf{x}}_i(\cdot)^2$ uniformly on compact sets,
\begin{equation*}
\langle \mathcal{N},\mathcal{N}\rangle_t=\int_{0}^t \sum_{i=1}^{\infty}\tilde{\mathsf{x}}_i(s)^2\mathrm{d}s, \ \forall t \ge 0.
\end{equation*}
Hence, there exists a one-dimensional Brownian motion $\left(\hat{\mathsf{w}}(t);t\ge 0\right)$ so that the stochastic equation \eqref{gammaSDE} can in fact be written as
    \begin{equation*}
\mathrm{d}\tilde{\boldsymbol{\gamma}}(t)=\sqrt{\sum_{i=1}^{\infty}\tilde{\mathsf{x}}_i(s)^2}\mathrm{d}\hat{\mathsf{w}}(t)+\left(\frac{1}{2}-\frac{\eta}{2}\tilde{\boldsymbol{\gamma}}(t)\right)\mathrm{d}t.
\end{equation*}
Finally, we note that the Brownian motion $\left(\hat{\mathsf{w}}(t);t\ge 0\right)$ is dependent on the sequence of independent Brownian motions $(\tilde{\mathsf{w}}_i)_{i=1}^\infty$ driving the $(\tilde{\mathsf{x}}_i)_{i=1}^\infty$.
\end{rmk}

We now prove that solutions corresponding to different parameters $\gamma$ are different in law.
 
 \begin{thm}\label{thm-NonUniqueness} 
 In the setting of Theorem \ref{thm-ISDE}, we denote by $\left(\mathbf{X}(t;\gamma);t\ge 0\right)$ the solution to \eqref{ISDEintro} corresponding to $\gamma$, then,
\begin{equation*}
\mathsf{Law} \big(\mathbf{X}(\cdot;\gamma)\big)\neq \mathsf{Law}\big(\mathbf{X}(\cdot;\tilde{\gamma})\big), \textnormal{whenever } \gamma \neq \tilde{\gamma}.
\end{equation*}
\end{thm}
  \begin{proof}
  Based on \eqref{Claim-gamma=sum} we have that a.s.
\begin{align}\label{gamma=sum}
    \sum_{i=1}^{\infty}\mathsf{x}_i(t;\gamma)
    =\boldsymbol{\gamma}(t),\  \ 
\sum_{i=1}^{\infty}\mathsf{x}_i(t;\tilde{\gamma})
=\tilde{\boldsymbol{\gamma}}(t),\  \ \forall t>0.
\end{align}
Note now that $\boldsymbol{\gamma}$ and $\tilde{\boldsymbol{\gamma}}$ do not have same distribution, as they are continuous functions with $\boldsymbol{\gamma}(0)\neq\tilde{\boldsymbol{\gamma}}(0)$, and thus, by virtue of \eqref{gamma=sum},
\begin{equation*}
\mathsf{Law} \left(\sum_{i=1}^{\infty}\mathsf{x}_i(\cdot;\gamma)\right)\neq
\mathsf{Law} \left(\sum_{i=1}^{\infty}\mathsf{x}_i(\cdot;\tilde{\gamma})\right),
\end{equation*}
 from which the conclusion follows immediately. 
  \end{proof}
 
 We have the following result which completes the statement of Theorem \ref{MainThm2Intro}.

 \begin{prop}\label{prop-MarkovProperty}
Out of all solutions to the ISDE \eqref{ISDEintro} constructed as in Theorem \ref{thm-ISDE}, there exists a unique one such that almost surely $t \mapsto \sum_{i=1}^\infty \mathsf{x}_i(t;\gamma)$ is continuous for all $t \ge 0$ given by the choice $\gamma=\sum_{i=1}^\infty x_i$.
Moreover, this solution is a Markov process. 
 \end{prop}
\begin{proof}
 We know from Theorem \ref{MainThm1Intro} that $\mathbf{X}^{\Omega_+}_{\cdot}=\left(\mathbf{X}(\cdot;\gamma),\boldsymbol{\gamma}(\cdot)\right)$ has continuous sample paths. 
 It then follows from continuity of $t\mapsto\boldsymbol{\gamma}(t)$ and \eqref{Claim-gamma=sum} that 
almost surely $t \mapsto \sum_{i=1}^\infty \mathsf{x}_i(t;\gamma)$ is continuous for all $t \ge 0$ if and only if $\gamma=\sum_{i=1}^\infty x_i$.

Towards the second assertion, note that by the Markov property of $\mathbf{X}^{\Omega_+}$, for all 
$t\geq s \geq 0$ and all bounded, measurable functions $\mathsf{F}:\Omega_+\to\mathbb{R}$, we have
 \begin{align}\label{MarkovPropOmega}
\mathbb{E}\left[\mathsf{F}\left(\mathbf{X}(t;\gamma),\boldsymbol{\gamma}(t)\right)\big| \mathscr{F}_s\right]=\mathbb{E}\left[\mathsf{F}\left(\mathbf{X}(t;\gamma),\boldsymbol{\gamma}(t)\right)\big|\sigma\left(\mathbf{X}(s;\gamma),\boldsymbol{\gamma}(s)\right)\right],
 \end{align}
 where, as before,
$\left(\mathscr{F}_t\right)_{t\geq 0}$
 is the natural filtration of the process $\left(\mathbf{X}^{\Omega_+}_t; t\geq 0\right)$ and $\mathbb{E}\left[\cdot|\sigma(\mathcal{C})\right]$ denotes conditional expectation with respect to $\sigma(\mathcal{C})$.

Now, consider the distinguished solution $\mathbf{X}(\cdot;\gamma_*)$  of the ISDE with $\gamma_*=\sum_{i=1}^\infty x_i$.
Note that, since $\sum_{i=1}^\infty \mathsf{x}_i(s;\gamma_*)=\boldsymbol{\gamma}(s)$, for all $s\ge 0$, we have
 \begin{align}\label{sigma-algebra}
\sigma\left(\mathbf{X}\left(s;\gamma_*\right),\boldsymbol{\gamma}(s)\right)=\sigma\left(\mathbf{X}(s;\gamma_*)\right),\  \ \forall s\geq 0,
 \end{align}
 and thus, in this case the natural filtration $\left(\mathscr{F}_t^{\mathbf{X}(\cdot;\gamma_*)}\right)_{t\ge 0}$ generated by $\left(\mathbf{X}(t;\gamma_*);t\geq 0\right)
 $ is equal to $\left(\mathscr{F}_t\right)_{t\geq 0}$. Denote by $\mathcal{W}^\infty$ the state space of the solution:
 \begin{equation*}
     \mathcal{W}^\infty=\left\{\mathbf{x}\in \mathbb{W}_{\infty,+}^\circ:\sum_{i=1}^\infty x_i<\infty\right\}.
 \end{equation*}
 Now, let $\mathsf{f}:\mathcal{W}^\infty\to\mathbb{R}$ be an arbitrary bounded measurable function and take $\mathsf{F}$ to be its obvious lift on $\Omega_+$. Thus summarising, by virtue \eqref{MarkovPropOmega} and \eqref{sigma-algebra}, we have
  \begin{align*}
     \mathbb{E}\left[\mathsf{f}\left(\mathbf{X}(t;\gamma_*)\right)\big|\mathscr{F}_s^{\mathbf{X}(\cdot;\gamma_*)}\right]=\mathbb{E}\left[\mathsf{f}\left(\mathbf{X}(t;\gamma_*)\right)\big|\sigma\left(\mathbf{X}(s;\gamma_*)\right)\right],\  \ \forall s\geq 0,
 \end{align*}
 which implies that  $\left(\mathbf{X}\left(t;\gamma_*\right)\right)_{t\geq 0}$ is a Markov process and concludes the proof. 
 \end{proof}

\begin{proof}[Proof of Theorem \ref{MainThm2Intro}]
    This follows from Theorem \ref{thm-NonCollision,>0}, Theorem \ref{thm-ISDE}, Theorem \ref{thm-NonUniqueness}, and Proposition \ref{prop-MarkovProperty}.
\end{proof}

\begin{proof}[Proof of Theorem \ref{MainThmEquilibriumProcess}]
    By Theorem \ref{thm-Boundary+}, the probability measure $\mathfrak{m}$ corresponds to the unique sequence of probability measures $\left(\mathfrak{m}_N\right)_{N=1}^{\infty}$ on $\left(\mathbb{W}_{N,+}\right)_{N=1}^{\infty}$ given by,
    \begin{equation*}
    \mathfrak{m}_N=\mathfrak{m}\Lambda_N^\infty, \ \ \forall N\in\mathbb{N}.
    \end{equation*}
    Let $\mathsf{x}^{(N)}$ be the solution of \eqref{rescaledSDE} started at $\mathsf{x}^{(N)}(0)\overset{\textnormal{d}}{=}\mathfrak{m}_N$. By combining Theorem \ref{thm-Boundary+} and Lemma \ref{LemmaDensity}, a standard argument gives that, under the embedding \eqref{embedding+} of $\mathbb{W}_{N,+}$ into $\Omega_+$, $\mathfrak{m}_N$ converges weakly to $\mathfrak{m}$. Namely, $\mathbf{X}^{(N)}_0\overset{\textnormal{d}}{\longrightarrow}\mathbf{X}^{\Omega_+}_0$, and hence,
 $\mathbf{X}^{(N)}\overset{\textnormal{d}}{\longrightarrow}\mathbf{X}^{\Omega_+}$ by Theorem \ref{thm-ConvMarkovProc+}. One can then follow the argument of Theorem \ref{thm-ISDE} to prove that $\left(\mathbf{X}(t;\mathfrak{m});t\ge 0\right)$ is a weak solution of the ISDE \eqref{ISDEintro} with initial condition $\mathbf{X}(0;\mathfrak{m})\overset{\textnormal{d}}{=}\mathsf{IBes}_{\eta}$. Note that, as $\mathfrak{M}^\eta$ is the invariant measure of $\mathbf{X}^{\Omega_+}$, we have 
 \begin{equation*}
\mathsf{Law}\big(\mathbf{X}(t;\mathfrak{M}^\eta)\big)=\mathsf{IBes}_\eta, \ \ \forall t \ge 0,
\end{equation*}
 On the other hand, if $\mathfrak{m}\neq\mathfrak{M}^\eta$, we must have
 \begin{align*}
\mathbb{P}_{\mathfrak{m}}\left(\boldsymbol{\gamma}(0;\mathfrak{m})>\sum_{i=1}^\infty \mathsf{x}_i(0;\mathfrak{m})\right)>0,
 \end{align*}
 where $\mathbb{P}_{\mathfrak{m}}$ denotes $\mathsf{Law}\left(\mathbf{X}^{\Omega_+}\right)$ if $\mathsf{Law}\left(\mathbf{X}_0^{\Omega_+}\right)=\mathfrak{m}$. Thus, based on the arguments in Theorem \ref{thm-NonUniqueness}, since $\boldsymbol{\gamma}(t;\mathfrak{m})=\sum_{i=1}^\infty \mathsf{x}_i(t;\mathfrak{m})$ for all $t>0$, and recall that $t\mapsto \boldsymbol{\gamma}(t;\mathfrak{m})$ is continuous, we obtain that the stochastic processes $t \mapsto \sum_{i=1}^\infty \mathsf{x}_i(t;\mathfrak{m})$ and $t\mapsto \sum_{i=1}^\infty \mathsf{x}_i(t;\mathfrak{M}^\eta)$ are not equal in law, for any $\mathfrak{m} \neq \mathfrak{M}^\eta$, and so, 
\begin{equation*}
\mathsf{Law} \big(\mathbf{X}\left(\cdot;\mathfrak{M}^\eta\right)\big)\neq \mathsf{Law}\big(\mathbf{X}\left(\cdot;\mathfrak{m}\right)\big).
\end{equation*} 
That is, the following holds if and only if $\mathfrak{m}=\mathfrak{M}^\eta$, 
 \begin{equation*}
\mathsf{Law}\big(\mathbf{X}\left(t;\mathfrak{m}\right)\big)=\mathsf{IBes}_\eta, \ \ \forall t \ge 0,
\end{equation*}
and this completes the proof.
\end{proof}

\section{Dynamical Cauchy model}\label{SectionCauchy}

In this section, we consider stochastic dynamics related to the Cauchy \cite{ForresterWitte,ForresterBook} or so-called Hua-Pickrell ensemble \cite{BorodinOlshanski}. This is the following probability measure on $\mathbb{W}_N$,
depending on a parameter $s\in\mathbb{C}$, with $\Re(s)>-\frac{1}{2}$,
\begin{align*}
    \mathfrak{N}_{N}^{(s)}(\mathrm{d}\mathbf{x})=\frac{1}{\mathcal{Z}_{N}^{(s)}}
    \Delta_N^2\left(\mathbf{x}\right)
    \prod_{j=1}^{N}\left(1 + x_j^2\right)^{-\Re({s})-N}\mathrm{e}^{2\Im({s})\tan^{-1}(x_j)}\mathrm{d}\mathbf{x},
\end{align*} 
where,
$\mathcal{Z}_{N}^{(s)}$ is a normalizing constant.

The associated $N$-particle dynamics, that we call the dynamical Cauchy model, introduced in \cite{H-P} (they were Hua-Pickrell diffusions therein),  is given by the unique solution to the following system of stochastic equations
 \begin{align}\label{HPFSDE}
		\mathrm{d}\mathfrak{x}_i(t) = \sqrt{2\left[\mathfrak{x}_i(t)^2+1\right]} \mathrm{d}\mathsf{w}_i(t) + 2\left[\left(1-N-\Re({s})\right)\mathfrak{x}_i(t)+\Im({s})\right]\mathrm{d}t+2\sum_{j=1:j\neq i}^N\frac{\mathfrak{x}_i(t)^2+1}{\mathfrak{x}_i(t)-\mathfrak{x}_j(t)}\mathrm{d}t,
\end{align} 
for $i=1,\dots, N$, where as usual the $\mathsf{w}_i$ are independent standard Brownian motions.
It is proven in \cite{H-P}, making use of the results of \cite{Graczyk-Malecki},  that
\eqref{HPFSDE} has a unique strong solution for any initial condition $\mathbf{x}\in \mathbb{W}_{N}$ and almost surely for all $t>0$, $\mathsf{x}(t)\in \mathbb{W}_N^\circ$ and moreover the associated semigroup $\left(\mathfrak{Q}^{(s)}_N(t)\right)_{t\ge 0}$ is Feller. The following result is the main theorem of \cite{H-P}.

\begin{thm}\label{thm-HPabstractThm}
Let $s\in \mathbb{C}$. Then, the semigroups $\left(\mathfrak{Q}_N^{(s)}(t)\right)_{t\ge 0}$ are intertwined:
\begin{equation*}
    \mathfrak{Q}_{N+1}^{(s)}(t)\Lambda_{N}^{N+1}=
\Lambda_{N}^{N+1}\mathfrak{Q}_{N}^{(s)}(t),\  \ \forall t\geq 0,\    N\in\mathbb{N}.
\end{equation*}
In particular, there exists a unique Feller semigroup $\left(\mathfrak{Q}_\infty^{(s)}(t)\right)_{t\ge 0}$, with associated stochastic process $\left(\mathbf{X}_t^{\Omega};t\ge 0\right)$ on $\Omega$, defined via the intertwining
\begin{align*}
    \mathfrak{Q}_{\infty}^{(s)}(t)\Lambda_{N}^{\infty}=
\Lambda_{N}^{\infty}\mathfrak{Q}_{N}^{(s)}(t),\  \ \forall t\geq 0,\    N\in\mathbb{N}.
\end{align*}   
Moreover, if $\Re(s)>-\frac{1}{2}$, then $\left(\mathfrak{Q}_\infty^{(s)}(t)\right)_{t\ge 0}$ has a unique invariant measure $\mathfrak{N}_{\infty}^{(s)}$ uniquely determined via the relation
\begin{equation*}
  \mathfrak{N}_{\infty}^{(s)}\Lambda_N^\infty = \mathfrak{N}_{N}^{(s)}, \ \ \forall N \in \mathbb{N}.
\end{equation*}
\end{thm}

Before giving our main result on the dynamical Cauchy model it is worth pointing out, that for real $s>-\frac{1}{2}$, the invariant measure associated to the dynamics $\left(\mathfrak{Q}^s_\infty(t)\right)_{t\ge 0}$ is known to admit an explicit description. We briefly spell this out.

\begin{defn}
Let $s>-\frac{1}{2}$. The Hua-Pickrell point process $\mathsf{HP}^{\textnormal{Conf}}_s$ with parameter $s$ is the determinantal point process on $(-\infty,0)\times(0,\infty)$ with correlation kernel $\mathfrak{K}_{s}$ given by 
    \begin{align*}
\mathfrak{K}_{s}(z,w) = \frac{1}{2\pi}\frac{\Gamma(s+1)^2}{\Gamma(2s+1)\Gamma(2s+2)}\frac{\mathcal{Y}_1^{(s)}(z)\mathcal{Y}_2^{(s)}(w)-\mathcal{Y}_1^{(s)}(w)\mathcal{Y}_2^{(s)}(z)}{z-w},
\end{align*}
where the functions $\mathcal{Y}_1^{(s)}(z), \mathcal{Y}_2^{(s)}(z)$ are given by 
\begin{equation*}
\mathcal{Y}_1^{(s)}(z) = 2^{2s-\frac{1}{2}}\Gamma\left(s+\frac{1}{2}\right) \cdot \frac{1}{|z|^{\frac{1}{2}}}J_{s-1/2}\left(\frac{1}{|z|}\right), \ \ \mathcal{Y}_2^{(s)}(z) = 2^{2s+\frac{1}{2}}\Gamma\left(s+\frac{3}{2}\right) \cdot \frac{1}{|z|^{\frac{1}{2}}} J_{s+1/2}\left(\frac{1}{|z|}\right),
\end{equation*}
where as before $J_{r}$ denotes the Bessel function with parameter $r$.
\end{defn}

For $s=0$ and under the transformation $x\mapsto x^{-1}$, $\mathsf{HP}^{\textnormal{Conf}}_0$ becomes the sine determinantal point process, the universal limit of random matrices in the bulk of the spectrum \cite{ForresterBook}. Since $\mathsf{HP}^{\textnormal{Conf}}_s$ consists of distinct points it gives rise to a unique probability measure  $\mathsf{HP}_s$ on $\mathbb{W}_{\infty,+}^\circ \times \mathbb{W}_{\infty,+}$ by labelling its positive and negative points respectively according to their absolute values. Define the space
\begin{align*}
\Omega^0\overset{\textnormal{def}}{=}\left\{\omega=\left(\mathbf{x}^+,\mathbf{x}^-,\gamma,\delta\right)\in\Omega:\gamma=\lim_{k\to \infty } \sum_{i=1}^\infty \left[x_i^+\mathbf{1}_{x_i^+\ge k^{-2}}- x_i^-\mathbf{1}_{x_i^-\ge k^{-2}}\right] , \delta=\sum_{i=1}^{\infty}\left[\left(x^+_i\right)^2+\left(x^-_i\right)^2\right]\right\},
\end{align*}
and write $\pi:\Omega \to \mathbb{W}_{\infty,+} \times \mathbb{W}_{\infty,+}$ for the map $\pi((\mathbf{x}^+,\mathbf{x}^-,\gamma,\delta))=(\mathbf{x}^+,\mathbf{x}^-)$. Then, it turns out that $\mathfrak{N}_\infty^{(s)}$ is the unique probability measure on $\Omega$ which is supported on $\Omega^0$ and moreover satisfies (see the discussion in \cite{H-P} and the references therein),
\begin{align*}
\pi_*\mathfrak{N}^{(s)}_\infty=\mathsf{HP}_{s}.
\end{align*}

We now upgrade the abstract construction of $\left(\mathbf{X}^\Omega_t;t\ge 0\right)$ from Theorem \ref{thm-HPabstractThm} to a strong approximation theorem from the finite dimensional dynamics which moreover implies that $\left(\mathbf{X}^\Omega_t;t\ge 0\right)$ has continuous sample paths. Finally, we prove convergence to equilibrium from any initial condition. We leave the very interesting problem of constructing an ISDE, the drift of which will most likely require renormalisation to even make sense, to future work. We define $\mathbf{X}^{(N)}=\left(\left(\mathsf{x}^{(N),+}_i\right)_{i=1}^\infty,\left(\mathsf{x}_i^{(N),-}\right)_{i=1}^\infty,\boldsymbol{\gamma}^{(N)},\boldsymbol{\delta}^{(N)}\right)$ to be the process on $\Omega$ obtained from (\ref{HPFSDE}) under the embedding (\ref{embedding}) of $\mathbb{W}_N$ into $\Omega$.

\begin{thm}\label{thm-HuaPickrell}
Let $s\in \mathbb{C}$. Consider the Feller process $\left(\mathbf{X}_t^{\Omega};t\ge 0\right)$ with semigroup $\left(\mathfrak{Q}_\infty^{(s)}(t)\right)_{t\ge 0}$ constructed in Theorem \ref{thm-HPabstractThm} and the embedded processes $\mathbf{X}^{(N)}$ on $\Omega$ defined above. Let $\mathbf{X}_0^{\Omega}=\omega \in \Omega$ be arbitrary and assume $\mathbf{X}^{(N)}(0)\to \omega$ in the topology of $\Omega$. Then, as $N \to \infty$,
\begin{equation*}
 \mathbf{X}^{(N)} \overset{\textnormal{d}}{\longrightarrow} \mathbf{X}^{\Omega},\  \ \textnormal{in}\ 
 \ C(\mathbb{R}_+,\Omega).
\end{equation*}
Moreover, there exists a coupling of the $\mathbf{X}^{(N)}$ and $\mathbf{X}^{\Omega}$ such that, almost surely, for all $T\ge0$, we have:
    \begin{align*}
\sup_{t\in[0,T]}\left(\sum_{i=1}^{\infty}\left|\mathsf{x}_i^{(N),+}(t)-\mathsf{x}_i^+(t)\right|^3+\sum_{i=1}^{\infty}\left|\mathsf{x}_i^{(N),-}(t)-\mathsf{x}_i^-(t)\right|^3\right)\overset{N\to\infty}{\longrightarrow} 0.
    \end{align*}
   Finally, for $\Re{(s)}>-\frac{1}{2}$ and $\mathfrak{K}$ an arbitrary probability measure on $\Omega$, if $\mathsf{Law}\left(\mathbf{X}_0^{\Omega} \right)=\mathfrak{K}$, then as $t\to\infty$,
\begin{equation*}
\mathbf{X}_{t}^{\Omega} \overset{\textnormal{d}}{\longrightarrow} \mathbf{Z}, \ \ \textnormal{where } \mathsf{Law}(\mathbf{Z})=\mathfrak{N}_\infty^s. 
\end{equation*}
\end{thm}

\begin{proof}
All the statements in the theorem, except for convergence to equilibrium, follow by combining Theorem \ref{thm-HPabstractThm} and the fact that $C^\infty_{c,\textnormal{sym}}(\mathbb{W}_N)$ is a core, and invariant under the generator of  $(\mathfrak{Q}_N^{(s)}(t))_{t\ge 0}$ (this can be proven in exactly the same way as Proposition \ref{prop-Core+TransitionDens}), by virtue of Theorem \ref{thm-ConvMarkovProc} and Proposition \ref{Prop-Conv in l^3}. The only thing that remains to note is convergence to equilibrium of the finite-dimensional processes which itself follows from that of the corresponding matrix process shown in \cite{Bougerol}. Then, the remaining statement in the theorem follows from Theorem \ref{thm-ergodicity}.
\end{proof}

\bibliographystyle{acm}
\bibliography{References}

\bigskip 

\noindent{\sc School of Mathematics, University of Edinburgh, James Clerk Maxwell Building, Peter Guthrie Tait Rd, Edinburgh EH9 3FD, U.K.}\newline
\href{mailto:theo.assiotis@ed.ac.uk}{\small theo.assiotis@ed.ac.uk}

\bigskip

\noindent{\sc School of Mathematics, University of Edinburgh, James Clerk Maxwell Building, Peter Guthrie Tait Rd, Edinburgh EH9 3FD, U.K.}\newline
\href{mailto:Z.S.Mirsajjadi@sms.ed.ac.uk}{\small Z.S.Mirsajjadi@sms.ed.ac.uk}
\end{document}